\documentclass[11pt]{amsart}   	% use "amsart" instead of "article" for AMSLaTeX format
\allowdisplaybreaks
\usepackage{hyperref}
\usepackage{geometry,comment}   
\usepackage{longtable}
\usepackage{graphicx}				% Use pdf, png, jpg, or eps§ with pdflatex; use eps in DVI mode
\usepackage{amssymb}

\newtheorem{theorem}{\bf Theorem}[section]
\newtheorem{proposition}[theorem]{\bf Proposition}
\newtheorem{lemma}[theorem]{\bf Lemma}
\newtheorem{corollary}[theorem]{\bf Corollary}

\newtheorem{remark}[theorem]{\bf Remark}

\def\field{K}
\def\sepbeta{\beta_{\mathrm{sep}}}

\title{The separating Noether number of small groups}
\author[M. Domokos]{M\'aty\'as Domokos}
\address{HUN-REN Alfr\'ed R\'enyi Institute of Mathematics,
Re\'altanoda utca 13-15, 1053 Budapest, Hungary,
ORCID iD: https://orcid.org/0000-0002-0189-8831}
\email{domokos.matyas@renyi.hu}
\author[B. Schefler]{Barna Schefler} 
\address{E\"otv\"os Lor\'and University, 
P\'azm\'any P\'eter s\'et\'any 1/C, 1117 Budapest, Hungary} 
\email{scheflerbarna@yahoo.com}
\thanks{Partially supported by the Hungarian National Research, Development and Innovation Office,  NKFIH K 138828.}

\begin{document}
\maketitle

\begin{abstract}
The present paper completes the computation of the separating Noether numbers for the groups with order strictly less than $32$.
Most of the results are proved for the case of a general (possibly finite) base field containing 
an element whose multiplicative order equals the size of the group. 
\end{abstract}

\section{Introduction}
Throughout this paper $G$ is a finite group, and $K$ is a field whose characteristic does not divide $|G|$. 
Given a linear action of $G$ on an $n$-dimensional $K$-vector space $V$ (i.e. $V$ is a $KG$-module), 
denote by $K[V]^G$ the subalgebra of the polynomial ring $K[V]=K[x_1,\dots,x_n]$ consisting of $G$-invariant polynomials. 
Since the action of $G$ on $K[V]$ preserves the standard grading, the algebra $K[V]^G$ is generated by homogeneous elements.  
Write $\beta(G,V)$ for the minimal positive integer $d$ such that $K[V]^G$ is generated by elements of degree at most $d$, and denote by $\beta^K(G)$ the supremum of the numbers $\beta(G,V)$, as $V$ ranges over all finite dimensional $KG$-modules. 
A classical theorem of Noether \cite{noether} asserts that $\beta^\mathbb{C}(G)\le|G|$, and therefore 
$\beta^K(G)$ is called \emph{the Noether number of $G$ (over $K$)}.  
The systematic study of $\beta^K(G)$ was initiated in \cite{schmid}, and continued in \cite{domokos-hegedus}, \cite{sezer:1}, 
\cite{CzD:1}, \cite{cziszter-domokos:indextwo}, \cite{cziszter:C7rtimesC3}. As far as the exact value of $\beta^K(G)$ is concerned, the state of the art is summarized in \cite{cziszter-domokos-szollosi}. 
Note that for an abelian group $G$, the Noether number $\beta^\mathbb{C}(G)$ equals the Davenport constant of $G$, which has an extensive literature, see for example \cite{zhao} for a recent work. 

Here we investigate a similar quantity, called the separating Noether number. 
It is related to one of the main motivations for studying $K[V]^G$, namely to the desire to distinguish the $G$-orbits in $V$. 
Denote by $G\cdot v$ the $G$-orbit of $v\in V$. For $v,w\in V$ with $G\cdot v\neq G\cdot w$, 
there exists an invariant $f\in \field[V]^G$ with $f(v)\neq f(w)$ (see for example \cite[Theorem 3.12.1]{derksen-kemper}, 
\cite[Theorem 16]{kemper} or \cite[Lemma 3.1]{draisma-kemper-wehlau}). 
Therefore we call a subset $S$ of $K[V]^G$ \emph{a separating set of $G$-invariants} (separating set for short) if whenever $G\cdot v\neq G\cdot w$, then there exists an $f\in S$ with $f(v)\neq f(w)$. 

The \emph{separating Noether number} $\sepbeta(G,V)$ is the minimal positive integer $d$ such that 
the elements of degree at most $d$ in $K[V]$ form a separating set, 
and $\sepbeta^K(G)$ is the supremum of $\sepbeta(G,V)$, as $V$ ranges over all finite dimensional $KG$-modules 
(these quantities were defined in \cite{kohls-kraft}, \cite{kemper}). 
Clearly, any generating set is a separating set, hence we have the obvious inequalitites 
\begin{equation}\label{eq:sepbeta<beta}
\sepbeta(G,V)\le \beta(G,V) \text{ and }\sepbeta^K(G)\le \beta^K(G).
\end{equation}
The study of separating sets of polynomial invariants became a popular topic in the past decades, 
see \cite{dufresne}, \cite{dufresne-elmer-kohls}, \cite{kemper},  \cite{draisma-kemper-wehlau}, 
\cite{neusel-sezer}, \cite{kohls-kraft}, 
\cite{lopatin-reimers}, \cite{kemper-lopatin-reimers} for a sample of papers on this subject. A recent interest in separating invariants arose 
in machine learning and signal processing, see for example \cite{bandeira-blumsmith-kileel-niles-weed-perry-wein}, 
\cite{cahill-contreras-hip}, \cite{cahill-iverson-mixon-packer}, \cite{fan-lederman-sun-wang-xu}, 
\cite{blumsmith-garcia-hidalgo-rodriguez}.  

In \cite{domokos-schefler:16} we initiated the following program: determine the separating Noether number for all the 
non-abelian groups whose Noether number is known (i.e. the groups covered by \cite{cziszter-domokos-szollosi}). 
The present paper completes this task. 
The case of abelian groups is covered by \cite{domokos:abelian}, 
\cite{schefler_c_n^r}, \cite{schefler_rank2}, \cite{schefler-zhao-zhong}.  
Interestingly, the range of finite abelian groups whose separating Noether number is known goes far beyond the range of finite abelian  groups with known Noether number (Davenport constant). 
The non-abelian groups of order at most $16$, as well as the groups with a cyclic subgroup of index $2$ and the direct products of dihedral groups with the $2$-element group are dealt with in \cite{domokos-schefler:16}. 
It remains to handle the non-abelian groups $G$ with $16< |G|<32$ which do not fit into the infinite series discussed in \cite{domokos-schefler:16}. This is the content of the present paper. 

\subsection{The main result}\label{sec:main results}

\begin{theorem}\label{thm:sepbeta(<32)}  
Let $G$ be a non-cyclic group with $16<|G|<32$, 
and assume that $\field$ contains an element of multiplicative order $|G|$. 
The value of $\sepbeta^\field(G)$ is given in Table~\ref{table:main}, in some cases under the additional assumption 
that  
$\mathrm{char}(\field)=0$ (this is indicated in the corresponding line of the table).   
\end{theorem} 

\begin{table}
\caption{Separating Noether numbers for $16<|G|<32$}
\begin{center}\label{table:main}
$\begin{array}{c|c|c|c|c}
\textrm{GAP} & G   & \beta^\field(G)  & \sepbeta^\field(G)  &\text{reference for }\sepbeta^\field(G) \\ \hline
(18,1) & \mathrm{Dih}_{18} & 10& 10 & \text{\cite[Theorem 2.1]{domokos-schefler:16}} \\
(18,3) &S_3 \times \mathrm{C}_3  & 8 & 6 & \mathrm{Theorem~\ref{thm:sepbeta(S3xC3)}} \\
(18,4) & (\mathrm{C}_3\times \mathrm{C}_3) \rtimes_{-1} \mathrm{C}_2  & 6 & 6 & \mathrm{Theorem~\ref{thm:sepbeta((C3xC3)rtimesC2)}}  \\
(18,5) &\mathrm{C}_3 \times \mathrm{C}_6  & 8 & 7 & 
\text{\cite[Theorem 1.1]{schefler_rank2}}\\
(20,1) & \mathrm{Dic}_{20}   & 12 & 12 & \text{\cite[Theorem 2.1]{domokos-schefler:16}} \\
(20,3)& \mathrm{C}_5 \rtimes \mathrm{C}_4 \qquad \mathrm{char}(\field)=0 & 8 & 6 & \mathrm{Theorem~\ref{thm:sepbeta(C5rtimesC4)}} \\
(20,4) & \mathrm{Dih}_{20} & 11& 11 & \text{\cite[Theorem 2.1]{domokos-schefler:16}} \\
(20,5) & \mathrm{C}_2 \times \mathrm{C}_{10} &  11 & 11 & \text{\cite[Theorem 3.10]{domokos:abelian}} \\
(21,1)& \mathrm{C}_7 \rtimes \mathrm{C}_3  & 9 & 8 & \text{\cite[Theorem 4.1]{cziszter:C7rtimesC3}} \\
(22,1) & \mathrm{Dih}_{22} & 12& 12 & \text{\cite[Theorem 2.1]{domokos-schefler:16}} \\
(24,1) & \mathrm{C}_3 \rtimes \mathrm{C}_8  & 13 & 13 & \text{\cite[Theorem 2.1]{domokos-schefler:16}} \\
(24,3)& \mathrm{SL}_2(\mathbb{F}_3) = \tilde{A}_4&  12 & 12 & \mathrm{Theorem~\ref{thm:sepbeta(A4tilde)}}  \\
(24,4) & \mathrm{Dic}_{24} = \mathrm{C}_3 \rtimes \mathrm{Dic}_8  &14 & 14 & \text{\cite[Theorem 2.1]{domokos-schefler:16}} \\
(24,5) & \mathrm{Dih}_6\times \mathrm{C}_4  &  13 & 13 & \text{\cite[Theorem 2.2]{domokos-schefler:16}}  \\
(24,6) & \mathrm{Dih}_{24} & 13& 13 & \text{\cite[Theorem 2.1]{domokos-schefler:16}} \\
(24,7) & \mathrm{Dic}_{12} \times \mathrm{C}_2 \qquad \mathrm{char}(\field)=0 & 9 & 8 & \mathrm{Theorem~\ref{thm:sepbeta(Dic12xC2)}} \\
(24,8)& \mathrm{C}_3 \rtimes \mathrm{Dih}_8 = (\mathrm{C}_6 \times \mathrm{C}_2) \rtimes_{\gamma} \mathrm{C}_2  & 9 & 9 & \mathrm{Theorem~\ref{thm:sepbeta((C6xC2)rtimesC2)}} \\
(24,9) & \mathrm{C}_2 \times \mathrm{C}_{12} &  13 & 13 & \text{\cite[Theorem 3.10]{domokos:abelian}} \\
(24,10)& \mathrm{Dih}_{8} \times \mathrm{C}_3 & 13 & 13 & \text{\cite[Theorem 2.1]{domokos-schefler:16}} \\
(24,11)& \mathrm{Dic}_8 \times \mathrm{C}_3 & 13 & 13 & \text{\cite[Theorem 2.1]{domokos-schefler:16}} \\
(24,12)& S_4  & 9  &9 & \mathrm{Theorem~\ref{thm:sepbeta(S4)}} \\
(24,13) & A_4 \times \mathrm{C}_2  & 8 & 6 & \mathrm{Theorem~\ref{thm:sepbeta(A4xC2)}} \\
(24,14) & \mathrm{Dih}_{12} \times \mathrm{C}_2 = (\mathrm{C}_ 6 \times \mathrm{C}_2)\rtimes_{-1} \mathrm{C}_2 &8  & 8& \text{\cite[Theorem 2.2]{domokos-schefler:16}} \\
(24,15) &\mathrm{C}_2 \times \mathrm{C}_2 \times \mathrm{C}_6  & 8 & 8 & 
\text{\cite[Theorem 3.10]{domokos:abelian}}\\
(25,2) &\mathrm{C}_5 \times \mathrm{C}_5 & 9 & 6 & 
\text{\cite[Theorem 1.2]{schefler_c_n^r}}\\
(26,1) & \mathrm{Dih}_{26}  &  14 & 14 & \text{\cite[Theorem 2.1]{domokos-schefler:16}} \\
(27,2) &\mathrm{C}_3 \times \mathrm{C}_9 & 11 & 10 & 
\text{\cite[Theorem 1.1]{schefler_rank2}}\\
(27,3)& \mathrm{H}_{27}=\mathrm{UT}_3(\mathbb{F}_3)  & 9 & 9 & 
\mathrm{Theorem~\ref{thm:sepbeta(H27)}} \\
(27,4)& \mathrm{M}_{27} = \mathrm{C}_9 \rtimes \mathrm{C}_3, \qquad \mathrm{char}(\field)=0  & 11 & 10 & 
\mathrm{Theorem~\ref{thm:sepbeta(M27)}}\\
(27,5) &\mathrm{C}_3 \times \mathrm{C}_3 \times \mathrm{C}_3 & 7 & 6 & 
\text{\cite[Theorem 1.2]{schefler_c_n^r}}\\
(28,1) & \mathrm{Dic}_{28}=\mathrm{C}_7\rtimes \mathrm{C}_4   & 16 & 16 & \text{\cite[Theorem 2.1]{domokos-schefler:16}} \\ 
(28,3) & \mathrm{Dih}_{28} &  15 & 15 & \text{\cite[Theorem 2.1]{domokos-schefler:16}} \\
(28,4) & \mathrm{C}_2 \times \mathrm{C}_{14} &  15 & 15 & \text{\cite[Theorem 3.10]{domokos:abelian}} \\
(30,1) & \mathrm{Dih}_6\times \mathrm{C}_5  &   16 & 16 & \text{\cite[Theorem 2.1]{domokos-schefler:16}} \\
(30,2) & \mathrm{Dih}_{10}\times \mathrm{C}_3 &  16 & 16& \text{\cite[Theorem 2.1]{domokos-schefler:16}} \\ 
(30,3) & \mathrm{Dih}_{30}  &  16 & 16 & \text{\cite[Theorem 2.1]{domokos-schefler:16}} \\
\hline 
\end{array}$ 
\end{center}
\end{table}
The new parts of Theorem~\ref{thm:sepbeta(<32)} are the values of $\sepbeta(G)$ for the 
$10$ non-abelian groups in Table~\ref{table:main}  that do not contain a cyclic subgroup of index $2$, and are not isomorphic to 
$\mathrm{D}_{2n}\times \mathrm{C}_2$ or to $\mathrm{C}_7\rtimes \mathrm{C}_3$. 
The first column of Table~\ref{table:main} contains the identifier of $G$ in the Small Groups Library of GAP (see \cite{GAP4}); that is, a pair of the form \texttt{(order, i)}. The GAP command $\mathsf{SmallGroup(id)}$ returns the $\texttt{i}$-th group of order $\texttt{order}$ in the catalogue. The second column gives $G$ in a standard notation or indicates its structure as a direct product or semidirect product.  
The symbol $\rtimes$ always stands for a semidirect product that is not a direct product 
($\gamma$ in the row $(24,8)$ stands for an involutive 
automorphisms of $\mathrm{C}_6\times \mathrm{C}_2$ different from $g\mapsto g^{-1}$). 
The third column of the table is filled in with the values of $\beta^\field(G)$ (taken from \cite{cziszter-domokos-szollosi}). The fourth column contains 
$\sepbeta^\field(G)$. In the last column one can find the references for the theorems discussing the value of $\sepbeta^\field(G)$. 

\subsection{Comments on the base field, on computer usage} 
By Lemma~\ref{lemma:base field} below, the value in the column of $\sepbeta^\field(G)$ 
in Table~\ref{table:main} is an upper bound for the separating Noether number under 
the assumption that "$\mathrm{char}(\field)$ does not divide $|G|$" 
(which is weaker than the assumption that "$\field$ contains an element of multiplicative order $|G|$"  in 
Theorem~\ref{thm:sepbeta(<32)}).   

Standard computer algebra packages can compute a minimal system of generators 
for rings of invariants of finite groups. We made a limited use of the online CoCalc platform \cite{CoCalc} for this purpose. 
Namely, we verified by it some desired upper bounds for $\beta(G,V)$ for certain representations of the groups $\mathrm{C}_5\rtimes\mathrm{C}_4$, $\mathrm{Dic}_{12}\times \mathrm{C}_2$, and $\mathrm{M}_{27}$, 
in order to avoid a significant increase in the length of the paper. In these cases we did the calculations 
only in the most classical case of a characteristic zero base field 
(we could not settle all the infinitely many non-modular characteristics by computer anyway),  
that is why for these groups we need to assume $\mathrm{char}(\field)=0$. 

We are not aware of a computer algebra package that computes the separating Noether number for a finite group. 
A simple algorithm for algebraically closed base fields is sketched in \cite{domokos-schefler:16}, and we used 
an implementation of it on \cite{CoCalc}  to verify over the complex base field the results of 
\cite{domokos-schefler:16} on the groups of order at most $16$. The same computational power 
turned out to be insufficient to deal with the groups studied in the present paper. 
Moreover, the results of the present paper in full power (valid in all but finitely many characteristics, for infinitely many different fields that are not necessarily algebraically closed)  
could not be obtained by computer calculations, regardless of the available computational power. 

In Section~\ref{sec:prel}  we introduce notation and 
mention general facts used throughout the paper. 
In Section~\ref{sec:easy groups} we deal with those groups from Theorem \ref{thm:sepbeta(<32)} for which the separating Noether number equals 
the Noether number. By the inequality \eqref{eq:sepbeta<beta} in these cases it suffices to give a representation and 
two points with distinct orbits that can not be separated by 
invariants of strictly smaller degree than the Noether number. 
The cases when the separating 
Noether number is strictly smaller than the Noether number require   
more elaborate work, and this is done in the remaining Sections: 
we need a thorough analysis of 
concrete representations, discussions of stabilizers, 
group automorphisms, 
construction of 
invariants and relative invariants, as well as  
several ad hoc ideas that help to understand the set of solutions of systems of 
polynomial equations. 
What makes these computations feasible is that general principles imply 
that
it is sufficient to deal with representations having a small 
number of irreducible direct summands (see \cite[Lemma 6.2]{domokos-schefler:16} about the Helly dimension). 

\section{Preliminaries} \label{sec:prel}

\subsection{Notational conventions.} \label{subsec:convention} 
In the sequel we shall typically define an irreducible $\field G$-module $V$ for a group $G$ given by generators and relations as follows. We specify a group homomorphism $\psi:G\to \mathrm{GL}_n(K)$ by giving the images of the generators. 
Then $V=K^n$ is the space of column vectors, and $g\cdot v:=\psi(g)v$ (matrix multiplication on the right hand side). 
The induced action of $G$ on the variables $x_1,\dots,x_n$ (which is a basis in the dual space $V^*$ of $V$) is given by   
\[g\cdot x_j=\sum_{i=1}^n\psi(g^{-1})_{ji}x_i.\] 

Let $\widehat G=\mathrm{Hom}(G,\field^\times)$ be the character group of $G$. 
We denote by $U_\chi$ the $1$-dimensional 
vector space corresponding to the element $\chi\in \widehat G$; 
that is, 
$g\cdot \lambda=\chi(g)\lambda$ for $g\in G$ and $\lambda\in \field=U_\chi$. 
Denote by $t_\chi$ the coordinate function on $U_\chi$. Therefore we have $g\cdot t_\chi=\chi(g^{-1})t_\chi$.
 
Usually we shall consider a $\field G$-module $V$ with a direct sum decomposition 
\begin{equation}\label{eq:V+U} 
V=W\oplus U, \qquad W=W_1\oplus\cdots\oplus W_l, \qquad U= U_{\chi_1} \oplus \cdots \oplus U_{\chi_m}, 
\end{equation}   
where the $W_i$ are pairwise non-isomorphic irreducible $\field G$-modules of dimension at least $2$, and $\chi_1,\dots,\chi_m$ are distinct characters of $G$. The coordinate functions on $W_1,W_2$ respectively $W_3$, etc, will be denoted by $x_1,x_2,\dots$, $y_1,y_2,\dots$, respectively $z_1,z_2,\dots$, etc. 

Associated to the above direct sum decomposition \eqref{eq:V+U} 
of $V$ there is 
an $\mathbb{N}_0^{l+m}$-grading on $\field[V]$. Namely, 
the component of multidegree $\alpha=(\alpha_1,\dots,\alpha_{l+m})$ 
of $\field[V]$ consists of the polynomials homogeneous 
of degree $\alpha_j$ on the $j$th direct summand of  
$V=W_1\oplus\cdots\oplus U_{\chi_m}$ for each $j=1,\dots,l+m$. 
This multigrading is induced 
by the following 
action of the torus $T:=\field^\times\times \cdots \times \field^\times$ ($l+m$ direct factors) 
on $V$: for $\lambda=(\lambda_1,\dots,\lambda_{l+m})\in T$ and $v=(w_1,\dots,w_l,u_1,\dots,u_m)\in V$ set 
$\lambda\cdot v:=(\lambda_1 w_1,\dots,\lambda_{l+m} u_m)$. 
This action commutes with the $G$-action on $V$, and therefore 
$\field[V]^G$ is spanned by multihomogeneous elements. 
Moreover, for $v,v'\in V$ we have that 
\begin{equation}\label{eq:rescaling} 
G\cdot v=G\cdot v' \text{ if and only if }G\cdot (\lambda v)=
G\cdot (\lambda v'). 
\end{equation} 
When we say in the text that 'after rescaling, we may assume that $v$ has some special form', then we mean 
that we replace $v$ by an appropriate element in its $T$-orbit, 
and we refer to \eqref{eq:rescaling}. 
The phrase `multihomogeneous' will refer to this multigrading in the text.

\subsection{Relative invariants} \label{sec:Davenport}

For a $\field G$-module $V$, denote by $\field[V]^G_+$ the sum of the positive degree homogeneous components of $\field[V]^G$. A homogeneous invariant $f\in \field[V]^G$ is \emph{indecomposable} if $f\notin (\field[V]^G_+)^2$. Clearly, $\field[V]^G$ is generated as an algebra by the positive degree homogeneous indecomposable invariants. The \emph{Hilbert ideal} $\mathcal{H}(G,V)$ is the ideal in the polynomial algebra $\field[V]$ generated by the positive degree homogeneous $G$-invariants. 

We say that an element $f\in \field[V]$ is a \emph{relative invariant of weight $\chi\in \widehat G$} if 
$g\cdot f=\chi(g^{-1})f$ for all $g\in G$ and set 
$\field[V]^{G,\chi}:=\{f\in \field[V]\mid f \text{ is a relative invariant of weight }\chi\}$.  
For example, the coordinate function $t_\chi$ on $U_\chi$ is a relative invariant of weight $\chi$. 
Observe that for a sequence $\chi^{(1)},\dots,\chi^{(k)}$ of elements of $\widehat G$, 
$t_{\chi^{(1)}}\cdots t_{\chi^{(k)}}$ is an indecomposable $G$-invariant if and only if $\chi^{(1)},\dots,\chi^{(k)}$  is an \emph{irreducible product-one sequence over $\widehat G$}, i.e.
$\chi^{(1)}\cdots \chi^{(k)}=1\in \widehat G$  
and there is no 
$\ell<k$, $1\le i_1<\cdots<i_{\ell}\le k$ with $\chi^{(i_1)}\cdots \chi^{(i_\ell)}=1\in \widehat G$. Here $k$ is the \emph{length} of the sequence $\chi^{(1)},\cdots, \chi^{(k)}$.  
We say that $\chi^{(1)},\dots,\chi^{(k)}$ is a \emph{product-one free sequence over $\widehat G$} if there is no $1\le i_1<\dots <i_\ell\le k$ with 
 $\chi^{(i_1)}\cdots \chi^{(i_\ell)}=1\in \widehat G$.  
The maximal length of an irreducible product-one sequence over $G$ is called the \emph{Davenport constant} $\mathsf{D}(G)$ of the group $G$. 
For a finite abelian group $G$, the maximal possible length of a product-one free sequence is $\mathsf{D}(G)-1$.  
For more information about the Davenport constant and its relevance for 
the invariant theory of abelian groups see for example \cite{cziszter-domokos-geroldinger}. 

In some cases we will compute a homogeneous generating system of 
$\field[V]^G$ by using the following observation:
\begin{lemma}\label{lemma:V+U} 
Let $V=W\oplus U$ be a $\field G$-module as in \eqref{eq:V+U}. 
Let $A$ be a homogeneous generating system of $\field[W]^G$, and for 
$\chi\in \{\chi_1,\dots,\chi_m\}$ let 
$A_{\chi}\subset \field[W]^{G,\chi}$ be a set of homogeneous relative invariants of weight $\chi$ 
that span a vector space direct complement in $\field[W]^{G,\chi}$ of 
$\mathcal{H}(G,W)\cap \field[W]^{G,\chi}=\field[W]^G_+ \field[W]^{G,\chi}$. 
Set 
\[B:=\{t_{\chi^{(1)}}\cdots t_{\chi^{(k)}}\mid \chi^{(1)},\dots, \chi^{(k)} 
\text{ is an irreducible product-one sequence over }\widehat G\},\]  
\begin{align*}C:=\{ht_{\chi^{(1)}}\cdots t_{\chi^{(k)}}\mid \chi^{(1)},\dots, \chi^{(k)} 
\text{ is a product-one free sequence over }\widehat G, \\
\quad h\in A_\chi \text{ where }\chi^{-1}=\chi^{(1)}\cdots \chi^{(k)}\}.\end{align*}
Then $A\cup B\cup C$ is a homogeneous generating system of 
$\field[V]^G$. 
\end{lemma} 
\begin{proof}
As we pointed out above, the algebra $\field[V]^G$ is generated by indecomposable 
multihomogeneous invariants. A multihomogeneous invariant $f$ belongs to 
$\field[W]^G$ or $\field[U]^G$ or is of the form 
$ht$, where $t$ is a monomial in $\field[U]$ of positive degree, 
and $h$ is a multihomogeneous element of positive degree in $\field[W]$.  
In the first two cases $f$ belongs to the $\field$-subalgebra 
of $\field[V]$ generated by $A\cup B$. 
In the third case $t=t_{\chi^{(1)}}\cdots t_{\chi^{(k)}}$, 
where $\chi^{(1)},\dots, \chi^{(k)}$ is a product-one free sequence over 
$\widehat G$. 
Then $t$ is a relative $G$-invariant of some weight $\chi$, and 
hence necessarily, $h\in \field[V]^{G,\chi^{-1}}$. 
There exists a linear combination $a$ of elements from 
$A_{\chi^{-1}}$ with $h-a\in \field[W]^G_+\field[W]^{G,\chi^{-1}}$. 
Then $at$ belongs to the $\field$-subspace of $\field[V]$ generated by 
$C$, and $f-at=(h-a)t\in (\field[V]^G_+)^2$. So we showed that any indecomposable 
multihomogenous element of $\field[V]^G$ can be reduced modulo 
$(\field[V]^G_+)^2$ to an element in the $\field$-subalgebra of $\field[V]$ 
generated by $A\cup B\cup C$. This clearly implies our statement. 
\end{proof}

Given a set $S\subset \field[V]$ of polynomials write 
$\mathcal{V}(S):=\{v\in V\mid \forall f\in S\colon f(v)=0\}$ 
for the common zero locus in $V$ of the elements of $S$. 
Next we restate \cite[Lemma 4.3]{cziszter:C7rtimesC3} in a sharp form: 

\begin{lemma}\label{lemma:common zero locus}
    We have the equality 
    $\mathcal{V}(\field[V]^{G,\chi})=\{v\in V\mid 
    \mathrm{Stab}_G(v)\nsubseteq \ker(\chi)\}$.
\end{lemma}

\begin{proof} 
Suppose that $v\in V$, $g
\in G$ with $g\cdot v=v$. 
For any $f\in \field[V]^{G,\chi}$ we have 
\[f(v)=f(g\cdot v)=\chi(g)f(v),\]
therefore $\chi(g)\neq 1$ implies $f(v)=0$. 
This shows the inclusion 
$\{v\in V\mid 
    \mathrm{Stab}_G(v)\nsubseteq \ker(\chi)\}\subseteq 
    \mathcal{V}(\field[V]^{G,\chi})$. 
The reverse inclusion is stated and proved in
\cite[Lemma 4.3]{cziszter:C7rtimesC3}. 
    \end{proof}

\subsection{Some reductions}
Both the Noether number and the separating Noether number are monotone for subgroups as well; 
that is, 
for any subgroup $H$ of a finite group $G$ we have 
\begin{equation}\label{eq:sepbeta(H)}
\beta^\field(G)\ge \beta^\field(H) 
\mbox{ (see \cite{schmid}, \cite{cziszter-domokos:lower bound}) and }
\sepbeta^\field(G)\ge \sepbeta^\field(H) 
\mbox{ (see  \cite[Theorem B]{kohls-kraft}).} 
\end{equation} 

\begin{lemma}\label{lemma:base field}\cite[Lemma 2.4 (iii)]{domokos-schefler:16}\label{lemma:spanning invariants} 
For any subfield $\field$ of a field $L$ we have the inequality 
$\sepbeta^\field(G)\le \sepbeta^L(G)$. 
\end{lemma}

A consequence of the results of \cite{draisma-kemper-wehlau} is that if $\field$ is a field with sufficiently many elements, one can compute compute $\sepbeta^\field(G)$ by dealing only with multiplicity-free representations. 

\begin{lemma}\cite[Lemma 6.1]{domokos-schefler:16} \label{lemma:multfree} 
Let $V_1,\dots,V_q$ be a complete irredundant list of representatives of the isomorphism classes of simple $\field G$-modules. 
Assume that 
\[|K|>(\max\{\dim(V_j)\mid j=1,\dots,q\}-1)|G|.\] 
Then we have the equality 
$\sepbeta^\field(G)=\sepbeta(G,V_1\oplus\cdots
\oplus V_q)$. 
\end{lemma} 
 
Given a group homomorphism $\rho$ from $G$ into the group $\mathrm{GL}(V)$ of invertible linear transformations of $V$, 
we write $(V,\rho)$ for the $\field G$-module $V$ with $g\cdot v:=\rho(g)(v)$ for $g\in G$, $v\in V$. 
We make use of the natural action of the automorphism group of $G$ on the isomorphism classes 
of representations of $G$ through the following statement: 

\begin{lemma}\cite[Lemma 6.4]{domokos-schefler:16}\label{lemma:auto}
Let $(V,\rho)$ be a $\field G$-module and $\alpha$ an automorphism of the group $G$. 
\begin{itemize} 
\item[(i)] For any $\chi\in \widehat G$ we have 
$\field[(V,\rho)]^{G,\chi}= 
\field[(V,\rho\circ\alpha)]^{G,\chi\circ\alpha}$.  
\item[(ii)] We have 
$\beta(G,(V,\rho))=\beta(G,(V,\rho\circ\alpha))
\text{ and }\sepbeta(G,(V,\rho))=\sepbeta(G,(V,\rho\circ\alpha))$.
\item[(iii)] Assume that $(V,\rho)\cong (V,\rho\circ\alpha)$ as $\field G$-modules, so there exists a 
$\field$-vector space isomorphism $T:V\to V$ with $T\circ \rho(g)=(\rho\circ\alpha)(g)\circ T$ 
for all $g\in G$. 
Then the map $f\mapsto f\circ T$ induces a $\field$-vector space isomorphism 
\[\field[(V,\rho)]^{G,\chi}\cong \field[(V,\rho)]^{G,\chi\circ\alpha}.\] 
\end{itemize}
\end{lemma}

\section{Groups $G$ with $\sepbeta^\field(G)=\beta^\field(G)$}\label{sec:easy groups} 
\subsection{The group $(\mathrm{C}_3\times \mathrm{C}_3)\rtimes_{-1}\mathrm{C}_2$}\label{sec:C3xC3rtimesC2} 
In this subsection 
\[G:=(\mathrm{C}_3\times \mathrm{C}_3)\rtimes_{-1}\mathrm{C}_2=\langle a,b,c\mid 1=a^3=b^3=c^2,\ ab=ba,\ 
cac^{-1}=a^{-1},\ cbc^{-1}=b^{-1}\rangle.\] 
Assume that $\field$ contains an element $\xi$ of multiplicative order $6$. So $\omega:=\xi^2$ has multiplicative order $3$.
Consider the following irreducible $2$-dimensional representations of $G$:  
\[ \psi_1: a\mapsto 
\begin{bmatrix} % or pmatrix or bmatrix or Bmatrix or ...
      \omega & 0 \\
      0 & \omega^2 \\
   \end{bmatrix}, \quad 
   b\mapsto \begin{bmatrix} 
      1 & 0 \\
      0 & 1 \\
   \end{bmatrix}, \quad 
   c\mapsto \begin{bmatrix} 
      0 & 1 \\
      1 & 0 \\
   \end{bmatrix}, \]
\[\psi_2: a\mapsto 
\begin{bmatrix} 
      1 & 0 \\
      0 & 1 \\
   \end{bmatrix}, \quad 
   b\mapsto \begin{bmatrix} 
      \omega & 0 \\
      0 & \omega^2 \\
   \end{bmatrix}, \quad 
   c\mapsto \begin{bmatrix} 
      0 & 1 \\
      1 & 0 \\
   \end{bmatrix}. 
\]
Denote by $W_j$ the vector space $\field^2$ endowed with the  representation $\psi_j$ for $j=1,2$ 
(see Section~\ref{subsec:convention}). 

\begin{proposition}\label{prop:C3xC3rtimesC2} 
We have $\sepbeta(G,W_1\oplus W_2)\ge 6$. 
\end{proposition} 

\begin{proof}
Note that $\ker(\psi_1)=\langle b\rangle$ and $\ker(\psi_2)=\langle a\rangle$. 
We have $G/\ker(\psi_1)\cong \mathrm{S}_3$ and $G/\ker(\psi_2)\cong \mathrm{S}_3$. 
Both these representations factor through the irreducible $2$-dimensional representation of $\mathrm{S}_3$.  
It is well known and easy to see that the corresponding algebra of invariants is generated by a degree $2$ and a degree $3$ invariant. Moreover, $\field[W_j]/\mathcal{H}(G,W_j)$ 
as an $\mathrm{S}_3=G/\ker(\psi_j)$-module is isomorphic to the regular representation of $\mathrm{S}_3$. 
Thus $x_1^3-x_2^3$ spans the only minimal non-trivial $G$-invariant subspace in a 
$\field G$-module direct complement of $\mathcal{H}(G,W_1)$ in $\field[W_1]$ on which $\ker(\psi_2)$ acts trivially, and $y_1^3-y_2^3$ spans the only minimal non-trivial $G$-invariant subspace in a 
$\field G$-module direct complement of $\mathcal{H}(G,W_2)$ in $\field[W_2]$ on which $\ker(\psi_1)$ acts trivially. 
It follows that $\field[W_1\oplus W_2]^G$ is generated by 
$\field[W_1]^G$, $\field[W_2]^G$, and $f:=(x_1^3-x_2^3)(y_1^3-y_2^3)$. 
Consider the points $v=(w_1,w_2)=([1,0]^T,[1,0]^T)$ and $v'=(w'_1,w'_2)=([1,0]^T,[0,1]^T)$. 
We claim that all invariants of degree at most $5$ agree on $v$ and on $v'$, 
but $f(v)\neq f(v')$. Indeed, $w_1=w'_1$ and $c\cdot w_2=w'_2$. Thus the invariants in $\field[W_1]^G$ and in $\field[W_2]^G$ agree on $v$ and on $v'$. 
On the other hand, $f(v)=1$ and $f(v')=-1$. The proof is finished. 
\end{proof}

\begin{theorem} \label{thm:sepbeta((C3xC3)rtimesC2)} 
Assume that $\field$ has an element of multiplicative order $6$. 
Then we have $\sepbeta^\field((\mathrm{C}_3\times \mathrm{C}_3)\rtimes_{-1} \mathrm{C}_2)=6$. 
\end{theorem}

\begin{proof} 
By \cite[Corollary 5.5]{cziszter-domokos:indextwo} we have $\beta^\field(G)=6$. 
Thus the result follows from Proposition~\ref{prop:C3xC3rtimesC2}. 
\end{proof} 

\subsection{The group $\mathrm{C}_7\rtimes \mathrm{C}_3$}
The value of $\sepbeta^{\field}(\mathrm{C}_7\rtimes \mathrm{C}_3)$ is obtained as a special case of a result in \cite{cziszter:C7rtimesC3} under the assumption that $\field$ is algebraically closed. This can be generalized to the case when $\field$ contains an element of multiplicative order $|G|$.

\begin{proposition}\cite[Theorem 4.1]{cziszter:C7rtimesC3}
Suppose that $p\equiv 1\mod 3$ is a prime and $\field$ is algebraically closed. Then for the group $\mathrm{C}_p\rtimes \mathrm{C}_3
=\langle a,b\mid a^p=b^3=1_G,\quad bab^{-1}=a^2\rangle$ we have $\sepbeta^{\field}(\mathrm{C}_p\rtimes \mathrm{C}_3)=p+1$.
\end{proposition}

\begin{corollary}
Consider the group $\mathrm{C}_7\rtimes \mathrm{C}_3
=\langle a,b\mid a^7=b^3=1_G,\quad bab^{-1}=a^2\rangle$ and suppose that $\field$ contains an element of multiplicative order $21$. 
Then $\sepbeta^\field(\mathrm{C}_7\rtimes \mathrm{C}_3)=8$. 
\end{corollary}
\begin{proof}
Actually,  \cite{cziszter:C7rtimesC3} works over an algebraically closed base field, 
but the representation providing the exact lower bound for the separating Noether number 
is defined over any field containing an element of multiplicative order $3p$, 
and in order to prove that $p+1$ is also an upper bound, 
we may pass to the algebraic closure using 
Lemma~\ref{lemma:spanning invariants}. As the special case $p=7$, we are done. 
\end{proof}

\subsection{The binary tetrahedral group} 

In this subsection 
$G:=\widetilde{\mathrm{A}}_4$ is the binary tetrahedral group: 
the special orthogonal group $\mathrm{SO}(3,\mathbb{R})$ has a subgroup isomorphic 
to $\mathrm{A}_4$ (the group of rotations mapping a regular tetrahedron into itself), and 
$G$ is isomorphic to the preimage of $\mathrm{A}_4$ under the classical two-fold covering 
$\mathrm{SU}(2,\mathbb{C})\to \mathrm{SO}(3,\mathbb{R})$. So $G$ can be  embedded as a subgroup into the special unitary group $\mathrm{SU} (2,\mathbb{C})$, acting via its defining representation on $V:=\mathbb{C}^2$. Thus $V$ is a $\field G$-module, and it is well known that $\mathbb{C}[V]^G$ is generated by three homogeneous invariants $f_6$, $f_8$, $f_{12}$, having degree $6$, $8$, and $12$ 
(see \cite[Lemma 4.1]{huffman} or \cite[Section 0.13]{popov-vinberg}). Since $G$ (as a subgroup of $\mathrm{GL}(V)$) is not generated by pseudo-reflections, $\mathbb{C}[V]^G$ does not contain an algebraically independent 
separating set by \cite{dufresne}, so 
the invariants $f_6$ and $f_8$ can not separate all the $G$-orbits in $V$.   
This implies that $\sepbeta^\mathbb{C}(G,V)\ge 12$. 
A second proof valid for more general base fields is given below.  

\begin{theorem} \label{thm:sepbeta(A4tilde)}
Suppose in addition that $\field$ contains an element 
$\xi$ of multiplicative order $8$. 
Then we have $\sepbeta^\field(\widetilde{\mathrm{A}}_4)=12$. 
\end{theorem} 

\begin{proof} 
An alternative way to think of $G=\widetilde{\mathrm{A}_4}$ is that 
$G\cong \mathrm{Dic}_8\rtimes \mathrm{C}_3$, the non-abelian semidirect product 
of the quaternion group of order $8$ and the $3$-element group. 
Denote by $a,b$ generators of $G$, where $a^4=1$, $b^3=1$, and 
$\langle a,bab^{-1}\rangle\cong \mathrm{Dic}_8$, such that 
$a,bab^{-1},b^2ab^{-2}$ correspond to the elements of the quaternion group 
traditionally denoted by $\mathrm{i}$, $\mathrm{j}$, $\mathrm{k}$. 
Then $G$ has the representation  

\[\psi:a\mapsto
\begin{bmatrix} % or pmatrix or bmatrix or Bmatrix or ...
      \xi^2 & 0 \\
      0 & -\xi^2 \\
   \end{bmatrix}, \quad    
   b\mapsto 
   \begin{bmatrix} 
      \frac 12(-1-\xi^2) & \frac 12(1+\xi^2) \\
      \frac 12(-1+\xi^2) &   \frac 12(-1+\xi^2) \\
   \end{bmatrix} 
   \]
(indeed, $\psi(b)^3$ is the identity matrix, and 
$\psi(b)\psi(a)\psi(b)^{-1}=
\begin{bmatrix} 
      0 & -1 \\
      1 & 0 \\
   \end{bmatrix}$ together with $\psi(a)$ generate a subgroup in 
   $\mathrm{GL}(\field^2)$ isomorphic to the quaternion group). 
Write $V$ for $\field^2$ endowed with the representation $\psi$. 
It is easy to see that $\field[V]^{\mathrm{Dic}_8}$ is generated by 
$f_1:=(x_1x_2)^2$, $f_2:=(x_1^2+x_2^2)^2$, $f_3:=x_1x_2(x_1^4-x_2^4)$. 
The generator $f_3$ is fixed by $b$ as well, hence $f_3$ is $G$-invariant, 
whereas $f_1$ and $f_2$ span a $\langle b\rangle$-invariant subspace 
in $\field[x_1,x_2]$.  The matrix of the linear transformation $b$ on 
the space $\mathrm{Span}_\field\{f_1,f_2\}$ with respect to the basis 
$4f_1,f_2$ is 
$\begin{bmatrix} 
      0 & 1 \\
      -1 & -1 \\
   \end{bmatrix}.$  
Endow the $3$-variable polynomial algebra $\field[s_1,s_2,s_3]$ with the action of 
$\langle b\rangle\cong \mathrm{C}_3$ given by 
$s_1\mapsto -s_2$, $s_2\mapsto s_1-s_2$, $s_3\mapsto s_3$. 
The $\field$-algebra homomorphism given by 
$s_1\mapsto 4f_1$, $s_2\mapsto f_2$, $s_3\mapsto f_3$ restricts to a surjective 
$\field$-algebra homomorphism from $\field[s_1,s_2,s_3]^{\langle b\rangle}$ onto 
$\field[V]^G$. 
Knowing that $\field[s_1,s_2,s_3]^{\langle b\rangle}$ is generated in degree 
at most $\mathsf{D}(\mathrm{C}_3)=3$, an easy direct computation yields 
\[\field[s_1,s_2,s_3]^{\langle b\rangle}=
\field[s_1^2 - s_1s_2 + s_2^2, s_1^2s_2 - s_1s_2^2, s_1^3 - 3s_1s_2^2 + s_2^3, s_3].\]
This shows that $\field[V]^G$ is generated by  
\begin{align*} 
h_8&:=(4f_1)^2 - 4f_1f_2 + f_2^2 
\\ h_6&:=x_1x_2(x_1^2+x_2^2)(x_1^2-x_2^2) 
\\ h_{12}^{(1)}&:=(4f_1)^2f_2 - 4f_1f_2^2 
\\ h_{12}^{(2)}&:=(4f_1)^3 - 12f_1f_2^2 + f_2^3
\end{align*}
having degrees  
$4,6,12,12$. In fact $h_{12}^{(1)}$ can be omitted, because we have 
$h_{12}^{(1)}=-4h_6^2$. 
Set $v:=[1,\xi^2]^T$ and $v':=\xi v$. Then $(x_1^2+x_2^2)(v)=0=(x_1^2+x_2^2)(v')$, 
hence $f_2(v)=0=f_2(v')$. It follows that $h_6(v)=0=h_6(v')$. 
Moreover, $h_8(v)=16(\xi^4)^2=16=16(\xi^8)^2=h_8(v')$. 
Thus no invariants of degree less than $12$ separate $v$ and $v'$. 
However, they have different $G$-orbit, since 
$h_{12}^{(2)}(v)=4^3(\xi^4)^3=-64$, whereas 
$h_{12}^{(2)}(v')=4^3(\xi^8)^3=64$. 
Thus we proved the inequality 
$\sepbeta(G,V)\ge 12$, implying in turn $\sepbeta^\field(G)\ge 12$.
On the other hand, 
by \cite[Corollary 3.6]{CzD:1} we know that $\beta^\field(G)=12$, 
implying the reverse inequality 
$\sepbeta^\field(G)\le 12$. 
\end{proof}

\subsection{The group $\mathrm{C}_3\rtimes \mathrm{D}_8\cong (\mathrm{C}_6\times \mathrm{C}_2)\rtimes \mathrm{C}_2$.} 
In this subsection 
\[G=\mathrm{C}_3\rtimes \mathrm{D}_8=(\mathrm{C}_6\times \mathrm{C}_2)\rtimes \mathrm{C}_2=\langle a,b,c\mid a^6=b^2=c^2=1,\ ba=ab,\ 
cac=a^{-1},\ cbc=a^3b\rangle.\]  
Assume that $\field$ contains an element $\omega$ of multiplicative order $6$, 
and consider the following irreducible $2$-dimensional representation  of $G$:  
\[ \psi: a\mapsto 
\begin{bmatrix} 
      \omega & 0 \\
      0 & \omega^{-1} \\
   \end{bmatrix}, \quad 
   b\mapsto \begin{bmatrix} 
      1 & 0 \\
      0 & -1 \\
   \end{bmatrix}, \quad 
   c\mapsto \begin{bmatrix} 
      0 & 1 \\
      1 & 0 \\
   \end{bmatrix}. \]
The commutator subgroup of $G$ is $\langle a\rangle$.
Denote by $\chi$ the $1$-dimensional representation of $G$ given by 
\[\chi:a\mapsto 1,\ b\mapsto -1,\ c\mapsto -1\] 
Denote by $W$ the vector space $\field^2$ endowed with the  representation $\psi$, and denote by $U_{\chi}$ the vector space $\field$ 
endowed with the  representation $\chi$. 

\begin{proposition} \label{prop:mingen(C6xC2)rtimesC2)} 
The algebra $\field[W\oplus U_\chi]^G$ is minimally generated by 
$(x_1x_2)^2$, $x_1^6+x_2^6$, $t^2$, $(x_1^6-x_2^6)x_1x_2t$. 
\end{proposition} 

\begin{proof} 
The abelian subgroup $\langle a,b\rangle$ preserves the $1$-dimensional subspaces $\field x_1$, $\field x_2$, $\field t$ in $\field[W\oplus U_\chi]$, hence $\field[W\oplus U_\chi]^{\langle a,b\rangle}$ is generated by $\langle a,b\rangle$-invariant monomials (see Lemma~\ref{lemma:V+U}). So $\field[W\oplus U_\chi]^{\langle a,b\rangle}$ is generated by 
$x_1^6$, $x_2^6$, $(x_1x_2)^2$, $t^2$, $x_1x_2t$. 
The element $c\in G$ fixes $(x_1x_2)^2$ and $t^2$, 
interchanges $x_1^6$ and $x_2^6$,  and $c\cdot (x_1x_2t)=-x_1x_2t$. 
It follows that $\field[W\oplus U_\chi]^G$ is generated by 
$(x_1x_2)^2$, $t^2$, $x_1^6+x_2^6$, $(x_1^6-x_2^6)x_1x_2t$, and 
$(x_1^6-x_2^6)^2$. The latter generator can be omitted, since 
we have $(x_1^6-x_2^6)^2=(x_1^6+x_2^6)^2-4(x_1^2x_2^2)^3$. 
\end{proof} 

\begin{theorem}\label{thm:sepbeta((C6xC2)rtimesC2)} 
Assume that $\field$ contains an element $\xi$ of multiplicative order $12$. 
Then we have $\sepbeta^\field((\mathrm{C}_6\times \mathrm{C}_2)\rtimes \mathrm{C}_2)=9$. 
\end{theorem} 
\begin{proof} 
Set
$v:=([1,\xi]^T,1)\in W\oplus U_\chi$ and $v':=([1,\xi]^T,-1)\in W\oplus U_\chi$. 
The invariant $(x_1^6-x_2^6)x_1x_2t$ has different values on $v$ and $v'$.  
By Proposition~\ref{prop:mingen(C6xC2)rtimesC2)} we see that all $G$-invariants of degree 
less than $9$ agree on $v$ and $v'$. This shows that 
$\sepbeta^\field((\mathrm{C}_6\times \mathrm{C}_2)\rtimes \mathrm{C}_2)\ge 9$. 
On the other hand, $\beta^\field((\mathrm{C}_6\times \mathrm{C}_2)\rtimes \mathrm{C}_2)=9$ by 
\cite[Proposition 3.5]{cziszter-domokos-szollosi}, implying the reverse inequality 
$\sepbeta^\field((\mathrm{C}_6\times \mathrm{C}_2)\rtimes \mathrm{C}_2)\le 9$. 
\end{proof} 

\subsection{The symmetric group $\mathrm{S}_4$ of degree $4$.} 

\begin{theorem}\label{thm:sepbeta(S4)} 
Assume that $\field$ has an element of multiplicative order 
$24$ or $\mathrm{char}(\field)\neq 5$. Then we have 
$\sepbeta^\field(\mathrm{S}_4)=9$. 
\end{theorem}
\begin{proof} 
We have $\beta^\field(\mathrm{S}_4)=9$ by \cite[Example 5.3]{cziszter-domokos-geroldinger}, 
implying the inequality $\sepbeta^\field(\mathrm{S}_4)\le 9$. 

In order to show the reverse inequality consider $V=\field^4$ endowed with the tensor product  of the sign representation of $\mathrm{S}_4$ and the standard action of $\mathrm{S}_4$ on $\field^4$ via permuting the coordinates. 
Thus $\pi \in \mathrm{S}_4$ maps the variable $x_j$ to 
$\mathrm{sign}(\pi)x_{\pi(j)}$. Denote by $\sigma_j$ ($j=1,2,3,4)$ the elementary symmetric polynomial of degree $j$, and set $\Delta:=\prod_{1\le i<j\le 4}(x_i-x_j)$. 
The algebra $\field[V]^{\mathrm{S}_4}$ is minimally generated by  
$\sigma_1^2$, $\sigma_2$,  $\sigma_1\sigma_3$, $\sigma_4$, $\sigma_3^2$, 
$\sigma_1\Delta$, $\sigma_3\Delta$ (see \cite[Example 5.3]{cziszter-domokos-geroldinger}). 

We shall show that there exists a $v\in V$ with 
\begin{equation}\label{eq:S4 good v} 
\sigma_1(v)=0, \quad \Delta(v)\neq 0 \quad \text{ and }
\sigma_3(v)\neq 0.
\end{equation} 
Then setting $v'=-v$ we have that all 
the even degree invariants agree on $v$ and $v'$. 
The only odd degree invariant of degree less than $9$ is $\sigma_1\Delta$ 
(up to non-zero scalar multiples), and it vanishes both on $v$ and $v'$ by 
$\sigma_1(v)=0=\sigma_1(v')$. 
On the other hand, $\Delta(v)=\Delta(v')\neq 0$, $\sigma_3(v)=-\sigma_3(v')\neq 0$ 
imply that $(\sigma_3\Delta)(v)\neq (\sigma_3\Delta)(v')$. 
This shows that $\sepbeta^\field(\mathrm{S}_4)\ge 9$. 

 It remains to prove the existence of $v$ 
 satisfying \eqref{eq:S4 good v}. 
If $\mathrm{char}(\field)\neq 5$,  
then $v:=[0,1,2,-3]^T\in V$ satisfies \eqref{eq:S4 good v}  
(recall that by the running assumption $\mathrm{char}(\field)\notin \{2,3\}$). 
If $\mathrm{char}(\field)=5$ and $\field$ contains an element 
of multiplicative order $24$, then $\field$ contains $\mathbb{F}_{25}$, 
the field of $25$ elements. Let $v_1,v_2$ be the roots 
of the polynomial $x^2+x+1$ in $\mathbb{F}_{25}\subseteq \field$, 
and let $v_3,v_4$ be the roots of 
$x^2-x+2$ in $\mathbb{F}_{25}\subseteq \field$. 
Set $v:=[v_1,v_2,v_3,v_4]^T$. 
Note that 
\[(x^2+x+1)(x^2-x+2)=x^4+2x^2+x+2\in \field[x], \]
showing that $\sigma_1(v)=0$, $\sigma_3(v)=-1$. 
Moreover, $v_1,v_2,v_3,v_4$ are distinct, so $\Delta(v)\neq 0$, 
and thus \eqref{eq:S4 good v} holds. 
\end{proof} 

\subsection{The Heisenberg group $\mathrm{H}_{27}$}
In this subsection 
\[G=\mathrm{H}_{27}=\langle a,b,c \mid a^3=b^3=c^3=1,\ a^{-1}b^{-1}ab=c,\ ac=ca,\ bc=cb \rangle\] 
is the \emph{Heisenberg group} of order $27$. 
It is generated by $a,b$, and its commutator subgroup 
$\langle c\rangle\cong \mathrm{C}_3$, whereas 
$G/G'\cong \mathrm{C}_3\times \mathrm{C}_3$. 

Assume that $\field$ has an element  $\omega$ of multiplicative order $3$, 
and consider the following irreducible $3$-dimensional representation of $G$: 
\[ \psi: a\mapsto 
\begin{bmatrix} 
      1 & 0 & 0 \\ 0 & \omega & 0  
      \\ 0 & 0 & \omega^2   
   \end{bmatrix}, \quad 
   b\mapsto \begin{bmatrix} 
     0 & 0 & 1 \\ 1 & 0 & 0 \\ 0 & 1 & 0   
       \end{bmatrix}, \quad 
   c\mapsto \begin{bmatrix} 
     \omega & 0 & 0 \\ 0 & \omega & 0 \\ 0 & 0 & \omega   
       \end{bmatrix}.\]
Write $V$ for the vector space $\field^3$ endowed with the  representation $\psi$. 

\begin{proposition}\label{prop:H27} 
We have $\sepbeta(G,V)\ge 9$ if $|\field|\neq 4$, whereas   $\sepbeta(G,V)=6$ if $|\field|=4$.
\end{proposition} 

\begin{proof} 
Denote by $H$ the subgroup of $G$ generated by $a$ and $c$. 
The algebra $\field[V]^H$ is generated by the monomials $x_1^3$, $x_2^3$, $x_3^3$, $x_1x_2x_3$. 
It follows that each homogeneous $G$-invariant has degree divisible by $3$. 
The element $b$ fixes $x_1x_2x_3$, and permutes cyclically 
$x_1^3$, $x_2^3$, $x_3^3$. So $\field[V]^G$ is generated by $x_1x_2x_3$ and the generators of $\field[x_1^3,x_2^3,x_3^3]^{\langle b\rangle}$.
One easily deduces that the $\langle b\rangle$-invariants 
in $\field[V]^H$ (i.e. the $G$-invariants in $\field[V]$) of degree at most $6$ are contained in the subalgebra generated by 
$f_1:=x_1x_2x_3$, $f_2:=x_1^3+x_2^3+x_3^3$, $f_3:=x_1^3x_2^3+x_1^3x_3^3+x_2^3x_3^3$. 
If $|\field|\neq 4$ then there exists a non-zero $\lambda\in \field$ with $\lambda^3\neq 1$. 
Now $f_1$, $f_2$, and $f_3$ are all symmetric in the variables $x_1,x_2,x_3$, 
therefore they agree on $v:=[1,\lambda,0]^T$ and $v':=[\lambda,1,0]^T$. 
On the other hand, $v$ and $v'$ have different $G$-orbits. 
Indeed, the algebra 
$\field[V]^G$ contains also 
$f_4:=x_1^3x_2^6 + x_1^6x_3^3 + x_2^3x_3^6$. 
We have  
$f_4(v)=\lambda^6$, whereas 
$f_4(v')=\lambda^3$, 
so $f_4(v)\neq f_4(v')$. 
Thus  $G\cdot v\neq G\cdot v'$, and $v$, $v'$ can not be separated by invariants of degree $<9$.

If $|\field|= 4$, then for any positive integer $n$, the polynomial $x_j^{3n}$ 
induces the same function on  $V$ as $x_j^3$ ($j=1,2,3$). 
Therefore all functions on $V$ that are induced by some element of 
$\field[x_1^3,x_2^3,x_3^3]$ can be induced by a polynomial in 
\[\mathrm{Span}_\field\{1,x_1^3,x_2^3,x_3^3,x_1^3x_2^3,x_1^3x_3^3,x_2^3x_3^3, 
x_1^3x_2^3x_3^3\}.\]
Consequently, $x_1x_2x_3$ together with the $G$-invariants of degree at most $6$ 
in the above space of polynomials   
 form a separating set in $\field[V]^G$. 
In particular, we have 
$\sepbeta(G,V)\le 6$. 
To see the reverse inequality, note that the degree $3$ invariants are linear combinations of 
$x_1x_2x_3$ and $x_1^3+x_2^3+x_3^3$, and they agree on $v:=[0,0,0]^T$ and $v':=[1,1,0]^T$. However, $v$ and $v'$ have different $G$-orbits, 
since $f_3(v)=0\neq 1=f_3(v')$. Thus  $\sepbeta(G,V)=6$ when $|\field|=4$.
\end{proof} 

\begin{theorem}\label{thm:sepbeta(H27)} 
Assume that $\field$ contains an element of multiplicative order $3$ 
and $|\field|\neq 4$. 
Then we have $\sepbeta^\field(\mathrm{H}_{27})=9$. 
\end{theorem} 

\begin{proof} 
By \cite[Corollary 18]{cziszter:p-group} we have $\beta^\field(G)=9$, therefore the result 
follows by Proposition~\ref{prop:H27}. 
\end{proof} 

\section{Some groups $G=H\times \mathrm{C}_2$ with $\sepbeta^\field(G)=\sepbeta^\field(H)$}
\label{sec:HxC2}

\subsection{The group $\mathrm{Dic}_{12}\times \mathrm{C}_2$.} 
In this subsection 
\[G=\mathrm{Dic}_{12}\times \mathrm{C}_2=\langle a,b\mid a^6=1,\ b^2=a^3,\ ba=a^{-1}b\rangle \times 
\langle c\mid c^2=1\rangle.\]  
Assume that $\field$ has an element $\xi$ of multiplicative order $12$; 
set $\omega:=\xi^2$ and $\mathrm{i}:=\xi^3$, so 
$\omega$ has multiplicative order $6$ and $\mathrm{i}$ has multiplicative order $4$. 
Consider the following irreducible $2$-dimensional representations of $G$:  
\[ \psi_1: a\mapsto 
\begin{bmatrix} 
      \omega & 0 \\
      0 & \omega^ {-1}\\
   \end{bmatrix}, \quad 
   b\mapsto \begin{bmatrix} 
      0 & 1 \\
      -1 & 0 \\
   \end{bmatrix}, \quad 
   c\mapsto \begin{bmatrix} 
      -1 & 0 \\
      0 & -1 \\
   \end{bmatrix}\]
\[\psi_2: a\mapsto 
\begin{bmatrix} 
      \omega & 0 \\
      0 & \omega^{-1} \\
   \end{bmatrix}, \quad 
   b\mapsto \begin{bmatrix} 
      0 & 1 \\
      -1 & 0 \\
   \end{bmatrix}, \quad 
   c\mapsto \begin{bmatrix} 
      1 & 0 \\
      0 & 1 \\
   \end{bmatrix} 
\]
\[\psi_3: a\mapsto 
\begin{bmatrix} 
      \omega^2 & 0 \\
      0 & \omega^{-2} \\
   \end{bmatrix}, \quad 
   b\mapsto \begin{bmatrix} 
      0 & 1 \\
      1 & 0 \\
   \end{bmatrix}, \quad 
   c\mapsto \begin{bmatrix} 
      -1 & 0 \\
      0 & -1 \\
   \end{bmatrix} 
\]
\[\psi_4: a\mapsto 
\begin{bmatrix} 
      \omega^2 & 0 \\
      0 & \omega^{-2} \\
   \end{bmatrix}, \quad 
   b\mapsto \begin{bmatrix} 
      0 & 1 \\
      1 & 0 \\
   \end{bmatrix}, \quad 
   c\mapsto \begin{bmatrix} 
      1 & 0 \\
      0 & 1 \\
   \end{bmatrix}. 
\]
These are pairwise non-isomorphic (have distinct characters). 
The other irreducible representations of $G$ are $1$-dimensional, and can be labelled by 
the group $\widehat G=\{\pm \mathrm{i}, \pm 1\}\times \{\pm 1\}\le \field^\times \times \field^\times$, where  $\chi=(\chi_1,\chi_2)\in \widehat G$ is identified with the representation 
\[\chi:a\mapsto \chi_1^2,\ b\mapsto \chi_1,\ c\mapsto \chi_2\] 
(note that $\langle a^2 \rangle$ is the commutator subgroup of $G$, so $a^2$ is in the kernel of any $1$-dimensional representation of $G$). 
For $j=1,2,3,4$ denote by $W_j$ the vector space $\field^2$ endowed with the  representation $\psi_j$, and for $\chi\in \widehat G$ denote by $U_\chi$ the vector space $\field$ 
endowed with the  representation $\chi$, and set 
$U:=\bigoplus_{\chi\in \widehat G}U_\chi$. 

The following result can be easily obtained using the CoCalc platform \cite{CoCalc}: 

\begin{proposition}\label{prop:Dic12xC2,V3+V4+U} 
Assume that $\field$ has characteristic $0$. 
Then we have the equality $\beta(G,V)=8$ if $V$ is any of the $\field G$-modules 
\begin{align*}W_1\oplus W_4 \oplus U,\ W_2\oplus W_4\oplus U,\  W_3\oplus W_4\oplus U, \\
\ W_1\oplus W_2\oplus W_4,\ W_1\oplus W_3\oplus W_4, \ W_2\oplus W_3\oplus W_4. 
\end{align*}  
\end{proposition}

\begin{remark} 
\begin{itemize}
\item[(i)] Proposition~\ref{prop:Dic12xC2,V3+V4+U} implies in particular that 
$\beta(W_i\oplus W_j)\le 8$ for all $1\le i<j\le 3$.
On the other hand, also using computer we found that $\beta(G,W_1\oplus W_2\oplus W_3)=9$, 
and Proposition~\ref{prop:Dic12xC2,V3+V4+U} can not yield immediately a better upper bound for 
$\sepbeta(G,W_1\oplus W_2\oplus W_3)$. 
Indeed, consider $v=(w_1,w_2,w_3):=([1,0]^T,[1,0]^T,[1,1]^T)\in W_1\oplus W_2\oplus W_3$ and 
$v'=(w'_1,w'_2,w'_3):=([1,0]^T,[1,0]^T,[-1,-1]^T)\in W_1\oplus W_2\oplus W_3$.  
Then $(w_1,w_2)=(w'_1,w'_2)$, $b^2c\cdot (w_1,w_3)=(w'_1,w'_3)$, $c\cdot (w_2,w_3)=(w'_2,w'_3)$. 
So all pairs $(w_i,w_j)$ and $(w'_i,w'_j)$ have the same $G$-orbit. 
However, $v$ and $v'$ have different $G$-orbits, because the $G$-invariant 
$x_2y_2z_1+x_1y_1z_2$ separates them.  
\item[(ii)] Note that $|\ker(\psi_4)|=4$ and $|\ker(\psi_j)|=2$ for $j=1,2,3$.  
This may be an explanation for the lack of symmetry in the roles of 
the four $2$-dimensional representations above. 
\end{itemize} 
\end{remark}

\begin{proposition}\label{prop:Dic12xC2,V1+V2+V3} 
Assume that $\field$ has characteristic $0$. 
For $V:=W_1\oplus W_2\oplus W_3$ we have the inequality $\sepbeta(G,V)\le 8$. 
\end{proposition} 

\begin{proof} 
Assume that for $v=(w_1,w_2,w_3)\in V$, $v'=(w'_1,w'_2,w'_3)\in V$ we have 
that $f(v)=f(v')$ for all $f\in \field[V]^G$ with $\deg(f)\le 8$. We need to show that 
$G\cdot v=G\cdot v'$. By Proposition~\ref{prop:Dic12xC2,V3+V4+U} we know that 
$\beta(G,W_i\oplus W_j)\le 8$ for all $i,j\in \{1,2,3,4\}$. In particular, 
$G\cdot (w_1,w_2)=G\cdot (w'_1,w'_2)$. So replacing $v'$ by an appropriate element in its orbit 
we may assume that $w'_1=w_1$, $w'_2=w_2$, and moreover, both $w_1$ and $w_2$ are non-zero. We shall show that necessarily $w_3=w'_3$.
Consider the following $G$-invariants in $\field[V]^G$: 
\begin{align*} 
f_1:=x_2y_2z_1+x_1y_1z_2,
\quad f_2:=y_1y_2(x_2y_2z_1-x_1y_1z_2), 
\quad f_3:=x_1^3y_1z_1+x_2^3y_2z_2,
\\ \ f_4:=x_2^3y_1z_1-x_1^3y_2z_2,  
\qquad f_5:=x_1y_2^3z_1-x_2y_1^3z_2. 
\end{align*} 
For $w\in V$ consider the matrices 
\begin{align*}M_1(w)&:=\begin{bmatrix} 
      x_2(w)y_2(w)& x_1(w)y_1(w) \\
       x_2(w)y_1(w)y_2^2(w) &  -x_1(w)y_1^2(w)y_2(w) 
   \end{bmatrix}
\\ M_2(w)&:=\begin{bmatrix} x_2(w)y_2(w)& x_1(w)y_1(w) \\ x_1^3(w)y_1(w) & x_2^3(w)y_2(w)
   \end{bmatrix} \qquad   
 M_3(w):=\begin{bmatrix} x_2^3(w)y_1(w) & -x_1^3(w)y_2(w) \\ x_1(w)y_2^3(w) & -x_2(w)y_1^3(w)
   \end{bmatrix}. 
\end{align*} 
The definition of the $f_j$ implies that for any $w\in V$ we have  the matrix equalities 
\begin{align*} 
M_1(w)\cdot \begin{bmatrix} z_1(w)\\ z_2(w)\end{bmatrix}=
\begin{bmatrix} f_1(w) \\ f_2(w) \end{bmatrix} 
\quad M_2(w)\cdot \begin{bmatrix} z_1(w)\\ z_2(w)\end{bmatrix}=
\begin{bmatrix} f_1(w) \\ f_3(w) \end{bmatrix}  
\qquad M_3(w)\cdot \begin{bmatrix} z_1(w)\\ z_2(w)\end{bmatrix}=
\begin{bmatrix} f_4(w) \\ f_5(w) \end{bmatrix} 
\end{align*} 
Note that $M_j(v)=M_j(v')$ for $j=1,2,3$ 
(because $(w_1,w_2)=(w'_1,w'_2)$), and since $\deg(f_j)\le 5$, by assumption, we have 
$f_j(v)=f_j(v')$ for $j=1,2,3,4,5$. 
By basic linear algebra we conclude that 
$\begin{bmatrix}z_1(v) \\ z_2(v)\end{bmatrix}
=\begin{bmatrix} z_1(v') \\ z_2(v')\end{bmatrix}$, unless   the matrix $M_j(v)$ has zero determinant for all $j\in \{1,2,3\}$. 
We claim that this is not the case. Suppose to the contrary that $\det M_j(v)=0$ for $j=1,2,3$. 
Then $\det M_1(v)=0$  says that one of $x_1(v)$, $x_2(v)$, $y_1(v)$, 
$y_2(v)$ equals zero. Suppose for example that $x_1(v)=0$. Then $x_2(v)\neq 0$ (as $w_2\neq 0$), and $\det M_2(v)=0$  yields $y_2(v)=0$, implying in turn that 
$y_1(v)\neq 0$. Then $\det M_3(v) \neq 0$, a contradiction. The cases when 
$x_2(v)$, $y_1(v)$, or $y_2(v)$ is zero can be dealt with similarly. 
\end{proof} 

\begin{proposition}\label{prop:Dic12xC2,3  summands}
Assume that $\field$ has characteristic $0$.
Let $i,j\in \{1,2,3\}$, $i\neq j$, and $\chi\in \widehat G$. 
Then $\sepbeta(G,W_i\oplus W_j\oplus U_\chi)\le 8$.    
\end{proposition} 

\begin{proof} 
Using the CoCalc platform \cite{CoCalc} we verified that for a $\field G$-module of the form 
$V=W_i\oplus W_j\oplus U_\chi$ we have $\beta(G,V)>8$ if and only if  $V$ is as in the table below: 
\[\begin{array}{c|c}
V & tf_1,tf_2,tf_3\in \field[V]^G
\\ \hline 
W_2\oplus W_3\oplus U_{(\mathrm{i},-1)} & (y_1z_1+\mathrm{i}y_2z_2)t,\ y_1y_2(y_1z_1-\mathrm{i}y_2z_2)t,\ (y_2z_1^5-\mathrm{i}y_1z_2^5)t
\\ 
W_2\oplus W_3\oplus U_{(-\mathrm{i},-1)} & (y_1z_1-\mathrm{i}y_2z_2)t,\ y_1y_2(y_1z_1+\mathrm{i}y_2z_2)t,\ (y_2z_1^5+\mathrm{i}y_1z_2^5)t
\\ 
W_1\oplus W_3\oplus U_{(\mathrm{i},1)} & (x_1z_1+\mathrm{i}x_2z_2)t,\ x_1x_2(x_1z_1-\mathrm{i}x_2z_2)t,\ (x_2z_1^5-\mathrm{i}x_1z_2^5)t 
\\ 
W_1\oplus W_3\oplus U_{(-\mathrm{i},1)} & (x_1z_1-\mathrm{i}x_2z_2)t,\ x_1x_2(x_1z_1+\mathrm{i}x_2z_2)t,\ (x_2z_1^5+\mathrm{i}x_1z_2^5)t 
\\ 
W_1\oplus W_2\oplus U_{(-1,-1)} & (x_2y_1+x_1y_2)t,\ x_1x_2(x_2y_1-x_1y_2)t,\ (x_1y_1^5-x_2y_2^5)t 
\\  
W_1\oplus W_2\oplus U_{(1,-1)} & (x_2y_1-x_1y_2)t,\ x_1x_2(x_2y_1+x_1y_2)t,\ (x_1y_1^5+x_2y_2^5)t 
\end{array} 
\]
We shall show that $\sepbeta(G,V)\le 8$ for $V=W_2\oplus W_3\oplus U_{(\mathrm{i},-1)}$; 
the argument for the other $V$ in the table above is the same. 
Take $v=(w_2,w_3,u)\in V$ and $v'=(w'_2,w'_3,u')\in V$, and assume that 
$f(v)=f(v')$ holds for all $f\in \field[V]^G$ with $\deg(f)\le 8$. 
We need to show that then $G\cdot v=G\cdot v'$. Recall that the coordinate functions on $W_2$ are denoted by $y_1,y_2$,  on $W_3$ by $z_1,z_2$. Write $t:=t_{(\mathrm{i},-1)}$ for the coordinate function on $U_{(\mathrm{i},-1)}$.  
By Proposition~\ref{prop:Dic12xC2,V3+V4+U} we know that if $w_2=0$, then $w'_2=0$, 
and $(w_3,u)$ and $(w'_3,u')$ have the same $G$-orbit, Similarly, if $w_3=0$, then 
$G\cdot v=G\cdot v'$. So we may assume that $w_2\neq 0$ and $w_3\neq 0$, and again by 
Proposition~\ref{prop:Dic12xC2,V3+V4+U}, by replacing $v'$ by an appropriate element 
in its $G$-orbit, we may assume that $w_2=w'_2$ and $w_3=w'_3$. It remains to show that $u=u'$. 
Denote by $tf_1$, $tf_2$, $tf_3$ the $G$-invariants on $V$ given in the table.  
All have degree less than $8$, so $(tf_j)(v)=(tf_j)(v')$ holds for $j=1,2,3$. 
One can easily deduce from $w_2\neq 0$ and $w_3\neq 0$ that $f_j(w_2,w_3)\neq 0$ for some $j\in \{1,2,3\}$, hence $t(v)=t(v')$, i.e. $u=u'$. 
\end{proof} 

\begin{theorem}~\label{thm:sepbeta(Dic12xC2)} 
Assume that $\field$ has characteristic $0$ and it has an element of multiplictive order $12$. Then we have the equality $\sepbeta^\field(\mathrm{Dic}_{12}\times \mathrm{C}_2)=8$. 
\end{theorem} 

\begin{proof}
Since $\mathrm{Dic}_{12}$ is a homomorphic image of $G$, we have the obvious inequality 
$\sepbeta^\field(G)\ge \sepbeta^\field(\mathrm{Dic}_{12})$, and by \cite[Theorem 2.1]{domokos-schefler:16}
we have $\sepbeta^\field(\mathrm{Dic}_{12})=8$. By Lemma~\ref{lemma:multfree} 
it remains to prove that $\sepbeta(G,V)\le 8$, where    
$V$ is the $\field G$-module $V:=W_1\oplus W_2\oplus W_3\oplus W_4\oplus U$. 
Let 
$v=(w_1,w_2,w_3,w_4,u)\in V$, 
$v'=(w'_1,w'_2,w'_3,w'_4,u')\in V$, 
and assume that 
\begin{equation} \label{eq:Dec12xC2,f(v)=f(v')}
f(v)=f(v') \text{ for all }f\in \field[V]^G\text{ with }\deg(f)\le 8.
\end{equation}
We shall show that $v$ and $v'$ have the same $G$-orbit. 

\emph{Case 1:} There exists some $i,j\in \{1,2,3\}$, $i\neq j$, such that $w_i\neq 0$ and $w_j\neq 0$. 

If $w_1\neq 0$ then $\mathrm{Stab}_G(w_1)=\ker(\psi_1)=\langle b^2c\rangle$ has order $2$. 
If $w_2\neq 0$ then $\mathrm{Stab}_G(w_2)=\ker(\psi_2)=\langle c\rangle$ has order $2$. 
Since $\langle b^2c\rangle\cap \langle c\rangle=\{1_G\}$ and no non-zero element of $W_3$ is fixed by $b^2c$ or $c$, we have that 
$\mathrm{Stab}_G(w_i,w_j)$ is trivial.  
By Proposition~\ref{prop:Dic12xC2,V3+V4+U} we know that $G\cdot (w_i,w_j)=G\cdot (w'_i,w'_j)$, 
so replacing $v'$ by an appropriate element in its $G$-orbit, we may assume 
that $w'_i=w_i$ and $w'_j=w_j$. Take a component 
$w\in\{w_1,w_2,w_3,w_4,u_\chi\mid \chi\in \widehat G\}$ of $v$ different from $w_i$ and $w_j$, 
and let $w'$ be the corresponding component of $v'$. 
By Proposition~\ref{prop:Dic12xC2,V3+V4+U}, Proposition~\ref{prop:Dic12xC2,V1+V2+V3} and 
Proposition~\ref{prop:Dic12xC2,3  summands} we conclude that 
$G\cdot (w_i,w_j,w)=G\cdot (w_i,w_j,w')$, 
so $w'$ is in the orbit of $w$ with respect to the stabilizer of $(w_i,w_j)$, 
implying in turn that $w=w'$. Since this holds for all components $w$ and $w'$ of $v$ and $v'$, 
we have $v=v'$. 

\emph{Case 2:} At most one of $w_1,w_2,w_3$ is non-zero, so there exists an $i\in \{1,2,3\}$ such that $w_j=0$ for each $j\in \{1,2,3\}\setminus \{i\}$. 
As $\sepbeta(W_j)\le 8$ for all $j\in \{1,2,3\}\setminus \{i\}$, we conclude that also 
$w'_j=0$ for all $j\in \{1,2,3\}\setminus \{i\}$. Thus $v$ and $v'$ belong to a submodule $X$ of $V$ for which $\beta(G,X)\le 8$ by Proposition~\ref{prop:Dic12xC2,V3+V4+U}, and hence \eqref{eq:Dec12xC2,f(v)=f(v')} implies $G\cdot v=G\cdot v'$.   
\end{proof}

\subsection{The group $\mathrm{A}_4\times \mathrm{C}_2$.} 
In this subsection 
\[G=\mathrm{A}_4\times \mathrm{C}_2=\langle a,b,c\mid a^2=b^2=c^3=1,\ ab=ba,\ cac^{-1}=b,
\ cbc^{-1}=ab \rangle \times 
\langle d\mid d^2=1\rangle.\]  
It has the following two non-isomorphic irreducible $3$-dimensional representations:  
\[ \psi_1: a\mapsto 
\begin{bmatrix} 
      -1 & 0 & 0 \\ 0 & -1 & 0 \\ 0 & 0 & 1 
   \end{bmatrix}, \quad 
   b\mapsto \begin{bmatrix} 
      1 & 0 & 0 \\ 0 & -1 & 0 \\ 0 & 0 & -1 
   \end{bmatrix}, \quad 
   c\mapsto \begin{bmatrix} 
      0 & 0 & 1 \\ 1 & 0 & 0 \\ 0 & 1 & 0
   \end{bmatrix}, 
   \quad d \mapsto \begin{bmatrix} -1 & 0 & 0 \\ 0 & -1 & 0 \\ 0 & 0 & -1
   \end{bmatrix}\]
\[\psi_2: a\mapsto 
\begin{bmatrix} 
      -1 & 0 & 0 \\ 0 & -1 & 0 \\ 0 & 0 & 1 
   \end{bmatrix}, \quad 
   b\mapsto \begin{bmatrix} 
      1 & 0 & 0 \\ 0 & -1 & 0 \\ 0 & 0 & -1 
   \end{bmatrix}, \quad 
   c\mapsto \begin{bmatrix} 
      0 & 0 & 1 \\ 1 & 0 & 0 \\ 0 & 1 & 0
   \end{bmatrix}, 
   \quad d \mapsto \begin{bmatrix} 1 & 0 & 0 \\ 0 & 1 & 0 \\ 0 & 0 & 1
   \end{bmatrix}.\]
   For $j=1,2$ denote by $W_j$ the vector space $\field^3$ endowed with the  representation $\psi_j$. 
Assume in addition that $\field$ contains an element $\omega$ of multiplicative order $3$. 
Then the other irreducible representations of $G$ are $1$-dimensional, and can be labelled by 
the group $\widehat G=\{\omega, \omega^2, 1\}\times \{\pm 1\}\le \field^\times \times \field^\times$, where  $\chi=(\chi_1,\chi_2)\in \widehat G$ is identified with the representation 
\[\chi:a\mapsto 1,\ b\mapsto 1,\ c\mapsto \chi_1, \  d\mapsto \chi_2\] 
(note that $\langle a,b \rangle$ is the commutator subgroup of $G$). 
For $\chi\in \widehat G$ denote by $U_\chi$ the vector space $\field$ 
endowed with the  representation $\chi$, and set 
$U:=\bigoplus_{\chi\in \widehat G}U_\chi$.

\begin{proposition}\label{prop:A4xC2,Vi,U}
We have the equalities 
\begin{itemize} 
\item[(i)] $\beta(G,W_1)=6$. 
\item[(ii)] $\beta(G,W_2\oplus U)=6$.
\end{itemize}
\end{proposition} 

\begin{proof} (i) The algebra $\field[W_1]^{\langle a,b,d\rangle}$ is generated by the monomials 
$x_1^2$, $x_2^2$, $x_3^2$. 
The group element $c$ permutes cyclically these monomials, hence 
the algebra $\field[W_1]^G$ is generated by $x_1^2+x_2^2+x_3^2$, $x_1^2x_2^2+x_2^2x_3^2+x_1^2x_3^2$,  
$x_1^2x_2^2x_3^2$, $x_1^4x_2^2+x_2^4x_3^2+x_1^2x_3^4$, and since the elements with degree less than $6$ are symmetric, they 
can not form a generating set. Therefore we have (i). 

(ii) By a similar reasoning,  
$\field[W_2]^G$ is generated by 
$y_1^2 + y_2^2 + y_3^2$,  $y_1y_2y_3$, $y_1^2y_2^2 + y_2^2y_3^2 + y_1^2y_3^2$, 
$y_1^4y_2^2 + y_2^4y_3^2 + y_3^4y_1^2$, hence $\beta(G,W_2)= 6$. 
The factor group $G/G'=G/\langle a,b\rangle$ is cyclic of order $6$, hence its Davenport constant is $6$, implying that $\beta(G,U)=6$.  
Following the notation of Lemma \ref{lemma:V+U}, write $A_\chi$ for a set of 
homogeneous elements in $\field[W_2]^{G,\chi}$ that span a vector space direct complement in $\field[W_2]^{G,\chi}$ of $\field[W_2]^{G,\chi}\cap \mathcal{H}(G,W_2)$. 
By Lemma \ref{lemma:V+U} it is 
sufficient to show that (for some choice of the sets $A_\chi$) the polynomials in the following set have degree at most $6$:
\begin{center}
    $C:=\{ht_{\chi^{(1)}}\cdots t_{\chi^{(k)}}\mid \chi^{(1)},\dots \chi^{(k)}$ 
 is a product-one free sequence over $\widehat G$, \\
$h\in A_\chi$ where $\chi^{-1}=\chi^{(1)}\cdots \chi^{(k)}\}.$
\end{center} 
As $d\in \ker(\psi_2)$, the space $\field[W_2]^{G,\chi}$ is non-zero only if $\chi_2=1$.  
Moreover, $a,b\in \ker(\chi)$ for any $\chi\in \widehat G$, hence 
$\field[W_2]^{G,\chi}$ is contained in $\field[W_2]^{\langle a,b\rangle}$. 
It is easy to see that $\field[W_2]^{\langle a,b\rangle}$ is generated by $y_1^2,y_2^2,y_3^2,y_1y_2y_3$, and 
 $\field[W_2]^G=\field[y_1^2,y_2^2,y_3^2,y_1y_2y_3]^{\langle c\rangle}$ is generated by $y_1^2+y_2^2+y_3^2$, $y_1^2y_2^2+y_2^2y_3^2+y_1^2y_3^2$,  $y_1^4y_2^2+y_2^4y_3^2+y_1^2y_3^4$, $y_1y_2y_3$. 
Set
 \[s_{\omega}:=y_1^2+\omega y_2^2+\omega^2 y_3^2,\quad s_{\omega^2}:=y_1^2+\omega^2 y_2^2+\omega y_3^2.\] 
 We may take $A_{(\omega,1)}:=\{s_{\omega}, s^2_{\omega^2}\}$ 
 and $A_{(\omega^2,1)}:=\{s_{\omega^2}, s^2_{\omega}\}$   
 (this can be verified by computing a Gr\"obner basis in $\mathcal{H}(G,W_2)$). 
 Now consider an element $f=ht_{\chi^{(1)}}\cdots t_{\chi^{(k)}}\in C$. 
Here $\deg(h)=2$, or $\deg(h)=4$ and $h$ is the product of two $\langle a,b,d\rangle$-invariants.  After a possible renumbering we may assume that 
$\chi^{(i)}_2=-1$ for each $i=1,...,\ell$ and $\chi^{(j)}_2=1$ for each $j=\ell+1,...,k$. Since $h$ is a $\langle d\rangle$-invariant, $\ell=2\ell'$ must be even. We have 
$f=h\prod_{i=1}^{\ell'}(t_{\chi^{(i)}}t_{\chi^{(i+\ell')}})\prod_{j=\ell+1}^{k}t_{\chi^{(j)}}$. 
So the $G$-invariant $f$ is written as a product of relative invariants of weight 
$\chi\in \langle (\omega,1)\rangle\cong\mathrm{C}_3$, where the number of factors is 
$k-\ell'+1$ or $k-\ell'+2$, depending on whether $\deg(h)=2$ or $\deg(h)=4$.
On the other hand, the number of factors is at most $\mathsf{D}(\mathrm{C}_3)=3$, since 
$ \chi^{(1)},\dots \chi^{(k)}$ is a product-one free sequence over $\widehat G$. 
Since $\deg(f)=\deg(h)+k$, we conclude in both cases that $\deg(f)\le 6$, and 
we are done.   
\end{proof} 

 \begin{lemma}\label{lemma:A4xC2,stabilizer} 
 For a non-zero $v\in W_1$ we have 
\begin{align*} 
|\mathrm{Stab}_G(v)|=\begin{cases} 3 &\text{ if } x_1^2(v)=x_2^2(v)=x_3^2(v)  \\ 
2 &\text{ if }x_j(v)=0 \text{ for a unique }j\in \{1,2,3\} \\
4 &\text{ if } x_j(v)\neq 0 \text{ for a unique }j\in \{1,2,3\} \\ 
1 &\text{ otherwise.}\end{cases}.   
\end{align*} 
 \end{lemma} 
\begin{proof} 
One can easily find the non-zero elements of $W_1$ that occur as an eigenvector with eigenvalue $1$ for some non-identity element of $G$. 
\end{proof} 

\begin{proposition}\label{prop:A4xC2,V1+V2}
We have the inequality $\sepbeta(G,W_1\oplus W_2)\le 6$. 
\end{proposition} 

\begin{proof} 
Take $v=(w_1,w_2), \ v'=(w'_1,w'_2)\in  W_1\oplus W_2$ such that $f(v)=f(v')$ for all 
$f\in \field[W_1\oplus W_2]^G$ with $\deg(f)\le 6$. We need to show that $G\cdot v=G\cdot v'$. By Proposition~\ref{prop:A4xC2,Vi,U}, 
$w_1$ and $w'_1$ have the same $G$-orbit, so we assume that $w_1=w'_1$ 
(i.e. $x_1(v)=x_1(v')$, $x_2(v)=x_2(v')$, $x_3(v)=x_3(v')$). 
Moreover, we may assume that $G\cdot w_2=G\cdot w'_2$, and it is sufficient to 
deal with the case when both $w_1$ and $w_2$ are non-zero. 
By Lemma~\ref{lemma:A4xC2,stabilizer}, the stabilizer of $w_1$ in $G$ 
is non-trivial if and only if 
$(x_1x_2x_3)(v)=0$ or $x_1^2(v)=x_2^2(v)=x_3^2(v)$. 
This dictates the distinction of several cases below. 

\emph{Case I:} $(x_1x_2x_3)(v)\neq 0$ and $x_i^2(v)\neq x_j^2(v)$ if $i\neq j$. Consider the $G$-invariants 
\[f_1:=y_1^2+y_2^2+y_3^2,\qquad f_2:=x_1^2y_1^2+x_2^2y_2^2+x_3^2y_3^2, 
\qquad f_3:=x_1^4y_1^2+x_2^4y_2^2+x_3^4y_3^2.\] 
For $w\in V$ set 
\[M(w):=\begin{bmatrix} 1& 1& 1\\ x_1^2(w)& x_2^2(w)& x_3^2(w) \\ x_1^4(w)& x_2^4(w)& x_3^4(w)
\end{bmatrix}.\] 
We have $\det M(w)=(x_1^2(w)-x_2^2(w))(x_1^2(w)-x_3^2(w))(x_2^2(w)-x_3^2(w))$, 
so $\det M(v)\neq 0$ by assumption, and as $M(v)$ depends only on  $w_1=w'_1$, we have 
$M(v)=M(v')$. Using that $f_i(v)=f_i(v')$ for $i=1,2,3$ (since $\deg(f_i)\le 6$), we conclude 
\[\begin{bmatrix} y_1^2(v) \\ y_2^2(v)\\ y_3^2(v)\end{bmatrix} 
=M(v)^{-1}\cdot \begin{bmatrix} f_1(v)\\ f_2(v)\\ f_3(v)\end{bmatrix} 
=M(v')^{-1}\cdot \begin{bmatrix} f_1(v')\\ f_2(v')\\ f_3(v')\end{bmatrix} 
=\begin{bmatrix} y_1^2(v') \\ y_2^2(v')\\ y_3^2(v')\end{bmatrix}.\]
Thus $y_i(v)=\pm y_i(v')$ for $i=1,2,3$. Taking into account that $G\cdot w_2=G\cdot w'_2$, 
one can easily see that either $y_i(v)=y_i(v')$ for $i=1,2,3$, so $v=v'$, and we are done, 
or for some $j\in \{1,2,3\}$ we have that $y_j(v)=y_j(v')$  and $y_i(v)=-y_i(v')$ for all 
$i\in \{1,2,3\}\setminus \{j\}$. 
By symmetry we may assume that 
\begin{equation}\label{eq:y3(v)=y3(v')}
y_3(v)=y_3(v'), \qquad y_1(v)=-y_1(v'), \qquad y_2(v)=-y_2(v').
\end{equation}  
Consider the  $G$-invariants 
\[f_4:=x_2x_3y_1+x_1x_3y_2+x_1x_2y_3, \qquad f_5:=x_1x_2x_3(x_1y_1+x_2y_2+x_2y_3).\] 
Then $f_4$ and $f_5$ have degree less than $6$, hence $f_4(v)=f_4(v')$ and 
$f_5(v)=f_5(v')$, implying by $y_3(v)=y_3(v')$ (see \eqref{eq:y3(v)=y3(v')}) that 
\[\begin{bmatrix} (x_2x_3)(v) & (x_1x_3)(v) \\ x_1(v) & x_2(v)\end{bmatrix} 
\cdot \begin{bmatrix} y_1(v) \\ y_2(v)\end{bmatrix} = 
\begin{bmatrix} (x_2x_3)(v) & (x_1x_3)(v) \\ x_1(v) & x_2(v)\end{bmatrix} 
\cdot \begin{bmatrix} y_1(v') \\ y_2(v')\end{bmatrix}.\] 
 The assumptions for Case I guarantee that the determinant of the $2\times 2$ coefficient matrix above is non-zero, hence $\begin{bmatrix} y_1(v) \\ y_2(v)\end{bmatrix}=\begin{bmatrix} y_1(v') \\ y_2(v')\end{bmatrix}$, 
 and so $v=v'$. 
 
\emph{Case II:} $(x_1x_2x_3)(v)\neq 0$ and $|\{x_1^2(v),x_2^2(v),x_3^2(v)\}|=2$, say 
$x_1^2(v)=x_2^2(v)\neq x_3^2(v)$. Similarly to Case I, by basic linear algebra we conclude from $f_1(v)=f_1(v')$ and $f_2(v)=f_2(v')$ that 
\begin{equation}\label{eq:y3^2(v)=y3^2(v')}
(y_1^2+y_2^2)(v)=(y_1^2+y_2^2)(v') \text{ and } 
y_3^2(v)=y_3^2(v').\end{equation}  
Consider the $G$-invariant 
\[f_6:=x_3^2y_1^2+x_1^2y_2^2+x_2^2y_3^2.\] 
Now \eqref{eq:y3^2(v)=y3^2(v')}, $f_6(v)=f_6(v')$ imply that 
$y_1^2(v)=y_1^2(v')$ and $y_2^2(v)=y_2^2(v')$. 
So we have that 
\[y_1(v)=\pm y_1(v'),\quad y_2(v)=\pm y_2(v'),\quad y_3(v)=\pm y_3(v').\]
Taking into account that $G\cdot w_2=G\cdot w'_2$ we conclude 
that either $w_2=w'_2$, and we are done, or 
for some $j\in \{1,2,3\}$ we have that $y_j(v)=y_j(v')$  and $y_i(v)=-y_i(v')$ for all 
$i\in \{1,2,3\}\setminus \{j\}$. 
So we have to deal with the cases 
II.a, II.b, II.c below: 

\emph{Case II.a:} $y_1(v)=-y_1(v')$, $y_2(v)=-y_2(v')$, $y_3(v)=y_3(v')$. 
Consider the invariant 
\[f_7:=x_2x_3^3y_1+x_1^3x_3y_2+x_1x_2^3y_3.\] 
It has degree less than $6$, and from $f_4(v)=f_4(v')$, $f_7(v)=f_7(v')$ we conclude 
\[\begin{bmatrix} (x_2x_3)(v) & (x_1x_3)(v)\\ (x_2x_3^3)(v) & (x_1^3x_3)(v)\end{bmatrix} 
\cdot \begin{bmatrix} y_1(v) \\ y_2(v)\end{bmatrix}=
\begin{bmatrix} 0 \\ 0 \end{bmatrix}.\]
The determinant of the $2\times 2$ matrix above is 
$x_1(v)x_2(v)x_3(v)^2(x_1(v)^2-x_3(v)^2)$, which is non-zero by the assumptions for Case II. 
Consequently, $y_1(v)=y_2(v)=0$, implying in turn that $v=v'$. 

\emph{Case II.b:} $y_1(v)=-y_1(v')$, $y_2(v)=y_2(v')$, $y_3(v)=-y_3(v')$. 
Similar to Case II.a, using invariants $f_4$ and $f_7$. 

\emph{Case II.c:} $y_1(v)=y_1(v')$, $y_2(v)=-y_2(v')$, $y_3(v)=-y_3(v')$. 
Similar to Case II.a, but instead of $f_7$ we have to use the invariant 
$x_2^3x_3y_1+x_1x_3^3y_2+x_1^3x_2y_3$. 

\emph{Case III.a:} Two of $x_1(v)$, $x_2(v)$, $x_3(v)$ are zero, say 
$x_1(v)=x_2(v)=0$, $x_3(v)\neq 0$. So $\mathrm{Stab}_G(w_1)$ has order $4$ by 
Lemma~\ref{lemma:A4xC2,stabilizer}, in fact $\mathrm{Stab}_G(w_1)=\langle a,bd\rangle$. 
Then $f_2(v)=f_2(v')$ implies $y_3(v)^2=y_3(v')^2$, 
$f_6(v)=f_6(v')$ implies $y_1(v)^2=y_1(v')^2$, and then $f_1(v)=f_1(v')$ implies $y_2(v)^2=y_2(v')^2$. Taking into account that $G\cdot w_2=G\cdot w'_2$ we conclude 
that either $w_2=w'_2$, and we are done, or 
for some $j\in \{1,2,3\}$ we have that $y_j(v)=y_j(v')$  and $y_i(v)=-y_i(v')$ for all 
$i\in \{1,2,3\}\setminus \{j\}$. 
In the latter case we have $w_2$ and $w'_2$ belong to the same orbit under 
$\langle a,bd\rangle=\mathrm{Stab}_G(w_1)$, hence $G\cdot v=G\cdot v'$. 

\emph{Case III.b:} $x_i(v)=0$ for a unique $i\in \{1,2,3\}$, say $x_3(v)=0$. Then $f_4(v)=f_4(v')$ implies $y_3(v)=y_3(v')$. The $G$-invariant 
\[f_8:=x_1^2x_3^2y_1^2+x_1^2x_2^2y_2^2+x_2^2x_3^2y_3^2\] 
shows that $y_2^2(v)=y_2^2(v')$, hence by $f_1(v)=f_1(v')$ we get 
$y_1^2(v)=y_1^2(v')$.  Recall that $G\cdot w_2=G\cdot w'_2$, hence 
either $v=v'$ and we are done, or $y_1(v)=-y_1(v')$ and $y_2(v)=-y_2(v')$.  
Then $ad\cdot v=v'$, so $v$ and $v'$ have the same $G$-orbit.

\emph{Case IV:} $x_1^2(v)=x_2^2(v)=x_3^2(v)$. Applying an element of $G$ and a 
rescaling on $W_1$ (see \eqref{eq:rescaling}) we may assume that $w_1=[1,1,1]^T$.
We have the $G$-invariants 
\[f_9:=x_1x_2y_1y_2+x_1x_3y_1y_3+x_2x_3y_2y_3,\qquad 
f_{10}:=y_1y_2y_3.\] 
The multiset $\{f_4(v)=f_4(v'),f_9(v)=f_9(v'),f_{10}(v)=f_{10}(v')\}$ consists of the elementary symmetric polynomials of $y_1(v),y_2(v),y_3(v)$ (respectively of  
$y_1(v'),y_2(v'),y_3(v')$). Thus $w'_2$ is obtained from $w_2$ by permuting its coordinates. 
In fact $w_2$ can be taken to $w'_2$ by an even permutation of the coordinates, 
because $f_{11}(v)=f_{11}(v')$, where $f_{11}$ is the $G$-invariant 
\[f_{11}:=x_2x_3y_1y_2^2+x_1x_2y_1^2y_3+x_1x_3y_2y_3^2.\] 
It means that $w'_2$ belongs to the $\langle c\rangle$-orbit of $w_2$. 
Since $\mathrm{Stab}_G(w_1)=\langle c\rangle$, we conclude that $G\cdot v=G\cdot v'$. 
\end{proof} 

\begin{lemma}\label{lemma:A4xC2,stab(v1)trivial}
Let $w_1$ be a non-zero element in $W_1$, and $w_2\in W_2$. 
\begin{itemize} 
\item[(i)] If $|\mathrm{Stab}_G(w_1)|\neq 3$, then 
for any $\chi\in \{(\omega,1),(\omega^2,1),(1,1)\}$ 
there exists a homogeneous $f\in \field[W_1]^{G,\chi}$ 
with $\deg(f)\le 4$ such that $f(w_1)\neq 0$. 
\item[(ii)] If $|\mathrm{Stab}_G(w_1)|\notin \{2,4\}$, 
then for any $\chi\in \{(1,-1),(1,1)\}$ there exists a homogeneous 
$f\in \field[W_1]^{G,\chi}$ with $\deg(f)\le 3$ 
and $f(w_1)\neq 0$. 
\item[(iii)] If there exist $i,j\in\{1,2,3\}$ with 
$y_i^2(w_2)\neq y_j^2(w_2)$, then 
for any $\chi\in \{(\omega,1),(\omega^2,1),(1,1)\}$ 
there exists a homogeneous $f\in \field[W_2]^{G,\chi}$ 
with $\deg(f)\le 4$ such that $f(w_2)\neq 0$. 
\end{itemize}
\end{lemma} 

\begin{proof} 
Set 
\begin{align*} 
r_{(\omega,1)}:=x_1^2+\omega x_2^2+\omega^2 x_3^2 
&\qquad \qquad s_{(\omega,1)}:=y_1^2+\omega y_2^2+\omega^2 y_3^2 
\\ 
r_{(\omega^2,1)}:=x_1^2+\omega^2 x_2^2+\omega x_3^2
&\qquad \qquad s_{(\omega^2,1)}:=y_1^2+\omega^2 y_2^2+\omega y_3^2 
\\ 
r_{(1,-1)}&:=x_1x_2x_3. 
\end{align*} 
Then $r_\chi$ and $s_\chi$ are relative $G$-invariants of weight $\chi$. 

(i) Take first $\chi=(\omega,1)$. 
Then $r_{(\omega,1)}$, 
$(r_{(\omega^2,1)})^2$ are both homogeneous 
relative invariants of degree at most $4$ and weight $\chi$,  
and their common zero locus in 
$W_1$ is $\{w\in W_1\mid x_1^2(w)=x_2^2(w)=x_3^2(w)\}$. 
Lemma~\ref{lemma:A4xC2,stabilizer} implies that $w_1$ 
does not belong to this common zero locus by 
the assumption on its  stabilizer. 
The proof for $\chi=(\omega^2,1)$ is similar, one uses  
the relative invariants 
$r_{(\omega^2,1)}$, 
$(r_{(\omega,1)})^2$. For $\chi=(1,1)$ we can take $f=1$. 

(ii) For $\chi=(1,-1)$ we can take $f=r_{(1,-1)}$, because 
Lemma~\ref{lemma:A4xC2,stabilizer} implies 
$(x_1x_2x_3)(v_1)\neq 0$ by the assumption on the stabilizer 
of $w_1$. For $\chi=(1,1)$ we can take $f=1$. 

(iii) Similarly to the proof of (i), 
the common zero locus of the weight $(\omega,1)$ 
relative invariants 
$s_{(\omega,1)}$, 
$(s_{(\omega^2,1)})^2$ is $\{w\in W_2\mid y_1^2(w)=y_2^2(w)=y_3^2(w)\}$. 
Now $w_2$ does not belong to this common zero locus by assumption. 
The case of $\chi=(\omega^2,1)$ is settled 
using the relative invariants 
$s_{(\omega^2,1)}$, 
$(s_{(\omega,1)})^2$. 
\end{proof}

\begin{theorem}\label{thm:sepbeta(A4xC2)}
Assume that $\field$ contains an element 
of multiplicative order $6$. Then we have the equality $\sepbeta^\field(\mathrm{A}_4\times \mathrm{C}_2)=6$. 
\end{theorem}

\begin{proof} Since $\mathrm{A}_4$ is a homomorphic image of $G$, 
we have $\sepbeta^\field(G)\ge \sepbeta^\field(\mathrm{A}_4)$, 
and by \cite[Proposition 5.1]{domokos-schefler:16} 
we have $\sepbeta^\field(\mathrm{A}_4)\ge 6$. 
By Lemma~\ref{lemma:spanning invariants} (iii), 
to prove the reverse inequality 
$\sepbeta^\field(G)\le 6$ it is sufficient to show that $\sepbeta^L(G)\le 6$ for 
some extension $L$ of $\field$. Therefore we may assume that $\field$ is large enough 
to apply  Lemma~\ref{lemma:multfree}, and so it is sufficient to prove 
$\sepbeta(G,V)\le 6$ for $V:=W_1\oplus W_2\oplus U$. 
Take $v=(w_1,w_2,u)\in V$, $v'=(w'_1,w'_2,u')\in V$, 
and assume that $f(v)=f(v')$ for all 
$f\in \field[V]^G_{\le 6}$. We need to show that 
then $G\cdot v=G\cdot v'$. 
By Proposition~\ref{prop:A4xC2,V1+V2} 
we know that 
$G\cdot (w_1,w_2)=G\cdot (w'_1,w'_2)$, 
so replacing $v'$ by an appropriate 
element in its $G$-orbit 
we may assume that $w'_1=w_1$, $w'_2=w_2$. 
By Proposition~\ref{prop:A4xC2,Vi,U} (ii) 
it is sufficient to deal with the case 
$w_1\neq 0$. The Davenport constant of 
$G/G'$ is $6$, therefore $G\cdot u=G\cdot u'$. 
By Proposition~\ref{prop:A4xC2,V1+V2} 
it is sufficient to deal with the case when 
$u\neq 0$, moreover, when $u$ has a non-zero 
component $u_\chi$ for some non-identity character 
$\chi\in \widehat G$. 

\emph{Case I:} $\mathrm{Stab}_G(w_1)$ is trivial. 
We claim that $u_{\chi}=u'_{\chi}$ for all 
$\chi\in \widehat G$, so $v=v'$. 
Since $\chi^2\in \langle (\omega,1)\rangle$ 
(respectively, $\chi^3\in \langle (1,-1)\rangle$),  
by Lemma~\ref{lemma:A4xC2,stab(v1)trivial}  
there exists a relative invariant 
$f\in \field[W_1]^{G,\chi^{-2}}$ 
(respectively, $f\in \field[W_1]^{G,\chi^{-3}}$)
such that 
$ft_\chi^2$ (respectively, $ft_\chi^3$) 
is a homogeneous $G$-invariant 
of degree at most $6$, and $f(v)=f(w_1)\neq 0$. 
From 
$(ft_\chi^2)(v)=(ft_\chi^2)(v')$ 
(respectively, $(ft_\chi^3)(v)=(ft_\chi^3)(v')$)
we conclude that both 
$t_\chi^2(u)=t_\chi^2(u')$ 
and $t_\chi^3(u)=t_\chi^3(u')$. 
It follows that 
$t_\chi(u)=t_\chi(u')$, i.e. 
$u_\chi=u'_\chi$. 

\emph{Case II.a:} $|\mathrm{Stab}_G(w_1)|=3$ (consequently, 
by Lemma~\ref{lemma:A4xC2,stabilizer} we have 
$x_1^2(v)=x_2^2(v)=x_3^2(v)\neq 0$) and 
there exist $i,j\in \{1,2,3\}$ with $y_i^2(v)\neq y_j^2(v)$. 
By Lemma~\ref{lemma:A4xC2,stab(v1)trivial} (iii) for $\chi \in \langle (\omega,1)\rangle$ 
we conclude the existence of a homogeneous 
$f\in \field[W_2]^{G,\chi}$ of degree at most $4$ with $f(w_2)\neq 0$. 
Moreover, $x_1x_2x_3\in \field[W_1]^{G,(1,-1)}$ does not vanish at $v$. 
In the same way as in Case I we conclude that for all $\chi\in \widehat G$ we have 
$t_\chi^2(v)=t_\chi^2(v')$ and $t_\chi^3(v)=t_\chi^3(v')$, 
implying in turn $t_\chi(v)=t_\chi(v')$, i.e. $u_\chi=u'_\chi$. 
This holds for all $\chi$, thus $v=v'$ in this case. 

\emph{Case II.b:} $|\mathrm{Stab}_G(w_1)|=3$ (consequently, by Lemma~\ref{lemma:A4xC2,stabilizer} we have 
$x_1^2(v)=x_2^2(v)=x_3^2(v)\neq 0$, hence in particular, $(x_1x_2x_3)(v)\neq 0$) and $\mathrm{Stab}_G(w_1)\subseteq \mathrm{Stab}_G(w_2)$. 
Write $H:=\mathrm{Stab}_G(w_1)$. Then $H\cong HG'/G'$ (since $|H|=3$ is coprime to $|G'|=4$). 
It follows that an $H$-invariant monomial $h$ in $\field[U]$ is either $G$-invariant or is 
a relative $G$-invariant with weight $(1,-1)$. Hence 
$h$ or $x_1x_2x_3h$ is a $G$-invariant, implying by $(x_1x_2x_3)(v)\neq 0$ that 
$h(u)=h(u')$ holds for all $H$-invariant monomials in $\field[U]$ of degree 
at most $3=\mathsf{D}(\mathrm{C}_3)=\mathsf{D}(H)=\sepbeta^\field(H)$. Consequently, there exists an element $g\in H$ with 
$g\cdot u=u'$. Now $g\cdot (w_1,w_2,u)=(w_1,w_2,u')$, so $G\cdot v=G\cdot v'$. 

\emph{Case II.c:} $|\mathrm{Stab}_G(w_1)|=3$ (consequently, 
by Lemma~\ref{lemma:A4xC2,stabilizer} we have 
$x_1^2(v)=x_2^2(v)=x_3^2(v)\neq 0$) and 
$y_1^2(v)=y_2^2(v)=y_3^2(v)\neq 0$. Replacing the pair
$(v,v')$ by an appropriate element in its $G$-orbit 
we may assume that $w_1$ is a non-zero scalar multiple 
of $[1,1,1]^T$. If $w_2$ is also a scalar multiple 
of $[1,1,1]^T$, then we are in Case II.b. So we may assume 
that $w_2$ is not a scalar multiple of $[1,1,1]^T$. 
Then necessarily $w_2$ is a non-zero scalar multiple 
of one of $[-1,1,1]^T$, $[1,-1,1]^T$, $[1,1,-1]^T$. 
Then the relative invariants 
\begin{align*}
x_1y_1+x_2y_2+x_3y_3 &\in \field[W_1\oplus W_2]^{G,(1,-1)}, 
\\ x_1y_1+\omega x_2y_2+\omega^2 x_3y_3 &\in \field[W_1\oplus W_2]^{G,(\omega,-1)}, 
\\ x_1y_1+\omega^2 x_2y_2+\omega x_3y_3 &\in \field[W_1\oplus W_2]^{G,(\omega^2,-1)}
\end{align*}
do not vanish at $(w_1,w_2)$. Multiplying them by $x_1x_2x_3$ we get degree $5$ relative invariants 
of weight $(1,1)$, $(\omega,1)$, $(\omega^2,1)$. 
It follows that for any $\chi\in \widehat G$ 
there exists a homogeneous $f\in \field[W_1\oplus W_2]^{G,\chi^{-1}}$ of degree 
at most $5$ with $f(w_1,w_2)\neq 0$; then $(ft_\chi)(v)=(ft_\chi)(v')$, 
implying $u_\chi=u'_\chi$. This holds for all $\chi$, so $u=u'$ and thus  $v=v'$ in this case.  

\emph{Case III:} $|\mathrm{Stab}_G(w_1)|=4$. By Lemma~\ref{lemma:A4xC2,stabilizer} 
exactly one of $x_1(v)$, $x_2(v)$, $x_3(v)$ is non-zero; by symmetry we may assume 
that $x_1(v)\neq 0$ and $x_2(v)=x_3(v)=0$. By Lemma~\ref{lemma:A4xC2,stab(v1)trivial} (i) for any $\chi\in \langle (\omega,1)\rangle$ there exists a homogeneous 
$f\in \field[W_1\oplus W_2]^{G,\chi}$ with $\deg(f)\le 4$ and $f(w_1,w_2)\neq 0$. 
If we can find a homogeneous $h\in \field[W_1\oplus W_2]^{G,(1,-1)}$ with 
$\deg(h)\le 3$ and $h(w_1,w_2)\neq 0$ then we are done as in Case I. 
This happens in the following case: 

\emph{Case III.a:} $y_1(v)\neq 0$ or $(y_2y_3)(v)\neq 0$.  

Then we can take $h:=x_1y_1+x_2y_2+x_3y_3$ or 
$h:=x_1y_2y_3+x_2y_3y_1+x_3y_1y_2$. 

Otherwise we are in the following case: 

\emph{Case III.b:} $y_1(v)=0$ and $(y_2y_3)(v)=0$. 
If $y_3(v)=0$, 
then $H:=\langle ad\rangle \subseteq 
\mathrm{Stab}_G(w_1,w_2)$. Note that $H\cong HG'/G'$. 
A monomial $m\in \field[U]^H$ is a relative $G$-invariant 
with weight in $\langle (\omega,1)\rangle$. If $\deg(m)\le 2$, then 
by Lemma~\ref{lemma:A4xC2,stab(v1)trivial} (iii) 
we have a $G$-invariant $mf$ with $\deg(mf)\le 6$ and $f(w_1,w_2)\neq 0$. 
It follows that $m(u)=m(u')$ for all $H$-invariant monomials $m\in \field[U]$ with 
$\deg(m)\le 2=\mathsf{D}(H)=\sepbeta(H)$. Consequently, $H\cdot u=H\cdot u'$, and as 
$H$ stabilizes $(w_1,w_2)$, we conclude $G\cdot v=G\cdot v'$. 
The case $y_2(v)=0$ is similar, we just need to take 
$H:=\langle abd\rangle\subseteq \mathrm{Stab}_G(w_1,w_2)$ in the above argument.  

\emph{Case IV:} $|\mathrm{Stab}_G(w_1)|=2$.  By Lemma~\ref{lemma:A4xC2,stabilizer} 
exactly one of $x_1(v)$, $x_2(v)$, $x_3(v)$ is zero; by symmetry we may assume 
that $x_1(v)\neq 0$,  $x_2(v)\neq 0$, and $x_3(v)=0$. Note that 
this implies that $H:=\mathrm{Stab}_G(w_1)=\langle ad\rangle$ is not contained in 
$G'$, hence $H\cong HG'/G'$. 
By Lemma~\ref{lemma:A4xC2,stab(v1)trivial} (i) for any $\chi\in \langle (\omega,1)\rangle$ there exists a homogeneous 
$f\in \field[W_1\oplus W_2]^{G,\chi}$ with $\deg(f)\le 4$ and $f(w_1,w_2)\neq 0$. 
If we can find a homogeneous $h\in \field[W_1\oplus W_2]^{G,(1,-1)}$ with 
$\deg(h)\le 3$ and $h(w_1,w_2)\neq 0$ then we are done as in Case I, 
whereas if $H\subseteq \mathrm{Stab}_G(w_2)$ then we can finish as in Case III.b. 

\emph{Case IV.a:} $y_1(v)=y_2(v)=0$. Then $ad\cdot w_2=w_2$, hence 
$H\subseteq \mathrm{Stab}_G(w_2)$, so 
we are done, as we pointed out above. 

\emph{Case IV.b:} exactly one of $y_1(v)$, $y_2(v)$ is zero. 
Then $h:=x_1y_1+x_2y_2+x_3y_3$ is a degree $\le 3$ relative $G$-invariant 
with weight $(1,-1)$ and not vanishing at $(w_1,w_2)$, so we are done, as we explained in the beginning of Case IV. 

\emph{Case IV.c:}   Both of $y_1(v)$, $y_2(v)$ are non-zero. We claim that 
for any non-zero weight $\chi\in \widehat G$ there exists a 
a relative invariant $f\in \field[W_1\oplus W_2]^{G,\chi^{-1}}$ with 
$\deg(f)\le 4$ such that $f(w_1,w_2)\neq 0$. 
Then $ft_\chi$ is a $G$-invariant of degree $\le 5$, so $(ft_\chi)(v)=(ft_\chi)(v')$ 
implies $u_\chi=u'_\chi$ for all non-trivial $\chi\in \widehat G$, thus $v=v'$. 
For the weights $\chi\in \langle (\omega,1)\rangle$ 
the claim follows from Lemma~\ref{lemma:A4xC2,stab(v1)trivial} (i). 
For $\chi=(1,-1)$ we may take 
$f:=x_1^2x_2y_2+x_2^2x_3y_3+x_3^2x_1y_1$. 
For $\chi=(\omega,-1)$ consider the 
relative invariants 
\begin{align*} 
r_{(\omega,-1)}^{(1)}:=x_1y_1+\omega x_2y_2+\omega^2x_3y_3 
\\ r_{(\omega,-1)}^{(2)}:=x_1y_1^3+\omega x_2y_2^3+\omega^2 x_3y_3^3 
\\ r_{(\omega,-1)}^{(3)}:=x_1^3y_1+\omega x_2^3y_2+\omega^2x_3^3y_3  
\end{align*}
in $\field[W_1\oplus W_2]^{G,\chi}$. 
One of them does not vanish at $(w_1,w_2)$.
Indeed, suppose for contradiction that 
all of them vanish at $(w_1,w_2)$. 
From $r_{(\omega,-1)}^{(1)}(w_1,w_2)=0=r_{(\omega,-1)}^{(2)}(w_1,w_2)$ 
we deduce $y_1^2(w_2)=y_2^2(w_2)$. 
Thus $[y_1(v),y_2(v)]$ is a non-zero scalar multiple of 
$[1,1]$ or $[1,-1]$. After that, from 
$r_{(\omega,-1)}^{(1)}(w_1,w_2)=0=r_{(\omega,-1)}^{(3)}(w_1,w_2)$ 
we conclude 
$x_1^2(v)=x_2^2(v)$ and hence $x_1(v)=\pm x_2(v)$,  
which clearly contradicts to the assumption that 
$r_{(\omega,-1)}^{(1)}(w_1,w_2)=0$. 
The case of the weight $\chi=(\omega^2,-1)$ can be settled 
similarly, using the relative invariants 
$x_1y_1+\omega^2 x_2y_2+\omega x_3y_3$,  
$x_1y_1^3+\omega^2 x_2y_2^3+\omega x_3y_3^3$,  
$x_1^3y_1+\omega^2 x_2^3y_2+\omega x_3^3y_3$.
\end{proof}

\section{The group $\mathrm{S}_3\times \mathrm{C}_3$} 
\label{sec:S3xC3}

In this section
\[G=\mathrm{S}_3\times \mathrm{C}_3=\langle a,b\mid a^3=b^2=1,\ ba=a^2b\rangle \times 
\langle c\mid c^3=1\rangle.\]  
Assume in addition that $\field$ contains an element $\omega$ of multiplicative order $3$, 
and consider the following pairwise non-isomorphic irreducible $2$-dimensional representations of $G$:  
\[ \psi_1: a\mapsto 
\begin{bmatrix} 
      \omega & 0 \\
      0 & \omega^2 \\
   \end{bmatrix}, \quad 
   b\mapsto \begin{bmatrix} 
      0 & 1 \\
      1 & 0 \\
   \end{bmatrix}, \quad 
   c\mapsto \begin{bmatrix} 
      1 & 0 \\
      0 & 1 \\
   \end{bmatrix}, \]
\[\psi_2: a\mapsto 
\begin{bmatrix} 
      \omega & 0 \\
      0 & \omega^2 \\
   \end{bmatrix}, \quad 
   b\mapsto \begin{bmatrix} 
      0 & 1 \\
      1 & 0 \\
   \end{bmatrix}, \quad 
   c\mapsto \begin{bmatrix} 
      \omega & 0 \\
      0 & \omega \\
   \end{bmatrix}, 
\]
\[\psi_3: a\mapsto 
\begin{bmatrix} 
      \omega & 0 \\
      0 & \omega^2 \\
   \end{bmatrix}, \quad 
   b\mapsto \begin{bmatrix} 
      0 & 1 \\
      1 & 0 \\
   \end{bmatrix}, \quad 
   c\mapsto \begin{bmatrix} 
      \omega^2 & 0 \\
      0 & \omega^2 \\
   \end{bmatrix}. 
\]
The other irreducible representations of $G$ are $1$-dimensional, and can be labelled by 
the group $\widehat G=\{\pm 1\}\times \{1,\omega,\omega^2\}$, where  $\chi=(\chi_1,\chi_2)\in \widehat G$ is identified with the representation 
\[\chi:a\mapsto 1,\ b\mapsto \chi_1,\ c\mapsto \chi_2\] 
(note that $\langle a\rangle$ is the commutator subgroup of $G$, so $a$ is in the kernel of any $1$-dimensional representation of $G$). 
For $j=1,2,3$ denote by $W_j$ the vector space $\field^2$ endowed with the  representation $\psi_j$, and for $\chi\in \widehat G$ denote by $U_\chi$ the vector space $\field$ 
endowed with the  representation $\chi$, and set 
$U:=\bigoplus_{\chi\in \widehat G}U_\chi$. 

For $\xi,\eta,\zeta\in \{x,y,z\}$ set 
\[q_{\xi\eta}:=\frac 12 (\xi_1\eta_2+\xi_2\eta_1) \quad \text{ and } \quad 
p_{\xi\eta\nu}:=\xi_1\eta_1\zeta_1+\xi_2\eta_2\zeta_2.\] 
For example, $q_{xx}=x_1x_2$, $q_{yz}=\frac 12(y_1z_2+y_2z_1)$, 
$p_{xxx}=x_1^3+x_2^3$, $p_{xxy}=x_1^2y_1+x_2^2y_2$. 
It is well known (see for example \cite[Theorem 4.1]{hunziker}) that the elements $q_{\xi\eta}$, $p_{\xi\eta\zeta}$ 
minimally generate $\field[W_1\oplus W_2\oplus W_3]^{\mathrm{S}_3}$, where $\mathrm{S}_3$ is identified 
with the subgroup $\langle a,b\rangle$ of $G$. 
Set 
\begin{align*}A_{(1,1)}&:=\{q_{xx},q_{yz},p_{xxx},p_{yyy},p_{zzz},p_{xyz}\} \\
A_{(1,\omega)}&:=\{q_{xy},q_{zz},p_{xxy},p_{xzz},p_{yyz}\} \\
A_{(1,\omega^2)}&:=\{q_{xz},q_{yy},p_{xxz},p_{xyy},p_{yzz}\}.
\end{align*} 
For $\chi\in\{(1,1),(1,\omega),(1,\omega^2)\}\subset \widehat G$ the elements of $A_\chi$ are relative $G$-invariants of weight $\chi$. 

\begin{proposition}\label{prop:S3xC3,V1+V2+V3mingen} 
A homogeneous generating system of the algebra $\field[W_1\oplus W_2\oplus W_3]^G$ is 
\begin{align}\label{eq:S3xC3,V1+V2+V3mingen} 
B:=A_{(1,1)}\cup \{f_1f_2\mid f_1\in A_{(1,\omega)},f_2\in A_{(1,\omega^2)}\}
\cup\{q_{xy}^3,q_{xy}^2q_{zz},q_{xy}q_{zz}^2,q_{zz}^3\} \\ \notag 
\cup\{q_{xz}^3,q_{xz}^2q_{yy},q_{xz}q_{yy}^2,q_{yy}^3\} 
\cup \{q_{zz}^2p_{xzz},\ q_{yy}^2p_{xyy}\}.
\end{align}
\end{proposition}  

\begin{proof}  
If $\chi^{(1)},\dots,\chi^{(k)}$ is an irreducible product-one sequence over the subgroup 
$\langle (1,\omega)\rangle$ of $\widehat G$ and 
$f_j\in A_{\chi^{(j)}}$ for $j=1,\dots k$, then 
$f_1\dots f_k$ is a $G$-invariant, and the invariants of this form 
constitute a homogeneous generating system of $\field[W_1\oplus W_2\oplus W_3]^G$ 
(see the discussion in Section~\ref{sec:Davenport} about the relevance of the 
Davenport constant for our study). 
It follows that $\field[W_1\oplus W_2\oplus W_3]^G$ is generated by 
\begin{align}\label{eq:S3xC3,V1+V2+V3gens} 
A:=A_{(1,1)}\cup\{f_1f_2\mid f_1\in A_{(1,\omega)}, f_2\in A_{(1,\omega^2)}\}
\\ \notag \cup 
\{f_1f_2f_3\mid f_1,f_2,f_3\in A_{(1,\omega)} \text{ or } f_1,f_2,f_3\in A_{(1,\omega^2)}\}.
\end{align} 
By \cite[Proposition 3.3]{cziszter-domokos-szollosi} we have the inequality 
$\beta^\field(G)=8$, so we can omit the degree $9$ elements from the generating system 
$A$. Therefore 
in order to prove that $B$ generates $\field[W_1\oplus W_2\oplus W_3]^G$, 
it is sufficient to show 
 that all the elements of 
$A\setminus B$ of degree $7$ or $8$ 
are contained in the ideal of $\field[W_1\oplus W_2\oplus W_3]$ generated by the invariants of degree at most 
$6$. All elements of degree $7$ or $8$ in $A\setminus B$ are products of the form $f_1f_2f_3$ 
where $\deg(f_1)=2$, $\deg(f_2)=3$, and both of $f_1$ and $f_2$ belong either to $A_{(1,\omega)}$ or to $A_{(1,\omega^2)}$. 
The equalities 
\begin{align*}
2q_{xy}p_{xxy}&=p_{xxx}q_{yy}+q_{xx}p_{xyy} \\
q_{xy}p_{xzz}&=q_{xx}p_{yzz}+q_{yz}p_{xxz}-p_{xyz}q_{xz} \\
q_{xy}p_{yyz}&=p_{yyy}q_{xz}-q_{yz}p_{xyy}+p_{xyz}q_{yy} \\
q_{zz}p_{xxy}&=2p_{xyz}q_{xz}-q_{xx}p_{yzz} \\
q_{zz}p_{yyz}&=2q_{yz}p_{yzz}-p_{zzz}q_{yy}
\end{align*}
show that with the only exception of 
$q_{zz}p_{xzz}$, all 
products $f_1f_2$ with $f_1,f_2\in A_{(1,\omega)}$, $\deg(f_1)=2$, $\deg(f_2)=3$ are contained in the ideal generated by the $G$-invariants of degree at most $3$. 
Furthermore, we have 
\[p_{xzz}^2=4q_{xx}q_{zz}^2-4q_{xz}^2q_{zz}+p_{zzz}p_{xxz},\]
showing that $q_{zz}p_{xzz}p_{xzz}$ is contained in the ideal generated by the 
$G$-invariants of degree at most $4$. 
Thus the elements of $A\setminus B$ of degree $7$ or $8$ that are the product of three factors from $A_{(1,\omega)}$ are contained in the subalgebra of $G$-invariants of degree at most $6$.  
Similar considerations work for products of the form $f_1f_2f_3$ with $f_i\in A_{(1,\omega^2)}$, one just needs to interchange the roles of $y$ and $z$. 
This shows that $B$ is indeed a homogeneous generating system of 
$\field[W_1\oplus W_2\oplus W_3]^G$. 
\end{proof} 

\begin{proposition}\label{prop:S3xC3,V1+V2+V3} 
We have the inequality $\sepbeta(G,W_1\oplus W_2\oplus W_3)\le 6$. 
\end{proposition} 

\begin{proof} 
Let $v=(w_1,w_2,w_3)$, $v'=(w'_1,w'_2,w'_3)\in W_1\oplus W_2\oplus W_3$, 
and suppose that 
$f(v)=f(v')$ for all $f\in \field[W_1\oplus W_2\oplus W_3]^G$ with $\deg(f)\le 6$.   
By Proposition~\ref{prop:S3xC3,V1+V2+V3mingen} it is sufficient to show that 
\begin{align}\label{eq:seven}(q_{zz}^2p_{xzz})(v)&=(q_{zz}^2p_{xzz})(v') \\ 
\label{eq:2seven} 
(q_{yy}^2p_{xyy})(v)&=(q_{yy}^2p_{xyy})(v').\end{align} 
We show \eqref{eq:seven}, the argument for \eqref{eq:2seven} is obtained by interchanging the roles of $y$ and $z$. 
By Proposition~\ref{prop:S3xC3,V1+V2+V3mingen} we know that $\field[W_3]^G$ is generated in degree $\le 6$. So by assumption $w_3$ and $w'_3$ belong to the same $G$-orbit. Replacing $w$ by an appropriate element in its orbit we may assume that $w_3=w'_3$. 
Proposition~\ref{prop:S3xC3,V1+V2+V3mingen} also shows that $\field[W_1]^G$ is generated in degree at most $3$, 
so $G\cdot w_1=G\cdot w'_1$. If $w_1=0$, then $w'_1=0$, and both sides of 
\eqref{eq:seven} are zero. From now on we assume that $w_1\neq 0$, 
and so equivalently, $w'_1\neq 0$. 
If $z_1(w_3)=0$ or $z_2(w_3)=0$, then $q_{zz}(v)=0=q_{zz}(v')$, hence 
$(q_{zz}^2p_{xzz})(v)=0=(q_{zz}^2p_{xzz})(v')$. In particular, 
\eqref{eq:seven} holds. 
From now on we assume in addition that 
\begin{equation}\label{eq:nonzero}
z_1(w_3)\neq 0 \text{ and }z_2(w_3)\neq 0.\end{equation}  
By assumption we have 
$(q_{zz}q_{xz})(v)=(q_{zz}q_{xz})(v')$,  
hence 
\[q_{xz}(v)=q_{xz}(v').\]
Again by assumption we have 
$(p_{xzz}q_{xz})(v)=(p_{xzz}q_{xz})(v')$.  
Therefore if $q_{xz}(v)\neq 0$, 
then 
we have 
\[p_{xzz}(v)=p_{xzz}(v'),\]
implying in turn the desired equality \eqref{eq:seven}. 
It remains to deal with the case 
\[q_{xz}(v)=0=q_{xz}(v').\]  
Now $q_{xz}(v)=0$ implies 
that 
\begin{equation}\label{eq:S3xC3,-z1,z2}w_1,w'_1\in \field^\times [-z_1(w_3),z_2(w_3)]^T. 
\end{equation}
Consequently, $w'_1$ is a non-zero scalar multiple of $w_1$. 
On the other hand, $w'_1\in G\cdot w_1$, thus $w_1$ 
is an eigenvector of some matrix in $\{\psi_1(g)\mid g\in G\}=\psi_1(G)$. If the corresponding 
eigenvalue is $1$, then $(w_1,w_3)=(w'_1,w'_3)$ and \eqref{eq:seven} holds.
If the eigenvalue is not $1$, then $w_1$ is the $-1$-eigenvector of some 
non-diagonal element of $\psi_1(G)$ (since the eigenvectors 
with eigenvalue different from $1$ of the diagonal elements 
in $\psi_1(G)$ are scalar multiples of $[1,0]^T$ or $[0,1]^T$, 
and $w_1$ is not of this form by \eqref{eq:S3xC3,-z1,z2} and 
\eqref{eq:nonzero}). 
After rescaling $W_1$ and $W_3$ 
(see Section~\ref{subsec:convention}, \eqref{eq:rescaling}), 
we have that 
\begin{align*} 
(w_1,w_3)=([1,-1]^T,[1,1]^T) \text{ and }
(w'_1,w'_3)=([-1,1]^T,[1,1]^T), \text{ or } 
\\ (w_1,w_3)=([\omega,-1]^T,[\omega,1]^T) \text{ and }
(w'_1,w'_3)=([-\omega,1]^T,[\omega,1]^T), \text{ or } 
\\ (w_1,w_3)=([\omega^2,-1]^T,[\omega^2,1]^T) \text{ and }
(w'_1,w'_3)=([-\omega^2,1]^T,[\omega^2,1]^T)
\end{align*}
Now one can directly check that 
$p_{xzz}(v)=0=p_{xzz}(v')$ in each of these cases 
(in fact, $(w'_1,w'_3)=g\cdot (w_1,w_3)$ for some $g\in \{b,ab,a^2b\}$ in these cases), 
and the desired 
\eqref{eq:seven} holds. 
\end{proof}

For $\xi,\eta,\zeta\in \{x,y,z\}$ set 
\[q^-_{\xi\eta}:=\frac 12 (\xi_1\eta_2-\xi_2\eta_1) \quad \text{ and } \quad 
p^-_{\xi\eta\nu}:=\xi_1\eta_1\zeta_1-\xi_2\eta_2\zeta_2,\] 
and define 
\begin{align*}A_{(-1,1)}&:=\{q^-_{yz},p^-_{xxx},p^-_{yyy},p^-_{zzz},p^-_{xyz}\} \\
A_{(-1,\omega)}&:=\{q^-_{xy},p^-_{xxy},p^-_{xzz},p^-_{yyz}\} \\
A_{(-1,\omega^2)}&:=\{q^-_{xz},p^-_{xxz},p^-_{xyy},p^-_{yzz}\}.
\end{align*} 
The elements of $A_\chi$ are relative invariants of weight $\chi$. 
(In fact  every relative $G$-invariant in $\field[W_1\oplus W_2\oplus W_3]$ is a linear combination of products of elements from $\bigcup_{\chi\in \widehat G}A_\chi$, 
but we do not need this.) 

\begin{proposition}\label{prop:S3xC3,U}
We have the equality $\beta(G,U)=5$. 
\end{proposition} 

\begin{proof} The action of $G$ on $U$ factors through the regular representation of $\mathrm{C}_3\times \mathrm{C}_3\cong G/G'$, the factor of $G$ modulo its commutator subgroup. 
Therefore 
$\beta(G,U)=\mathsf{D}(\mathrm{C}_3\times \mathrm{C}_3)=5$. 
\end{proof} 

\begin{proposition}\label{prop:S3xC3,V1+U}  
We have the inequality $\beta(W_1\oplus U)\le 6$. 
\end{proposition} 

\begin{proof} 
We have $\field[W_1]^{G,(-1,1)}=\field p^-_{xxx}\oplus (\mathcal{H}(G,W_1)\cap 
\field[W_1]^{G,(-1,1)})$, and $\field[W_1]^{G,\chi}=\{0\}$ for all $\chi\in \widehat G\setminus \{(\pm 1,1)\}$. Therefore by Lemma~\ref{lemma:V+U}, $\field[W_1\oplus U]^G$ is generated by 
$\field[W_1]^G$, $\field[U]^G$, and 
\begin{align*}
\{p^-_{xxx}t_{\chi^{(1)}}\cdots t_{\chi^{(k)}}\mid 
\chi^{(1)}, \dots, \chi^{(k)} \text{ is a product-one free sequence over }\widehat G, 
\\ \chi^{(1)}\cdots \chi^{(k)}=(-1,1)\}.
\end{align*}  
Since $\mathsf{D}(\mathrm{C}_6/\mathrm{C}_2)=\mathsf{D}(\mathrm{C}_3)=3$, for $k\ge 3$ the sequence 
$\chi^{(1)}, \dots \chi^{(k)}$ necessarily has a subsequence with product 
$(\pm 1,1)$, so the above generators of the form  
$p^-_{xxx}t_{\chi^{(1)}}\cdots t_{\chi^{(k)}}$ have degree at most $3+3$. 
We have $\field[W_1]^G=\field[q_{xx},p_{xxx}]$, and 
$\beta(G,U)=\mathsf{D}(\mathrm{C}_6)=6$. 
Thus each generator of $\field[W_1\oplus U]^G$ has degree at most $6$. 
\end{proof} 

\begin{lemma} \label{lemma:S3xC3stabilizer} 
For a non-zero $w_1\in W_1$ we have  $|\mathrm{Stab}_G(w_1)/\langle c\rangle|\in \{1,2\}$, and 
\begin{align*}|\mathrm{Stab}_G(w_1)/\langle c\rangle|=2\iff p^-_{xxx}(w_1)= 0\iff \mathrm{Stab}_G(w_1)\in 
\{\langle b,c\rangle, \langle ab,c\rangle, \langle a^2b,c\rangle\}.
\end{align*} 
For a non-zero $w_j\in W_j$ ($j\in \{2,3\}$) we have 
$|\mathrm{Stab}_G(w_j)|\in \{1,2,3\}$, 
and 
\begin{itemize} 
\item[(i)] for $j=2$:    
\begin{align*}|\mathrm{Stab}_G(w_2)|=2\iff p^-_{yyy}(w_2)=0\iff \mathrm{Stab}_G(w_2)\in 
\{\langle b\rangle, \langle ab\rangle, \langle a^2b\rangle\} \\  
|\mathrm{Stab}_G(w_2)|=3\iff q_{yy}(w_2)=0\iff 
\mathrm{Stab}_G(w_2)\in 
\{\langle ac \rangle, \langle ac^2\rangle\}.
\end{align*} 
\item[(ii)] for $j=3$:  
\begin{align*}|\mathrm{Stab}_G(w_3)|=2\iff p^-_{zzz}(w_3)= 0\iff \mathrm{Stab}_G(w_3)\in 
\{\langle b\rangle, \langle ab\rangle, \langle a^2b\rangle\} \\ 
 |\mathrm{Stab}_G(w_3)|=3\iff q_{zz}(w_3)= 0 
 \iff 
\mathrm{Stab}_G(w_3)\in 
\{\langle ac \rangle, \langle ac^2\rangle\}.
\end{align*} 
\end{itemize} 
\end{lemma} 

\begin{proof} The statement can be verified by straightforward direct computation. 
\end{proof}

\begin{lemma}\label{lemma:S3xC3,stab-inv} 
Let $W$ stand for $W_2$ or $W_3$.  
Let $w\in W$, $u,u'\in U$ such that 
$f(w,u)=f(w,u')$ holds  for all $f\in \field[W\oplus U]^G$ 
with $\deg(f)\le 6$. 
Then $u$ and $u'$ belong to the same $\mathrm{Stab}_G(w)$-orbit. 
\end{lemma} 

\begin{proof} 
Consider first the case $W=W_2$.  
Assume  that $\mathrm{Stab}_G(w)=\{1_G\}$. 
Then $q_{yy}(w)\neq 0$ and 
$p^-_{yyy}(w)\neq 0$ by Lemma~\ref{lemma:S3xC3stabilizer}. If $\chi\in\{(1,\omega),(-1,1),(1,\omega^2),(-1,\omega)\}$, then there is an $f\in \{q_{yy},p^-_{yyy},q_{yy}^2,q_{yy}p^-_{yyy}\}$ such that $ft_\chi$ is a $G$-invariant of degree at most $6$, hence $t_\chi(u)=t_\chi(u')$. If $\chi=(-1,\omega^2)$, then 
$q_{yy}t_\chi^2$ and $p^-_{yyy}t_\chi^3$ are $G$-invariants showing that $t_\chi(u)=t_\chi(u')$. 
Finally, if $\chi=(1,1)$, then $t_\chi$ is a $G$-invariant of degree $1$, so $t_\chi(u)=t_\chi(u')$. Summarizing, we have $u=u'$. 

Assume next that  $|\mathrm{Stab}_G(w)|=2$; this means that 
$H:=\mathrm{Stab}_G(w)$ is $\langle b\rangle$,  $\langle ab\rangle$, or $\langle a^2b\rangle$, and  $q_{yy}(w)\neq 0$. For any $H$-invariant monomial $T:=t_\chi\cdots t_{\chi'}$ with $\deg(T)\le 2$, either $T$, $q_{yy}T$, or $q_{yy}^2T$ is a $G$-invariant of degree at most $6$. It follows that  $T(u)=T(u')$ for any $H$-invariant 
monomial in $\field[U]$ of degree at most $2$. Therefore there exists an $h\in H$ with $h\cdot u=u'$ 
(note that $\sepbeta^\field(H)=\beta^\field(H)=2$). 

Assume next that $|\mathrm{Stab}_G(w)|=3$; this means that 
$H:=\mathrm{Stab}_G(w)$ is $\langle ac\rangle$ or $\langle ac^2\rangle$, and 
$p^-_{yyy}(w)\neq 0$. For any $H$-invariant monomial $T:=t_\chi\cdots t_{\chi'}$ with $\deg(T)\le 3$, either $T$ or $p^-_{yyy}T$ is a $G$-invariant of degree at most $6$. It follows that  $T(u)=T(u')$ for any $H$-invariant 
monomial in $\field[U]$ of degree at most $3$. Therefore there exists an $h\in H$ with $h\cdot u=u'$ 
(note that $\sepbeta^\field(H)=\beta^\field(H)=3$). 

By Lemma~\ref{lemma:S3xC3stabilizer}, the only remaining case is when $w=0$, and then the result follows from Proposition~\ref{prop:S3xC3,U}. 

The case $W=W_3$ follows by Lemma~\ref{lemma:auto} (ii). 
Indeed,  
the group $G$ has the automorphism $\alpha:a\mapsto a$, $b\mapsto b$, $c\mapsto c^{-1}$. We have $\psi_2=\psi_3\circ \alpha$, so $\alpha$ interchanges $W_2$ and $W_3$, and permutes the $1$-dimensional representations.  
\end{proof} 

\begin{lemma}\label{lemma:S3xC3,V2+V1+U} 
Let $v=(w_i,w_j)\in V=W_i\oplus W_j$, where $(i,j)\in \{(1,2),(1,3),(2,3)\}$. 
Suppose that $\mathrm{Stab}_G(w_i)$, $\mathrm{Stab}_G(w_j)$ are both nontrivial, and 
$\mathrm{Stab}_G(v)=\{1_G\}$. Then for any $\chi\in\widehat G$ there exists a 
relative invariant $f\in \field[V]^{G,\chi}$ with $\deg(f)\le 4$ and $f(v)\neq 0$. 
\end{lemma}

\begin{proof} For the trivial character $(1,1)$ we may take $f=1\in \field[V]^G$. 

\emph{Case 1:}  $V=W_2\oplus W_3$. 
Then the assumptions on the stabilizers imply by Lemma~\ref{lemma:S3xC3stabilizer} 
that the stabilizers of $w_2$ and $w_3$ both have order $2$ or $3$, moreover, 
exactly one of $q_{yy}(w_2)$ and $p^-_{yyy}(w_2)$ is zero, 
and exactly one of $q_{zz}(w_3)$ and $p^-_{zzz}(w_3)$ is zero. 

\emph{Case 1.a:} $q_{yy}(w_2)=0=q_{zz}(w_3)$. 
Then both $w_2$ and $w_3$ are non-zero scalar multiples of $[1,0]^T$ or $[0,1]^T$. 
There is no harm to assume that $w_2=[1,0]^T$ (the case $w_2\in \field^\times [0,1]^T$ is similar).  Then 
$\mathrm{Stab}_G(w_2)=\langle ac^2\rangle$, and $w_3$ can not be a scalar multiple of 
$[0,1]^T$ (since the stabilizer of $[0,1]^T\in W_3$ is also  $\langle ac^2\rangle$). 
This means that $w_3$ is a non-zero scalar multiple of $[1,0]^T$ as well, 
and so none of  $\{p_{yyz}, p_{yzz}, p^-_{yyz}, p^-_{yzz},  p^-_{zzz}\}$ vanishes on 
$v=(w_2,w_3)$. These are relative invariants of degree $3$, 
and their weights represent all the five non-trivial characters. 
 
\emph{Case 1.b:} $q_{yy}(w_2)=0=p^-_{zzz}(w_3)$.  
Then $w_2$ is a non-zero scalar multiple of $[1,0]^T$ or $[0,1]^T$, and 
$w_3$ is a non-zero scalar multiple of $[1,1]^T$, $[1,\omega]^T$, $[1,\omega^2]^T$. 
It follows that none of  $\{p_{yyz}, p_{yzz}, p^-_{yyz}, p^-_{yzz},  p^-_{yyy}\}$ 
vanishes on $v$.
These are relative invariants of degree $3$, and their weights represent all the five non-trivial characters. 

\emph{Case 1.c:} $p^-_{yyy}(w_2)=0=q_{zz}(w_3)$: same as Case 1.b, we just need to interchange the roles of $W_2$ and $W_3$. 

\emph{Case 1.d:} $p^-_{yyy}(w_2)=0=p^-_{zzz}(w_3)$. 
Then both of $w_2,w_3$ are non-zero scalar multiples of $[1,1]^T$, $[1,\omega]^T$, 
or $[1,\omega^2]^T$. Moreover, $w_2$ and $w_3$ are linearly independent (otherwise they would have the same stabilizer). Thus none of 
$\{q_{yy},q_{zz},q^-_{yz},q_{yy}q^-_{yz},q_{zz}q^-_{yz}\}$ vanishes on $v$. 
These are relative invariants of degree at most $4$, 
and their weights represent all the five non-trivial characters.  

\emph{Case 2:} $V=W_1\oplus W_2$. 
Then the assumptions on the stabilizers imply by Lemma~\ref{lemma:S3xC3stabilizer} that 
exactly one of $q_{yy}(w_2)$ and $p^-_{yyy}(w_2)$ is zero. 

\emph{Case 2.a:} $q_{yy}(w_2)=0$ and $p^-_{yyy}(w_2)\neq 0$. Then $w_2$ is a non-zero scalar multiple of 
$[1,0]^T$ or $[0,1]^T$, 
say $w_2$ is a non-zero scalar multiple of $[1,0]^T$. 
If $w_1$ is a scalar multiple of $[0,1]^T$, then none of $q_{xy}$, $q^-_{xy}$, 
$q_{xy}q^-_{xy}$, $q_{xy}^2$, $p^-_{yyy}$ vanishes on $v=(w_1,w_2)$, and these relative invariants on $V$ of degree at most $4$ exhaust all the five non-trivial weights, so we are done. 
Otherwise the first coordinate of $w_1$ is non-zero, and thus $p_{xxy}$, $p_{xyy}$, 
$p^-_{xxy}$, $p^-_{xyy}$ are relative invariants on 
$V$ of degree $3$ not vanishing on $v$ (just like $p^-_{yyy}$).  
So we have exhausted all the five non-trivial weights, and we are done.

\emph{Case 2.b:} $p^-_{yyy}(w_2)=0$ and $q_{yy}(w_2)\neq 0$. 
Then $w_2$ is a non-zero scalar multiple of 
$[1,1]^T$, $[1,\omega]^T$, or $[1,\omega^2]^T$. Moreover, $w_1$ is not a scalar multiple of $w_2$ 
(since otherwise its stabilizer would contain the stabilizer of $w_2$). 
Then 
$q_{yy}$, $q^-_{xy}$, $q_{yy}^2$, $q^-_{xy}q_{yy}$, $p^-_{xyy}$ are non-zero on 
$v=(w_1,w_2)$,  and these relative invariants on $W_1\oplus W_2$ of degree at most $4$ 
exhaust all the five non-trivial weights. 

\emph{Case 3:} $V=W_1\oplus W_3$. 
The group $G$ has the automorphism $\alpha:a\mapsto a$, $b\mapsto b$, $c\mapsto c^2$. We have $\psi_3=\psi_2\circ \alpha$, so $\alpha$ interchanges 
$W_2$ and $W_3$, and permutes the $1$-dimensional representations.  So Case 3 
follows from Case 2 by Lemma~\ref{lemma:auto} (i). 
\end{proof} 

\begin{theorem}\label{thm:sepbeta(S3xC3)} 
Assume that $\field$ has an element of multiplicative order $6$. 
Then we have $\sepbeta^\field(\mathrm{S}_3\times \mathrm{C}_3)=6$. 
\end{theorem} 

\begin{proof} 
The cyclic group $\mathrm{C}_6$ is a homomorphic image of $G$, 
therefore $\sepbeta^\field(G)\ge \sepbeta^\field(\mathrm{C}_6)=6$. 

Let us turn to the proof of the reverse inequality $\sepbeta^\field(G)\le 6$. 
By Lemma~\ref{lemma:spanning invariants} we may assume that $\field$ is large enough so that 
Lemma~\ref{lemma:multfree} applies and we can reduce to the study of multiplicity-free representations. 
Set $V:=W_1\oplus W_2\oplus W_3\oplus U$, 
take $v:=(w_1,w_2,w_3,u) \in V$ and  
$v':=(w_1',w_2',w_3',u') \in V$, and assume that $f(v)=f(v')$ for all $f\in \field[V]^G$ with 
$\deg(f)\le 6$. We need to show that $G\cdot v=G\cdot v'$. 
By Proposition~\ref{prop:S3xC3,V1+V2+V3}, $(w_1,w_2,w_3)$ and 
$(w_1',w_2',w_3')$ have the same $G$-orbit in $W_1\oplus W_2\oplus W_3$. 
So replacing $v'$ by an appropriate element in its $G$-orbit, we may assume that 
$w_1=w_1'$, $w_2=w_2'$, $w_3=w_3'$.  

If $w_2=w_3=0$, then $v$ and $v'$ are in the same $G$-orbit by 
Proposition~\ref{prop:S3xC3,V1+U}. 
Otherwise there exists a $k\in \{2,3\}$ such that $w_k\neq 0$, and thus 
$|\mathrm{Stab}_G(w_k)|\in\{1,2,3\}$.
If $\mathrm{Stab}_G(w_2)=\{1_G\}$ or $\mathrm{Stab}_G(w_3)=\{1_G\}$, then 
$u=u'$ (and hence $v=v'$) by Lemma~\ref{lemma:S3xC3,stab-inv}. 
Similarly, if $\mathrm{Stab}_G(w_l)\supseteq \mathrm{Stab}_G(w_k)$ for both 
$l\in  \{1,2,3\}\setminus \{ k\}$, then $\mathrm{Stab}_G(w_k)\subseteq 
\mathrm{Stab}_G(w_1,w_2,w_3)$. By Lemma~\ref{lemma:S3xC3,stab-inv} we have 
a $g\in \mathrm{Stab}_G(w_k)$ with $g\cdot u=u'$. Then $g\cdot v=v'$, and we are done. 

It remains to deal with the case when there exists a 
pair $(i,j)\in \{(1,2),(1,3),(2,3)\}$ such that $(w_i,w_j)$ 
satisfies the assumptions of Lemma~\ref{lemma:S3xC3,V2+V1+U}. 
Therefore for any $\chi\in \widehat G$ there exists a 
relative invariant $f\in \field[W_1\oplus W_2\oplus W_3]^{G,\chi^{-1}}$ 
with $f(w_1,w_2,w_3)\neq 0$ and $\deg(f)\le 4$. 
Then $ft_\chi\in \field[V]^G$ has degree at most $5$, thus 
$(ft_\chi)(v)=(ft_\chi)(v')$, implying in turn that 
\[t_\chi(u)=\frac{(ft)(v)}{f(w_1,w_2,w_3)}=
\frac{(ft)(v)}{f(w_1,w_2,w_3)}=t_\chi(u').\] 
This holds for all $\chi\in \widehat G$, so $u=u'$ and hence 
$v=v'$. 
\end{proof} 

\section{The group $\mathrm{C}_5\rtimes \mathrm{C}_4$}
\label{sec:C5rtimesC4}
In this section 
\[G=\mathrm{C}_5\rtimes \mathrm{C}_4=\langle a,b\mid a^5=b^4=1,\ bab^{-1}=a^2 \rangle.\]  
Assume that $\field$ has an element $\xi$ of multiplicative order $20$. 
Then $\omega:=\xi^4$ has multiplicative order $5$, and $\mathrm{i}:=\xi^5$ 
has multiplicative order $4$. 
Consider the following irreducible $4$-dimensional representation of $G$:  
\[ \psi: a\mapsto 
\begin{bmatrix} 
      \omega & 0 & 0 & 0 \\ 0 & \omega^2 & 0 & 0 
      \\ 0 & 0 & \omega^4 & 0 \\ 0 & 0 & 0 & \omega^3  
   \end{bmatrix}, \quad 
   b\mapsto \begin{bmatrix} 
     0 & 1 & 0 & 0 \\ 0 & 0 & 1 & 0 \\ 0 & 0 & 0 & 1  
     \\ 1 & 0 & 0 & 0
   \end{bmatrix}.\]
The other irreducible representations of $G$ are $1$-dimensional, and can be labelled by 
the group $\widehat G=\{\pm\mathrm{i}, \pm 1\}\le \field^\times$, where  $\chi\in \widehat G$ is identified with the representation 
\[\chi:a\mapsto 1,\ b\mapsto \chi\] 
(note that $\langle a \rangle$ is the commutator subgroup of $G$). 
Write $W$ for the vector space $\field^4$ endowed with the  representation $\psi$, 
and for $\chi\in \widehat G$ denote by $U_\chi$ the vector space $\field$ 
endowed with the  representation $\chi$, and set 
$U:=\bigoplus_{\chi\in \widehat G}U_\chi$. 

Set 
\begin{align*} 
h_0:=x_1x_3-x_2x_4,\qquad h_1:=x_1x_2^2+x_3x_4^2,\qquad 
h_2:=x_2x_3^2+x_4x_1^2, 
\\ h_3:=x_1^3x_2+x_3^3x_4,\qquad h_4:=x_2^3x_3+x_4^3x_1, 
\end{align*}
and consider the following elements in $\field[W]^G$: 
\begin{align*}
& f_1:=x_1x_3+x_2x_4,\qquad  &f_2:=x_1x_2x_3x_4, \qquad &f_3:=h_1+h_2,
\\ & f_4:=h_1h_2,  \qquad &f_5:=(h_1-h_2)h_0, \qquad  &f_6:=h_3+h_4,
\\ & f_7:=x_1^5+x_2^5+x_3^5+x_4^5,\qquad &f_8:=x_1^6x_3+x_1x_3^6+x_2^6x_4+x_2x_4^6,   
\qquad &f_9:=(h_3-h_4)h_0. &    
\end{align*}

The following result was obtained using the CoCalc platform \cite{CoCalc}: 
\begin{proposition}\label{prop:C5rtimesC4,mingen}
Suppose additionally that $\field$ has characteristic zero. 
Then the elements $f_1,\dots,f_8$ form a minimal homogeneous generating system of the $\field$-algebra 
$\field[W]^G$. 
\end{proposition}

It turns out that the elements of degree at most $6$ from the above generating system are sufficient to separate 
$G$-orbits in $W$: 

\begin{proposition}\label{prop:C5rtimesC4,V} 
If for $v,v'\in W$ we have that $f_i(v)=f_i(v')$ for all $i=1,\dots,7$, 
then $f_8(v)=f_8(v')$. 
In particular, if $\field$ has characteristic zero, 
then the $G$-invariants $f_1,\dots,f_7$ separate the $G$-orbits in $W$.     
\end{proposition}

\begin{proof}
Let $v,v'\in W$ such that $f_j(v)=f_j(v')$ for $j=1,\dots,7$. 
In order to prove $f_8(v)=f_8(v')$, 
we will prove the existence of a chain of subsets 
\[\{f_1,\dots,f_7\}=S^{(0)}\subset S^{(1)}\subset \dots\subset S^{(k)}\subset \field[V]\] 
and 
elements $v'=w^{(0)},\dots,w^{(k)}$ in the $G$-orbit of $v'$ such that 
\begin{itemize} 
\item  $f(v)=f(w^{(j)})$ for all $f\in S^{(j)}$ and all $j=0,\dots,k$;
\item  $f_8\in S^{(k)}.$  
\end{itemize} 
Then $f_8(v)=f_8(w^{(k)})$, and since $f_8$ is a $G$-invariant and $G\cdot v'=G\cdot w^{(k)}$, 
we have $f_8(w^{(k)})=f_8(v')$ 
(and hence $G\cdot v=G\cdot v'$ by Proposition~\ref{prop:C5rtimesC4,mingen} when $\mathrm{char}(\field)=0$).
We shall denote by $T^{(j)}$ the subalgebra of $\field[V]$ generated by $S^{(j)}$; 
obviously, $f(v)=f(w^{(j)})$ holds for all $f\in T^{(j)}$ as well, so any element of $T^{(j)}$ can be added to $S^{(j)}$. 

From $f_1(v)=f_1(v')$ and $f_2(v)=f_2(v')$ we deduce 
\[\{(x_1x_3)(v),(x_2x_4)(v)\}=\{(x_1x_3)(v'),(x_2x_4)(v')\}.\]
So either $(x_1x_3)(v)=(x_1x_3)(v')$ and 
$(x_2x_4)(v)=(x_2x_4)(v')$, or 
$(x_1x_3)(v)=(x_1x_3)(b\cdot v')$ 
and 
$(x_2x_4)(v)=(x_2x_4)(b\cdot v')$.
Therefore we can take $w^{(1)}=v'$ or $w^{(1)}=b\cdot v'$, 
and 
\[S^{(1)}=S^{(0)}\cup \{x_1x_3,x_2x_4\}.\] 
In particular, this implies $h_0\in T^{(1)}$. 

\emph{Case I.:} $f_2(v)\neq 0$ and $h_0(v)\neq 0$. 
From $f_3(v)=f_3(w^{(1)})$, $f_5(v)=f_5(w^{(1)})$, and $h_0(v)=h_0(w^{(1)})\neq 0$ we infer that $h_1(v)=h_1(w^{(1)})$ and $h_2(v)=h_2(w^{(1)})$. 
The product of the two summands of $h_1$ (respectively $h_2$) 
is $(x_1x_3)(x_2x_4)^2$ (respectively $(x_1x_3)^2(x_2x_4)$), 
which belongs to the subalgebra $T^{(1)}$ of $\field[V]$. 
It follows that 
\begin{align*}\{(x_1x_2^2)(v),(x_3x_4^2)(v)\}=\{(x_1x_2^2)(w^{(1)}),(x_3x_4^2)(w^{(1)})\}\qquad \text{ and }
\\ \{(x_2x_3^2)(v),(x_4x_1^2)(v)\}=\{(x_2x_3^2)(w^{(1)}),(x_4x_1^2)(w^{(1)})\}.\end{align*}
Similarly, from $f_6(v)=f_6(w^{(1)})$ and $f_9(v)=f_9(w^{(1)})$ 
(note that $f_9=f_1f_6-2f_4$, hence 
$f_9(v)=f_9(v')=f_9(w^{(1)})$) we get 
\begin{align*}\{(x_1^3x_2)(v),(x_3^3x_4)(v)\}=\{(x_1^3x_2)(w^{(1)}),(x_3^3x_4)(w^{(1)})\}
\end{align*}
(and also $\{(x_2^3x_3)(v),(x_4^3x_1)(v)\}=\{(x_2^3x_3)(w^{(1)}),(x_4^3x_1)(w^{(1)})\}$, but we do not use it below). 
Note that the elements of $S^{(1)}$ are $b^2$-invariant, whereas 
$b^2$ interchanges $x_1x_2^2$ and $x_3x_4^2$, $x_2x_3^2$ and $x_4x_1^2$, 
$x_1^3x_2$ and $x_3^3x_4$. It follows that with $w^{(2)}=w^{(1)}$ or 
$w^{(2)}=b^2\cdot w^{(1)}$ one of the following sets can be taken as 
$S^{(2)}$: 
\begin{enumerate} 
\item[(i)] $S^{(2)}=S^{(1)}\cup \{x_1x_2^2,x_3x_4^2,x_2x_3^2,x_4x_1^2\}$ 
\item[(ii)] $S^{(2)}=S^{(1)}\cup \{x_1x_2^2,x_3x_4^2,x_1^3x_2,x_3^3x_4\}$
\item[(iii)] $S^{(2)}=S^{(1)}\cup \{x_2x_3^2,x_4x_1^2,x_1^3x_2,x_3^3x_4\}$
\end{enumerate}
In case (i), we have $(x_1x_2^2)^2(x_2x_3^2)=x_2^5(x_1x_3)^2\in T^{(2)}$. 
Since $0\neq (x_1x_3)(v)=(x_1x_3)(w^{(2)})$, we conclude that 
$x_2(v)^5=x_2(w^{(2)})^5$. For an appropriate $s\in \{0,1,2,3,4\}$, we have 
$x_2(v)=x_2(a^s\cdot w^{(2)})$. Set $w^{(3)}:=a^s\cdot w^{(2)}$. 
Since the elements of $S^{(2)}$ are $\langle a\rangle$-invariant, we may take 
\[S^{(3)}:=S^{(2)}\cup \{x_2\}.\] 
Then from $(x_1x_2^2)(v)=(x_1x_2^2)(w^{(3)})$ and $x_2(v)=x_2(w^{(3)})\neq 0$ we get 
$x_1(v)=x_1(w^{(3)})$. Given that, from $(x_1x_3)(v)=(x_1x_3)(w^{(3)})$ 
and $x_1(v)=x_1(w^{(3)})\neq 0$ we infer $x_3(v)=x_3(w^{(3)})$. 
Finally, from $(x_2x_4)(v)=(x_2x_4)(w^{(3)})$ and $x_2(v)=x_2(w^{(3)})\neq 0$
we get $x_4(v)=x_4(w^{(3)})$. So keeping $w^{(4)}:=w^{(3)}$ 
we can take 
\[S^{(4)}:=S^{(3)}\cup \{x_1,x_3,x_4\}.\] 
Then $f_8\in T^{(4)}$, hence 
with $w^{(5)}=w^{(4)}$ we can take $S^{(5)}=S^{(4)}\cup\{f_8\}$, and we are done. 
The cases when $S^{(2)}$ is as in (ii) or (iii) can be dealt with similarly. 

\emph{Case II.:} $h_0(v)=0$.  
Then we have 
\begin{equation}\label{eq:C5rtimesC4,x1x3=x2x4}
    (x_1x_3)(w^{(1)})=(x_1x_3)(v)=(x_2x_4)(v)=(x_2x_4)(w^{(1)}). 
    \end{equation}
Observe that 
\begin{equation}\label{eq:f8=x1x3f7} 
f_8=(x_1x_3)(x_1^5+x_3^5)+(x_2x_4)(x_2^5+x_4^5).\end{equation}
Therefore by \eqref{eq:C5rtimesC4,x1x3=x2x4} we obtain 
\[f_8(v)=(x_1x_3)(v)f_7(v)=(x_1x_3)(w^{(1)})f_7(w^{(1)})=f_8(w^{(1)}),\]
so with $w^{(2)}=w^{(1)}$ we can take 
$S^{(2)}=S^{(1)}\cup \{f_8\}$. 

\emph{Case III.:} $h_0(v)\neq 0$ and $f_2(v)=0$. By symmetry, we may assume that $x_1(v)=0$. 
From $(x_1x_3)(v)=(x_1x_3)(w^{(1)})$ we deduce that $x_1(w^{(1)})=0$ or 
$x_3(w^{(3)})=0$. Since $b^2$ interchanges $x_1$ and $x_3$, 
and fixes the elements of $S^{(1)}$, we may 
take $w^{(2)}=w^{(1)}$ or $w^{(2)}=b^2\cdot w^{(1)}$ 
and 
\[S^{(2)}=S^{(1)}\cup \{x_1\}.\] 
Moreover, we have 
\begin{equation} \label{eq:C5rtimesC4,x1neq0}
x_1(v)=0=x_1(w^{(2)})   
\end{equation}
and hence 
\begin{equation}\label{eq:C5rtimesC4,x2x4neq0}
0\neq h_0(v)=-(x_2x_4)(v)=-(x_2x_4)(w^{(2)})=h_0(w^{(2)}).
\end{equation}
We have   
\[h_1(v)-h_2(v)=\frac{f_5(v)}{h_0(v)}=
\frac{f_5(w^{(2)})}{h_0(w^{(2)})}=h_1(w^{(2)})-h_2(w^{(2)}).\]
Together with 
$h_1(v)+h_2(v)=f_3(v)=f_3(w^{(2)})
=h_1(w^{(2)})+h_2(w^{(2)})$ 
and with \eqref{eq:C5rtimesC4,x1neq0}
this implies 
\begin{align*}
(x_3x_4^2)(v)=h_1(v)=h_1(w^{(2)})
=(x_3x_4^2)(w^{(2)}) 
\text{ and } 
\\ (x_2x_3^2)(v)=h_2(v)=h_2(w^{(2)})=(x_2x_3^2)(w^{(2)}).
\end{align*}
Therefore with $w^{(3)}=w^{(2)}$ we can take 
\[S^{(3)}=S^{(2)}\cup \{x_2x_3^2,x_3x_4^2\}.\]
The equality 
\[(x_2x_3^2)^2(x_3x_4^2)=(x_3^5)(x_2x_4)^2\in T^{(3)}\]
with $(x_2x_4)(v)=(x_2x_4)(w^{(3)})\neq 0$ (see \eqref{eq:C5rtimesC4,x2x4neq0})
imply that 
\begin{equation}\label{eq:C5rtimesC4,x3^5}
x_3^5(v)=x_3^5(w^{(3)}). 
\end{equation}
So we may take $w^{(4)}=w^{(3)}$ and 
\[S^{(4)}:=S^{(3)}\cup\{x_3^5\}.\] 
Then $f_7$, $x_1$, $x_3^5$ all belong to $S^{(4)}$, implying that 
$x_2^5+x_4^5=f_7-x_1^5-x_3^5\in T^{(4)}$. 
So we may take $w^{(5)}=w^{(4)}$ and 
\[S^{(5)}=S^{(4)}\cup \{x_2^5+x_4^5\}\] 
Thus $x_1x_3$, $x_2x_4$, $x_1^5+x_3^5$, $x_2^5+x_4^5$ all belong to $T^{(5)}$, hence equality \eqref{eq:f8=x1x3f7} shows that 
$f_8\in T^{(5)}$. Therefore $f_8(v)=f_8(w^{(5)})$, and we are done. 
\end{proof}

\begin{theorem}\label{thm:sepbeta(C5rtimesC4)}
Assume that $\field$ has characteristic $0$, and contains an element of multiplicative order $20$. 
Then we have the equality $\sepbeta^\field(\mathrm{C}_5\rtimes \mathrm{C}_4)=6$. 
\end{theorem}

\begin{proof}
The subgroup $\langle a,b^2\rangle$ of $G$ is isomorphic to 
$\mathrm{D}_{10}$, the dihedral group of order $10$, hence 
by inequality \eqref{eq:sepbeta(H)} and \cite[Theorem 2.1]{domokos-schefler:16} we have the inequality $\sepbeta^\field(G)\ge 6$. 

In view of Lemma~\ref{lemma:multfree}, to prove the reverse inequality take 
$v=(w,u),v'=(w',u')\in V=W\oplus U$ such that 
\begin{equation}\label{eq:proofC5rtimesC4assumption}
f(v)=f(v') \text{ holds for all homogeneous }f\in \field[V]^G \text{ with }
\deg(f)\le 6.
\end{equation}
We have to show that $G\cdot v=G\cdot v'$. 
By Proposition~\ref{prop:C5rtimesC4,V} 
we have $G\cdot w=G\cdot w'$, so replacing $v'$ by an appropriate element in its $G$-orbit we may assume that 
$w=w'$. Moreover, 
$G/G'=G/\langle a\rangle\cong \langle b\rangle\cong \mathrm{C}_4$, and $\mathsf{D}(\mathrm{C}_4)=4$, 
so $G\cdot u=G\cdot u'$ as well, implying that 
$u'_{\mathrm{i}}\in \{\pm u_{\mathrm{i}}, \pm\mathrm{i}u_{\mathrm{i}}\}$, 
$u'_{-1}=\pm u_{-1}$, and $u'_1=u_1$. 
So it is sufficient to deal with the case when $w\neq 0$ 
and $u_\chi\neq 0$ for some 
$\chi\in \widehat G\setminus \{1\}$. 

The eigenvalues of $\psi(g)$ for a non-identity element $g\in \langle a\rangle$ differ from $1$, therefore 
$\mathrm{Stab}_G(w)\cap \langle a\rangle=\{1_G\}$, and hence $\mathrm{Stab}_G(w)$ is isomorphic to a subgroup of $\mathrm{C}_4$. 

\emph{Case I.:} $\mathrm{Stab}_G(w)\cong \mathrm{C}_4$. 
Then $\mathrm{Stab}_G(w)G'=G$, hence from 
$G\cdot u=G\cdot u'$ it follows that there exists an element 
$h\in \mathrm{Stab}_G(w)$ with $h\cdot u=u'$. Thus we have 
$h\cdot (w,u)=(w,u')$, and we are done. 

\emph{Case II.:} $\mathrm{Stab}_G(w)\cong \mathrm{C}_2$. 
Then $\mathrm{Stab}_G(w)$ is conjugate in $G$ to $\langle b^2\rangle$, 
so $\mathrm{Stab}_G(w)=g\langle b^2 \rangle g^{-1}$ for some $g\in G$.  
Replacing $(w,u)$ and $(w,u')$ by 
$g^{-1}\cdot (w,u)$ and $g^{-1}\cdot (w,u')$ we may assume 
that $\mathrm{Stab}_G(w)=\langle b^2\rangle$. 
It follows that $w=[\nu,\eta,\nu,\eta]^T\in \field^4=V$, where 
$\nu\neq \eta$. 
Note that if $m$ is  a $b^2$-invariant monomial in $\field[U]$ 
of degree at most $2$,  
then either it is $b$-invariant, or $b\cdot m=-m$. 
So if $m$ is not $b$-invariant, then both $h_0m$ and 
$(h_1-h_2)m$ are $G$-invariants of degree at  most $5$. 
Note that $h_0(v)=\nu^2-\eta^2$ and 
$(h_1-h_2)(v)=2\nu\eta(\nu-\eta)$, so 
$\nu\neq\eta$ implies that $h_0(v)$ and $(h_1-h_2)(v)$ 
can not be simultaneously zero. We infer by 
\eqref{eq:proofC5rtimesC4assumption} that 
$m(u)=m(u')$ holds for all $\langle b^2\rangle$-invariant 
monomials in $\field[U]$ of degree at most $2$. 
As $\mathsf{D}(\mathrm{C}_2)=2$, we deduce that 
$u$ and $u'$ belong to the same orbit under $\langle b^2\rangle=\mathrm{Stab}_G(w)$, implying in turn that $v=(w,u)$ and $v'=(w,u')$ belong to the same $G$-orbit. 

\emph{Case III.:} $\mathrm{Stab}_G(w)=\{1_G\}$. 
Clearly it is sufficient to show that $u_\chi=u'_\chi$ for all $\chi\in\widehat G\setminus \{1\}=\{-1,\pm\mathrm{i}\}$. 
First we deal with $\chi=-1$. 
Using the CoCalc platform \cite{CoCalc} we obtained that 
a minimal homogeneous generating system of $\field[W\oplus U_{-1}]^G$ consists of the generators $f_1,\dots,f_8$ 
of $\field[V]^G$ from Proposition~\ref{prop:C5rtimesC4,mingen}, 
together with $t_{-1}^2$, $h_0t_{-1}$, $(h_1-h_2)t_{-1}$, 
$(h_3-h_4)t_{-1}$, $(x_1^5-x_2^5+x_3^5-x_4^5)t_{-1}$. 
So all the generators involving $t_{-1}$ have degree at most 
$6$. This implies that $G\cdot (w,u_{-1})=G\cdot (w,u'_{-1})$. 
Taking into account that the stabilizer of $w$ is trivial, this means that $u_{-1}=u'_{-1}$. 

Next we deal with $\chi=-\mathrm{i}$ and show that that 
$u_{-\mathrm{i}}=u'_{-\mathrm{i}}$ 
(the argument for $u_\mathrm{i}=u'_\mathrm{i}$ is 
obtained by obvious modification). 
Since we know already 
$G\cdot u_{-\mathrm{i}}=G\cdot u'_{-\mathrm{i}}$, i.e. 
$t_{-\mathrm{i}}^4(u)=t_{-\mathrm{i}}^4(u')$, our claim is equivalent to 
\begin{equation}\label{eq:C5rtimesC4,t or t^3}
    t_{-\mathrm{i}}(u)=t_{-\mathrm{i}}(u') 
    \text{ or }t_{-\mathrm{i}}^3(u)=t_{-\mathrm{i}}^3(u'). 
\end{equation}
The latter holds if there exists a relative invariant 
$f\in \field[W]^{G,\mathrm{i}}$ (respectively, $f\in \field[W]^{G,-\mathrm{i}}$) 
with $\deg(f)\le 5$ (respectively, $\deg(f)\le 3$) with $f(w)\neq 0$, 
because then $ft_{-\mathrm{i}}$ (respectively, $ft_{-\mathrm{i}}^3$) 
is a $G$-invariant of degree a most $6$, 
and therefore $f(w)t_{-\mathrm{i}}(u)=f(w)t_{-\mathrm{i}}(u')$ 
(respectively, $f(w)t_\mathrm{i}^3(u)=f(w)t_\mathrm{i}^3(u')$), 
so \eqref{eq:C5rtimesC4,t or t^3} holds.  
Suppose for contradiction that $w$ belongs 
to the common zero locus of 
$\field[W]^{G,\mathrm{i}}_{\le 5}$ 
and $\field[W]^{G,-\mathrm{i}}_{\le 3}$. 
Set 
\begin{align*}
k_1:=x_1x_2^2-x_3x_4^2, \qquad &k_2:=x_2x_3^2-x_4x_1^2 
\\
k_3:=x_1^3x_2-x_3^3x_4,\qquad &k_4:=x_2^3x_3-x_4^3x_1 
\\ 
k_5:=x_1^5-x_3^5, \qquad &k_6:=x_2^5-x_4^5. 
\end{align*}
Then 
\begin{align*}
    k_1-\mathrm{i}k_2,\quad k_3-\mathrm{i}k_4,\quad k_5-\mathrm{i}k_6\in 
    \field[V]^{G,\mathrm{i}}_{\le 5} 
    \\ 
    k_1+\mathrm{i}k_2\in \field[V]^{G,-\mathrm{i}}_{\le 3}.
\end{align*}
So the above four relative invariants all vanish at $w$. 
In particular, it follows that $k_1(w)=0$ and $k_2(w)=0$, i.e. 
\begin{equation}\label{eq:C5rtimesC4,x1x2^2}
(x_1x_2^2)(w)=(x_3x_4^2)(w)
\text{ and } (x_2x_3^2)(w)=(x_4x_1^2)(w).\end{equation}

\emph{Case III.a:} $x_j(w)=0$ for some $j\in \{1,2,3,4\}$; 
by symmetry we may assume that $x_1(w)=0$. 
By \eqref{eq:C5rtimesC4,x1x2^2} we conclude 
$x_3(w)=0$ or $x_4(w)=0=x_2(w)$. In the latter case 
by $(k_5-\mathrm{i}k_6)(w)=0$ we deduce $x_3(w)=0$, leading to the contradiction that $w=0$. 
So $x_1(w)=0=x_3(w)$. Then $(k_5-\mathrm{i}k_6)(w)=0$ imply 
$x_2^5(w)=x_4^5(w)$, so (as $w\neq 0$) we have that $x_4(w)=\omega^jx_2(w)$ for some 
$j\in \{0,1,2,3,4\}$. Then one can easily check that 
the stabilizer of $w$ is a conjugate of the subgroup 
$\langle b^2\rangle$ of $G$, a contradiction. 

\emph{Case III.b:} $(x_1x_2x_3x_4)(w)\neq 0$. From 
\eqref{eq:C5rtimesC4,x1x2^2}
we deduce 
\[x_3(w)=x_1(w)x_2(w)^2x_4(w)^{-2}\text{ and then }
x_4(w)=x_2(w)(x_1(w)x_2(w)^2x_4(w)^{-2})^2x_1(w)^{-2}.\]
The latter equality implies 
\[x_2(w)^5=x_4(w)^5,\] 
which together with $(k_5-\mathrm{i}k_6)(w)=0$ yields 
\[x_1(w)^5=x_3(w)^5.\]
So there exist unique $j,k\in\{0,1,2,3,4\}$ with 
\[x_3(w)=\omega^jx_1(w)\text{ and }x_4(w)=\omega^kx_2(w).\]
From \eqref{eq:C5rtimesC4,x1x2^2} it follows that 
\[(j,k)\in \{(0,0),(1,2),(2,4),(3,1),(4,3)\}.\]
If $(j,k)=(0,0)$, then $b^2\in \mathrm{Stab}_G(w)$, a contradiction. If $(j,k)=(1,2)$, then $a^4b^2\in \mathrm{Stab}_G(w)$, a contradiction. Similarly we get to a contradiction for all other possible $(j,k)$.  
\end{proof}

\section{The group $\mathrm{M}_{27}$}\label{sec:m27}
In this section
\[G=\mathrm{M}_{27}=\langle a,b \mid a^9=b^3=1,\ bab^{-1}=a^4 \rangle\cong \mathrm{C}_9\rtimes \mathrm{C}_3\] 
is the non-abelian group of order $27$ with an index $3$ cyclic subgroup. 
Assume that $\field$ contains an element $\omega$ of multiplicative order $9$, and consider the following 
two non-isomorphic irreducible $3$-dimensional representations of $G$:  
\[ \psi_1: a\mapsto 
\begin{bmatrix} 
      \omega & 0 & 0 \\ 0 & \omega^4 & 0  
      \\ 0 & 0 & \omega^7   
   \end{bmatrix}, \quad 
   b\mapsto \begin{bmatrix} 
     0 & 1 & 0 \\ 0 & 0 & 1 \\ 1 & 0 & 0   
       \end{bmatrix}.\]
 \[ \psi_2: a\mapsto 
\begin{bmatrix} 
      \omega^2 & 0 & 0 \\ 0 & \omega^8 & 0  
      \\ 0 & 0 & \omega^5   
   \end{bmatrix}, \quad 
   b\mapsto \begin{bmatrix} 
     0 & 1 & 0 \\ 0 & 0 & 1 \\ 1 & 0 & 0   
       \end{bmatrix}.\]
The commutator subgroup of $G$ is $G'=\langle a^3\rangle$, and    
$G/G'\cong \mathrm{C}_3\times \mathrm{C}_3$. So      
the remaining irreducible representations of $G$ are $1$-dimensional and can be labeled 
by  $\widehat G
=\{\varepsilon,\varepsilon^2,1\}\times \{\varepsilon,\varepsilon^2,1\}\le \field^\times \times\field^\times$, 
where $\varepsilon:=\omega^3$ has multiplicative order $3$. 
Identify $\chi=(\chi_1,\chi_2)\in \widehat G$ with the representation 
\[\chi:a\mapsto\chi_1,\qquad b\mapsto \chi_2.\]
Write $W_j$ for the vector space $\field^3$ endowed with the  representation $\psi_j$, and write $U_\chi$ for the $1$-dimensional vector space $\field$ endowed by 
the  representation $\chi$, and set 
$U:=\bigoplus_{\chi\in \widehat G}U_\chi$.  

The following result was obtained using the CoCalc platform 
\cite{CoCalc}:

\begin{proposition}\label{prop:M27,V1+V2}
Assume in addition that $\mathrm{char}(\field)=0$. Then we have $\beta(G,W_1\oplus W_2)=9$. 
\end{proposition}

\begin{lemma}\label{lemma:M27,stabilizers} 
If $\mathrm{Stab}_G(v)$ is non-trivial for some $v\in W_1\oplus W_2$, then either $v=0$ (and then $\mathrm{Stab}_G(v)=G$), or $v\neq 0$ and then 
$\mathrm{Stab}_G(v)\in \{\langle b\rangle, 
\langle a^3b\rangle, \langle a^6b\rangle\}$ (in particular, then 
$|\mathrm{Stab}_G(v)|=3$). 
Moreover, 
for a non-zero $w_1\in W_1$ and $w_2\in W_2$ we have that  
\begin{itemize}
    \item $\mathrm{Stab}_G(w_1)=\langle b\rangle 
    \iff w_1\in\field [1,1,1]^T$; 
    \item $\mathrm{Stab}_G(w_2)=\langle b\rangle 
    \iff w_2\in\field [1,1,1]^T$; 
    \item $\mathrm{Stab}_G(w_1)=\langle a^3b\rangle 
    \iff w_1\in\field [\varepsilon^2,\varepsilon,1]^T$; 
    \item $\mathrm{Stab}_G(w_2)=\langle a^3b\rangle 
    \iff w_2\in \field [\varepsilon,\varepsilon^2,1]^T$; 
    \item $\mathrm{Stab}_G(w_1)=\langle a^6b\rangle 
    \iff w_1\in\field [\varepsilon,\varepsilon^2,1]^T$;
    \item $\mathrm{Stab}_G(w_2)=\langle a^6b\rangle \iff w_2\in \field [\varepsilon^2,\varepsilon,1]^T$. 
\end{itemize}
\end{lemma}

\begin{proof}
One can check by basic linear algebra that 
$\psi_j(g)$ has the eigenvalue $1$ for some $g\in G$ and $j\in \{1,2\}$ if and only if $g$ belongs to one of the order $3$ subgroups $\langle b\rangle$, $\langle a^3b\rangle$, 
$\langle a^6b\rangle$ of $G$. By computing the corresponding eigenvectors we obtain the result. 
\end{proof}

Recall that given a set $S\subset \field[V]$ of polynomials,  
$\mathcal{V}(S)$ 
stands 
for the common zero locus in $V$ of the elements of $S$. 

\begin{lemma}\label{lemma:M27,common zero locus} 
\begin{itemize}
\item[(i)] For $\chi\in \{(\varepsilon,1),(\varepsilon^2,1)\}$ 
we have 
\[\mathcal{V}(f\in \field[W_1\oplus W_2]^{G,\chi}\mid \deg(f)\le 6)=\{0\}.\]
\item[(ii)] For $\chi\in \{(\varepsilon,\varepsilon),(\varepsilon,\varepsilon^2),(\varepsilon^2,\varepsilon),(\varepsilon^2,\varepsilon^2),(1,\varepsilon),(1,\varepsilon^2)\}$ 
we have  
\[\mathcal{V}(f\in \field[W_1\oplus W_2]^{G,\chi}\mid \deg(f)\le 9)=\{v\in W_1\oplus W_2\mid \mathrm{Stab}_G(v)\neq \{1_G\}\}.\]    \end{itemize}
\end{lemma} 

\begin{proof}
(i) For $\chi=(\varepsilon,1)$, consider the following relative  invariants in 
$\field[W_1\oplus W_2]^{G,(\varepsilon,1)}$: 
\begin{align*}
    f_{(\varepsilon,1)}^{(1)}:=x_1x_2x_3, 
    &\qquad h_{(\varepsilon,1)}^{(1)}:=(y_1y_2y_3)^2, 
    \\ 
    f_{(\varepsilon,1)}^{(2)}:=x_1^3+x_2^3+x_3^3, 
    &\qquad h_{(\varepsilon,1)}^{(2)}:=y_1^6+y_2^6+y_3^6,
    \\ 
    f_{(\varepsilon,1)}^{(3)}:=x_1^5x_3+x_2^5x_1+x_3^5x_2, 
    &\qquad h_{(\varepsilon,1)}^{(3)}:=y_1^2y_2+y_2^2y_3+y_3^2y_1. 
\end{align*}
It is easy to see that 
$0=f_{(\varepsilon,1)}^{(1)}(v)
=f_{(\varepsilon,1)}^{(2)}(v)
=f_{(\varepsilon,1)}^{(3)}(v)$ 
implies $0=x_1(v)=x_2(v)=x_3(v)$, 
and similarly 
$0=h_{(\varepsilon,1)}^{(1)}(v)
=h_{(\varepsilon,1)}^{(2)}(v)
=h_{(\varepsilon,1)}^{(3)}(v)$ 
implies $0=y_1(v)=y_2(v)=y_3(v)$. 
For $\chi=(\varepsilon^2,1)$ we just need to interchange the roles of the variable sets $\{x_1,x_2,x_3\}$ and 
$\{y_1,y_2,y_3\}$ in the relative invariants constructed above. 

(ii) Take $\chi\in \{(\varepsilon,\varepsilon),(\varepsilon,\varepsilon^2),(\varepsilon^2,\varepsilon),(\varepsilon^2,\varepsilon^2),(1,\varepsilon),(1,\varepsilon^2)\}$. None of $b,a^3b,a^6b$ belongs to $\ker(\chi)$, hence by 
Lemma~\ref{lemma:M27,stabilizers}, $\mathrm{Stab}_G(v)\neq \{1_G\}$ if and only if 
$\mathrm{Stab}_G(v)\nsubseteq \ker(\chi)$. Therefore 
the inclusion "$\supseteq$" holds by 
Lemma~\ref{lemma:common zero locus}. 
We turn to the proof of the reverse inclusion "$\subseteq$". 
So assume that all elements of $\field[W_1\oplus W_2]^{G,\chi}_{\le 9}$ vanish at $(w_1,w_2)$. 
We have to show that $\mathrm{Stab}_G(w_1,w_2)\neq\{1_G\}$. 

\emph{Case I.:} $\chi=(\varepsilon,\varepsilon)$. 
Consider the following relative invariants in 
$\field[W_1\oplus W_2]^{G,(\varepsilon,\varepsilon)}$: 
\begin{align*}
    f_{(\varepsilon,\varepsilon)}^{(1)}&:=x_1^3+\varepsilon^2 x_2^3+\varepsilon x_3^3 
    \qquad &h_{(\varepsilon,\varepsilon)}^{(1)}:=y_1^6+\varepsilon^2 y_2^6+\varepsilon y_3^6 
    \\ f_{(\varepsilon,\varepsilon)}^{(2)}&:=x_1^5x_3+\varepsilon^2 x_2^5x_1+\varepsilon x_3^5x_2 
    \qquad &h_{(\varepsilon,\varepsilon)}^{(2)}:=y_1^2y_2+\varepsilon^2 y_2^2y_3+\varepsilon y_3^2y_1 
    \\ 
    f_{(\varepsilon,\varepsilon)}^{(3)}&:=
    x_1^7x_3^2+\varepsilon^2 x_2^7x_1^2+\varepsilon x_3^7x_2^2
    \qquad &h_{(\varepsilon,\varepsilon)}^{(3)}:=
    y_1^5y_3^4+\varepsilon^2 y_2^5y_1^4+\varepsilon y_3^5y_2^4
 \\
 k_{(\varepsilon,\varepsilon)}^{(1)}&:=
x_1y_1+\varepsilon^2 x_2y_2+\varepsilon x_3y_3 \qquad 
 &h_{(\varepsilon,\varepsilon)}^{(4)}:=
 y_1y_2y_3(y_1^3+\varepsilon^2 y_2^3+\varepsilon y_3^3) 
 \\
k_{(\varepsilon,\varepsilon)}^{(2)}&:=x_1^2y_1^5+\varepsilon^2 x_2^2y_2^5+\varepsilon x_3^2y_3^5 
\qquad & 
\end{align*}
The above relative invariants have degree at most $9$, so 
each of them vanishes at $(w_1,w_2)\in W_1\oplus W_2$. 
In particular, 
$0=f_{(\varepsilon,\varepsilon)}^{(1)}(w_1)=
f_{(\varepsilon,\varepsilon)}^{(2)}(w_1)
=f_{(\varepsilon,\varepsilon)}^{(3)}(w_1)$. 
Then 
\[0=\det\begin{bmatrix} 
x_1^3 & x_2^3 & x_3^3\\ x_1^5x_3 & x_2^5x_1 
& x_3^5x_2 \\ x_1^7x_3^2 & x_2^7x_1^2 & x_3^7x_2^2
\end{bmatrix}(w_1)
=(x_1x_2x_3)^4(x_1x_2-x_3^2)(x_1x_3-x_2^2)(x_2x_3-x_1^2))(w_1).\]
If $x_j(w_1)=0$ for some $j\in\{1,2,3\}$, then 
$0=f_{(\varepsilon,\varepsilon)}^{(1)}(w_1)=
f_{(\varepsilon,\varepsilon)}^{(2)}(w_1)
=f_{(\varepsilon,\varepsilon)}^{(3)}(w_1)$ implies 
$0=x_1(w_1)=x_2(w_1)=x_3(w_1)$, and so 
$\mathrm{Stab}_G(w_1,w_2)=\mathrm{Stab}_G(w_2)$. 

Suppose next that $(x_1x_2x_3)(w_1)\neq 0$. Then   
one of $x_1x_2-x_3^2$, $x_1x_3-x_2^2$, 
$x_2x_3-x_1^2$ vanishes at $w_1$. Assume 
that 
\begin{equation}\label{eq:M27,x1x2-x3^2}
(x_1x_2-x_3^2)(w_1)=0\end{equation}  
(the other cases follow by cyclic symmetry). 
From \eqref{eq:M27,x1x2-x3^2} and $f_{(\varepsilon,\varepsilon)}^{(1)}(w_1)=0$ we deduce 
\begin{equation}\label{eq:M27,x1x2x3}
\varepsilon (x_1x_2x_3)(w_1)=-x_1^3(w_1)-\varepsilon^2 x_2^3(w_1). 
\end{equation}
From \eqref{eq:M27,x1x2-x3^2}, \eqref{eq:M27,x1x2x3} and $f_{(\varepsilon,\varepsilon)}^{(2)}(w_1)=0$ we deduce 
\begin{align*}0&=(x_1^5x_3+\varepsilon^2 x_2^5x_1+\varepsilon (x_1x_2)^2x_3x_2)(w_1)
\\ &=(x_1^5x_3+\varepsilon^2 x_2^5x_1-(x_1^3+\varepsilon^2 x_2^3)x_1x_2^2)(w_1)
\\ &=x_1^4(w_1)(x_1x_3-x_2^2)(w_1), 
\end{align*}
so we have 
\begin{equation}\label{eq:M27,x1x3-x2^2} 
(x_1x_3)(w_1)=(x_2^2)(w_1)=0. 
\end{equation}
From $0=f_{(\varepsilon,\varepsilon)}^{(3)}(w_1)$, 
\eqref{eq:M27,x1x2-x3^2} and \eqref{eq:M27,x1x3-x2^2} we deduce 
\begin{align*} 
0&=(x_1^5(x_1x_3)^2+\varepsilon^2 x_2^7x_1^2+\varepsilon  
(x_1x_2)^3x_2^2x_3)(w_1)
\\ &=(x_1^5x_2^4+\varepsilon^2 x_2^7x_1^2+\varepsilon  
x_2^2x_1^2x_2^5)(w_1)
\\ &=(x_1^2x_2^4)(w_1)(x_1^3-x_2^3)(w_1), 
\end{align*} 
so we have 
\begin{equation}\label{eq:M27,x1^3-x_2^3}
x_1^3(w_1)=x_2^3(w_1).
\end{equation} 
So $x_1(w_1)=x_2(w_1)$ or $x_1(w_1)=\varepsilon x_2(w_1)$ 
or $x_1(w_1)=\varepsilon^2 x_2(w_1)$. 
Taking into account \eqref{eq:M27,x1x2-x3^2} and 
\eqref{eq:M27,x1x3-x2^2} we conclude that 
$w_1$ is a non-zero scalar multiple of $[1,1,1]^T$,$[\varepsilon^2,\varepsilon,1]^T$, or 
$[\varepsilon, \varepsilon^2,1]^T\}$. 
By Lemma~\ref{lemma:M27,stabilizers} it follows that 
$\mathrm{Stab}_G(w_1)\in \{\langle b\rangle,\langle a^3b\rangle, \langle a^6b\rangle\}$. In particular, we are done if $w_2=0$, since then $\mathrm{Stab}_G(w_1,w_2)=\mathrm{Stab}_G(w_1)$. 

Assume next that $w_2\neq 0$. 
Let us take into account that 
$0=h_{(\varepsilon,\varepsilon)}^{(1)}(w_2)=h_{(\varepsilon,\varepsilon)}^{(2)}(w_2)=
h_{(\varepsilon,\varepsilon)}^{(3)}(w_2)=h_{(\varepsilon,\varepsilon)}^{(4)}(w_2)$. 
Since $w_2\neq 0$, the above equalities imply 
that $y_j(w_2)\neq 0$ for $j=1,2,3$. 
Therefore $0=h_{(\varepsilon,\varepsilon)}^{(4)}(w_2)$ 
implies 
\begin{equation}\label{eq:M27,y1^3+...}
(y_1^3+\varepsilon^2 y_2^3+\varepsilon y_3^3)(w_2)=0.     
\end{equation}
It follows that 
\[0=\det\begin{bmatrix}1&1&1\\y_1^3&y_2^3&y_3^3
\\y_1^6&y_2^6&y_3^6\end{bmatrix}(w_2)
=(y_2^3-y_1^3)(w_2)(y_3^3-y_1^3)(w_2)(y_3^3-y_2^3)(w_2).\]
Assume that 
\begin{equation}\label{eq:M27,y1^3=y2^3}
y_1^3(w_2)=y_2^3(w_2)     
\end{equation}
(the cases when $y_1^3(w_2)=y_3^3(w_2)$ or $y_2^3(w_2)=y_3^3(w_2)$ follow by cyclic symmetry). 
From \eqref{eq:M27,y1^3+...}, \eqref{eq:M27,y1^3=y2^3} and 
$0=h_{(\varepsilon,\varepsilon)}^{(2)}(w_2)$  
we deduce 
\[0=\det\begin{bmatrix} 1&1&1\\ y_1^2y_2&y_2^2y_3&y_3^2y_1\\ y_1^3&y_2^3&y_3^3\end{bmatrix}(w_2)=
\det\begin{bmatrix}1&0&1\\y_1^2y_2&y_2(y_2y_3-y_1^2)&y_3^2y_1\\y_1^3&0&y_3^3\end{bmatrix}(w_2),\]
hence 
\begin{equation}\label{eq:M27,y1^2=y2y3}
y_1^2(w_2)=(y_2y_3)(w_2)
\end{equation}
or 
\begin{equation}\label{eq:M27,y1^3-y3^3}
y_1^3(w_2)=y_3^3(w_2).
\end{equation}

Now \eqref{eq:M27,y1^3=y2^3}, \eqref{eq:M27,y1^2=y2y3} 
clearly imply that 
$w_2$ is a non-zero scalar multiple of $[1,1,1]^T$, $[\varepsilon,\varepsilon^2,1]^T$, or 
$[\varepsilon^2,\varepsilon,1]^T$, implying in turn by 
Lemma~\ref{lemma:M27,stabilizers} that 
$\mathrm{Stab}_G(w_2)\in \{\langle b\rangle,
\langle a^3b\rangle, \langle a^6b\rangle\}$. 

If \eqref{eq:M27,y1^3-y3^3} holds, then together with \eqref{eq:M27,y1^3=y2^3} 
it implies that $y_2(w_2)=\nu_2y_1(w_2)$ and  
$y_3(w_2)=\nu_3y_1(w_2)$, 
where $\{\nu_2,\nu_3\}\in\{1,\varepsilon,\varepsilon^2\}$. Recall that $y_1(w_2)\neq 0$. 
From $h_{(\varepsilon,\varepsilon)}^{(2)}(w_2)=0$ we 
infer $\nu_2+\varepsilon^2 \nu_2^2\nu_3+\varepsilon \nu_3^2=0$, 
and hence $\nu_2\nu_3=1$ or $\nu_2\nu_3=\varepsilon^2$. 
On the other hand, 
$h_{(\varepsilon,\varepsilon)}^{(3)}(w_2)=0$ yields 
$\nu_3+\varepsilon^2 \nu_2^2+\varepsilon \nu_3^2\nu_2=0$, 
and hence $\nu_2\nu_3=1$ or $\nu_2\nu_3=\varepsilon$. 
Thus necessarily we have $\nu_2\nu_3=1$. So  
$\nu_2=\nu_3=1$ (i.e. $w_2\in \field[1,1,1]^T$) or 
$(\nu_2,\nu_3)=(\varepsilon,\varepsilon^2)$ (i.e. $w_2\in 
\field[\varepsilon,\varepsilon^2,1]^T$)  or 
$(\nu_2,\nu_3)=(\varepsilon^2,\varepsilon)$ (i.e. $w_2\in 
\field[\varepsilon^2,\varepsilon,1]^T$). It follows by Lemma~\ref{lemma:M27,stabilizers} 
that $\mathrm{Stab}_G(w_2)\in \{\langle b\rangle, 
\langle a^3b\rangle,\langle a^6b\rangle\}$. 

Thus we showed  that both $w_1$ and $w_2$ are 
scalar multiples of $[1,1,1]^T$, $[\varepsilon^2,\varepsilon,1]^T$, 
or $[\varepsilon,\varepsilon^2,1]^T$. 
If $w_1=0$ or $w_2=0$, then we have the desired 
$\mathrm{Stab}_G(w_1,w_2)\neq \{1_G\}$ by 
Lemma~\ref{lemma:M27,stabilizers}. 
If both $w_1$ and $w_2$ are non-zero 
scalar multiples of $[1,1,1]^T$, $[\varepsilon^2,\varepsilon,1]^T$, 
or $[\varepsilon,\varepsilon^2,1]^T$, then again 
using Lemma~\ref{lemma:M27,stabilizers} one can verify by direct computation that 
for such a pair $(w_1,w_2)$, we have 
\begin{equation*}\label{eq:M27,k}
k_{(\varepsilon,\varepsilon)}^{(1)}(w_1,w_2)=0=k_{(\varepsilon,\varepsilon)}^{(2)}(w_1,w_2)
\iff \mathrm{Stab}_G(w_1)=\mathrm{Stab}_G(w_2).
\end{equation*}
Therefore $\mathrm{Stab}_G(w_1,w_2)=\mathrm{Stab}_G(w_1)\neq \{1_G\}$. 

\emph{Case II.:} $\chi=(1,\varepsilon)$. 
Consider the following relative invariants in 
$\field[W_1\oplus W_2]^{G,(1,\varepsilon)}$: 
\begin{align*} 
f_{(1,\varepsilon)}^{(1)}&:=(x_1x_2x_3)^2(x_1^3+\varepsilon^2 x_2^3+\varepsilon x_3^3) 
\qquad & h_{(1,\varepsilon)}^{(1)}:=
y_1^2y_3+\varepsilon^2 y_2^2y_1+\varepsilon y_3^2y_2
\\f_{(1,\varepsilon)}^{(2)}&:=x_1x_2x_3(x_1^6+\varepsilon^2 x_2^6+\varepsilon x_3^6) 
\qquad &h_{(1,\varepsilon)}^{(2)}:=
y_1^4y_3^2+\varepsilon^2 y_2^4y_1^2+\varepsilon y_3^4y_2^2
\\ f_{(1,\varepsilon)}^{(3)}&:=x_1x_2^2+\varepsilon^2 x_2x_3^2+\varepsilon x_3x_1^2 \qquad 
&h_{(1,\varepsilon)}^{(3)}:=y_1^9+\varepsilon^2 y_2^9+
\varepsilon y_3^9
\\ f_{(1,\varepsilon)}^{(4)}&:=x_1^9+\varepsilon^2 x_2^9+\varepsilon x_3^9 
\qquad &k_{(1,\varepsilon)}^{(1)}:=x_1y_2+\varepsilon^2 x_2y_3+\varepsilon x_3y_1 
\\ & \qquad & k_{(1,\varepsilon)}^{(2)}:=x_1^2y_2^2+\varepsilon^2 x_2^2y_3^2+\varepsilon x_3^2y_1^2.
\end{align*}
The above relative invariants have degree at most $9$,  
so each of them vanishes at $(w_1,w_2)\in W_1\oplus W_2$. 
If $x_j(w_1)=0$ for some $j\in \{1,2,3\}$, then 
$0=f_{(1,\varepsilon)}^{(3)}(w_1)=f_{(1,\varepsilon)}^{(4)}(w_1)$ implies $w_1=0$, and thus $\mathrm{Stab}_G(w_1,w_2)=
\mathrm{Stab}_G(w_2)$. 

Otherwise $(x_1x_2x_3)(w_1)\neq 0$, and 
$0=f_{(1,\varepsilon)}^{(1)}(w_1)=f_{(1,\varepsilon)}^{(2)}(w_1)$ implies 
\[0=\det\begin{bmatrix} 1&1&1 \\ x_1^3(w_1) & x_2^3(w_1) &
x_3^3(w_1) \\ x_1^6(w_1) & x_2^6(w_1) &
x_3^6(w_1)\end{bmatrix}=(x_2^3-x_1^3)(x_3^3-x_1^3)(x_3^3-x_2^3)(w_1).\]
By symmetry it is sufficient to deal with the case when 
\begin{equation*}\label{eq:M27,x1^3=x2^3,2}
x_1^3(w_1)=x_2^3(w_1). 
\end{equation*}
By $0=f_{(1,\varepsilon)}^{(1)}(w_1)$, $0=f_{(1,\varepsilon)}^{(3)}$ and 
$(x_1x_2x_3)^2(w_1)\neq 0$ we have  
\[0=\det\begin{bmatrix} 1 & 1 & 1\\
x_1^3 & x_2^3 & x_3^3 \\
x_1x_2^2 & x_2x_3^2 & x_3x_1^2\end{bmatrix}(w_1).\]
Taking into account $x_1^3(w_1)=x_2^3(w_1)$
we end up with 
\begin{equation}\label{eq:oct4} 
0=x_2(x_1x_2-x_3^2)(w_1)(x_3^3-x_1^3)(w_1).\end{equation}
If $(x_1x_2)(w_1)=x_3^2(w_1)$, 
then we have 
\[0=f_{(1,\varepsilon)}^{(3)}(w_1)=
(x_1x_2^2+\varepsilon^2 x_2(x_1x_2)
+\varepsilon x_3x_1^2)(w_1)
=\varepsilon x_1(w_1)(x_1x_3-x_2^2)(w_1).\]
Thus $(x_1x_3)(w_1)=x_2^2(w_1)$, and this together with $(x_1x_2)(w_1)=x_3^2(w_1)$ implies that $w_1$ is a non-zero 
scalar multiple of $[1,1,1]^T$, 
$[\varepsilon,\varepsilon^2,1]^T$, 
or $[\varepsilon^2,\varepsilon,1]^T$. 

If $x_3^3(w_1)=x_1^3(w_1)$ (the other alternative from \eqref{eq:oct4}, then 
(recall that $x_2^3(w_1)=x_1^3(w_1)$) 
there exist some cubic roots $\nu_2,\nu_3$ of $1$ 
such that $w_1$ is a non-zero scalar multiple of $[1,\nu_2,\nu_3]$. 
Then $0=f_{(1,\varepsilon)}^{(3)}(w_1)$ reduces 
to 
\[\nu_2^2+\varepsilon^2 \nu_2\nu_3^2
+\varepsilon \nu_3=0.\] 
The sum of three cubic roots of $1$ is zero only if the three summands are $1,\varepsilon,\varepsilon^2$ in some order. 
We conclude that $w_1$ is a non-zero scalar multiple of 
$[1,1,1]^T$, 
$[\varepsilon,\varepsilon^2,1]^T$, 
or $[\varepsilon^2,\varepsilon,1]^T$. 

In particular, if $w_2=0$, then $\mathrm{Stab}_G(w_1,w_2)=\mathrm{Stab}_G(w_1)\neq \{1_G\}$ 
by Lemma~\ref{lemma:M27,stabilizers}, and we are done. 

Assume next that $w_2\neq 0$. 
Then by $0=h_{(1,\varepsilon)}^{(1)}(w_2)=
h_{(1,\varepsilon)}^{(3)}(w_2)$ we deduce 
that none of $y_1(w_2)$, $y_2(w_2)$, $y_3(w_2)$ 
is zero. From 
$0=h_{(1,\varepsilon)}^{(1)}(w_2)=
h_{(1,\varepsilon)}^{(2)}(w_2)$ we get 
\[0=\det\begin{bmatrix} 
1 & 1 & 1 \\
y_1y_2^2 & y_2y_3^2 & y_3y_1^2 
\\ y_1^2y_2^4 & y_2^2y_3^4 & y_3^2y_1^4 
\end{bmatrix}(w_2)
=(y_1y_2y_3)(y_1y_3-y_2^2)(y_3^2-y_1y_2)(y_1^2-y_2y_3)(w_2).\]
By symmetry we may assume that 
$y_2^2(w_2)=(y_1y_3)(w_2)$. 
Then we have 
\[0=h_{(1,\varepsilon)}^{(1)}(w_2)=
(y_1^2y_3+\varepsilon^2 (y_1y_3)y_1+\varepsilon  
y_3^2y_2)(w_2)=
\varepsilon y_3(w_2)(y_2y_3-y_1^2)(w_2).\]
Thus we have 
\[\frac{y_2^2(w_2)}{y_1(w_2)}=y_3(w_2)=
\frac{y_1^2(w_2)}{y_2(w_2)},\] 
and so $y_1^3(w_2)=y_2^3(w_2)$. 
Obviously, this together with $y_2^2(w_2)=y_1(w_2)y_3(w_2)$ means that 
$w_2$ is a non-zero scalar multiple of 
$[1,1,1]^T$, 
$[\varepsilon,\varepsilon^2,1]^T$, 
or $[\varepsilon^2,\varepsilon,1]^T$. 

So both $w_1$ and $w_2$ belong to $\field[1,1,1]^T\cup 
\field[\varepsilon,\varepsilon^2,1]^T\cup \field [\varepsilon^2,\varepsilon,1]^T$.  
If $w_1=0$ or $w_2=0$, then $\mathrm{Stab}_G(w_1,w_2)\neq \{1_G\}$ by Lemma~\ref{lemma:M27,stabilizers}. 
If both $w_1$ and $w_2$ are non-zero, then again 
using Lemma~\ref{lemma:M27,stabilizers} one can verify by direct computation that 
\begin{equation*}\label{eq:M27,k,2}
k_{(1,\varepsilon)}^{(1)}(w_1,w_2)=0=k_{(1,\varepsilon)}^{(2)}(w_1,w_2)\iff \mathrm{Stab}_G(w_1)=\mathrm{Stab}_G(w_2).
\end{equation*}
Therefore $\mathrm{Stab}_G(w_1,w_2)=\mathrm{Stab}_G(w_1)\neq \{1_G\}$. 
This finishes the proof for Case II. 

The map $a\mapsto a^2$, $b\mapsto b$ extends to an automorphism 
$\alpha$ of $G$, and for $\chi=(\varepsilon,\varepsilon)$ we have 
$\chi\circ\alpha=(\varepsilon^2,\varepsilon)$. 
So by Lemma~\ref{lemma:auto} (iii), 
the conclusion of statement (ii) holds also for the weight $\chi=(\varepsilon^2,\varepsilon)$.  
The only information on $\varepsilon$ used in the constructions of relative invariants in Cases I and II was that it 
has multiplicative order $3$. Therefore we can replace it by 
the other element of multiplicative order $3$, 
namely by $\varepsilon^2$, and we 
get that (ii) holds also for the weights 
$\chi\in\{\varepsilon,\varepsilon^2),
(\varepsilon^2,\varepsilon^2),(1,\varepsilon^2)\}$. 
\end{proof}

\begin{theorem}\label{thm:sepbeta(M27)}
Assume that $\field$ has characteristic zero, and it contains an element of multiplicative order 
$9$. Then we have the equality 
$\sepbeta^\field(\mathrm{M}_{27})=10$. 
\end{theorem}

\begin{proof} 
In view of Lemma~\ref{lemma:multfree},  
take $v=(w_1,w_2,u),v'=(w'_1,w'_2,u')\in W_1\oplus W_2\oplus U$ with 
\begin{equation}\label{eq:proofM27assumption}
f(v)=f(v') \text{ for all homogeneous }f\in \field[W_1\oplus W_2\oplus U]^G \text{ where }
\deg(f)\le 10.
\end{equation}
We need to show that $G\cdot v=G\cdot v'$. 
Replacing $v'$ by an appropriate element in its $G$-orbit, we may assume by Proposition~\ref{prop:M27,V1+V2} that 
$(w_1,w_2)=(w'_1,w'_2)$. Moreover, $G\cdot u=G\cdot u'$, since 
$\mathsf{D}(G/G')=\mathsf{D}(\mathrm{C}_3\times\mathrm{\mathrm{C}_3})=5$. Therefore it is sufficient to deal with the 
case when $(w_1,w_2)\in W_1\oplus W_2$ is non-zero. 

\emph{Case I.:} $\mathrm{Stab}_G(w_1,w_2)\neq \{1_G\}$. 
Take any $\langle b\rangle$-invariant monomial $m\in \field[U]^{\langle b\rangle}$ with $\deg(m)\le 3$. It belongs to $\field[W_1\oplus W_2\oplus U]^{G,\chi}$ 
for some $\chi \in \{\varepsilon,1),(\varepsilon^2,1),(1,1)\}$. By Lemma~\ref{lemma:M27,common zero locus} (i) there exists an $f\in \field[W_1\oplus W_2\oplus U]^{G,\chi^{-1}}$ 
of degree at most $6$ with $f(w_1,w_2)\neq 0$. Note that 
$fm\in \field[W_1\oplus W_2\oplus U]^G$ has degree at most $9$. 
It follows by \eqref{eq:proofM27assumption} that 
$(fm)(v)=(fm)(v')$, implying in turn that 
$m(u)=m(u')$. This holds for all monomials 
$m\in \field[U]^{\langle b\rangle}$ with $\deg(m)\le 3$. 
Since $\mathsf{D}(\langle b\rangle)=3$, we conclude that 
$u$ and $u'$ belong to the same $\langle b\rangle$-orbit. 
Note that $\mathrm{Stab}_G(w_1,w_2)G'=\langle b\rangle G'$ by Lemma~\ref{lemma:M27,stabilizers}. So the $\langle b\rangle$-orbits of $u$ and $u'$ coincide with their 
$\mathrm{Stab}_G(w_1,w_2)$-orbits. Thus 
$\mathrm{Stab}_G(w_1,w_2)\cdot u=\mathrm{Stab}_G(w_1,w_2)\cdot u'$, and therefore $G\cdot (w_1,w_2,u)=G\cdot (w'_1,w'_2,u')$. 

\emph{Case II.:} $\mathrm{Stab}_G(w_1,w_2)=\{1_G\}$. 
We claim that $u=u'$; that is, we claim that 
$u_\chi=u'_\chi$ for all $\chi\in \widehat G$. 
This is obvious for $\chi=(1,1)\in \widehat G$, since $u$ and $u'$ 
have the same $G$-orbit. For $\chi \in  \widehat G\setminus \{(1,1)\}$, by Lemma~\ref{lemma:M27,common zero locus} there exists an $f\in \field[W_1\oplus W_2\oplus U]^{G,\chi^{-1}}$ 
with $\deg(f)\le 9$ and $f(w_1,w_2)\neq 0$. 
Then $ft_\chi$ is a $G$-invariant of degree at most $10$. 
Thus by \eqref{eq:proofM27assumption} we have 
$(ft_\chi)(w_1,w_2,u)=(ft_\chi)(w_1,w_2,u')$, implying in turn that $t_\chi(u)=t_\chi(u')$, i.e. $u_\chi=u'_\chi$. 
This finishes the proof of the inequality 
$\sepbeta^\field(G)\le 10$. 

To see the reverse inequality, consider the $\field G$-module 
$W_1\oplus U_{(1,\varepsilon^2)}$. 
Consider the vectors 
$v:=([1,0,0]^T,\varepsilon)$ and 
$v ':=([1,0,0]^T,\varepsilon^2)$. 
We have 
$(f_{(1,\varepsilon)}^{(4)}t_{(1,\varepsilon^2)})(v)=\varepsilon$, whereas  
$(f_{(1,\varepsilon)}^{(4)}t_{(1,\varepsilon^2)})(v')=\varepsilon^2$, so the invariant 
$(f_{(1,\varepsilon)}^{(4)}
t_{(1,\varepsilon^2)}$) 
separates $v$ and $v'$. We claim that no $G$-invariant of degree at most $9$ separates $v$ and $v'$. 
The elements of 
$\field[W_1]$ and $\field[U_{(1,\varepsilon^2)}]^G=
\field[t_{(1,\varepsilon)}^3]$ agree on 
$v$ and $v'$. 
Suppose for contradiction that there 
exists a multihomogeneous invariant 
$f=ht_{(1,\varepsilon^2)}$ (respectively $f=ht_{(1,\varepsilon^2)}^2$) 
of degree at most $9$ with $f(v)\neq f(v')$, where 
$h\in \field[W_1]^{G,(1,\varepsilon)}$ 
(respectively 
$h\in \field[W_1]^{G,(1,\varepsilon^2)}$). 
Then $h$ has a monomial of the form $x_1^d$ with non-zero coefficients (since $0=x_2(v)=x_3(v)=
x_2(v')=x_3(v')$). Moreover, $x_1^d$ must be an 
$\langle a\rangle$-invariant monomial. 
However, the smallest $d$ for which $x_1^d$ 
is $\langle a\rangle$-invariant is $9$. 
This is a contradiction, because the degree of 
$h$ is strictly less than $9$. 
The inequality 
$\sepbeta(G,W_1\oplus U_{(1,\varepsilon)})\ge 10$ is proved. 
\end{proof}

%\section*{Acknowledgements} 

%This research was partially supported by the Hungarian National Research, Development and Innovation Office,  NKFIH K 138828.

%%%%%%%%%%%%%%%%%%%%%%%%%%%%%%%%%%

\end{document}